\documentclass[mnsc,nonblindrev,copyedit]{informs3a}
\usepackage{amsmath,amssymb,mathtools,mathrsfs}
\usepackage[dvipsnames]{xcolor}
\usepackage{booktabs,enumerate,graphicx,tikz,placeins,natbib}
\usetikzlibrary{arrows.meta,positioning,calc}
\definecolor{DarkBlue}{rgb}{0,0.08,0.45}
\usepackage[backref=false,bookmarks,breaklinks=true,colorlinks=true,hypertexnames=false,plainpages=false,citecolor=DarkBlue,urlcolor=DarkBlue,linkcolor=DarkBlue,pdftitle={A 5/3 Guarantee for Echelon Stock (R,nQ) Policies in Two-Stage Stochastic Serial Systems},pdfauthor={Ming Hu}]{hyperref}
\urldef{\provemeurl}\url{https://prove2.me/missions/A_%245%2F3%24_Guarantee_for_Echelon_Stock_%24(R%2CnQ)%24_Policies_in_Two-Stage_Stochastic_Serial_Systems}
\bibpunct[, ]{(}{)}{,}{a}{}{,}
\def\bibfont{\fontsize{10.5}{17.5}\selectfont}

\setlength{\textfloatsep}{8pt plus 2pt minus 2pt}
\setlength{\intextsep}{8pt plus 2pt minus 2pt}
\TheoremsNumberedThrough\EquationsNumberedThrough\MANUSCRIPTNO{}
\newcommand{\OPTsh}{\mathrm{OPT}^{\mathrm{sh}}}
\newcommand{\OPTlot}{\mathrm{OPT}^{\mathrm{lot}}}
\newcommand{\Csh}{C^{\mathrm{sh}}}\newcommand{\Clot}{C^{\mathrm{lot}}}
\newcommand{\RnQ}{(R,nQ)}
\newcommand{\LBsp}{\mathrm{LB}^{\mathrm{SP}}}

\newcommand{\E}{\mathbb E}\newcommand{\ph}{\phi}

\let\OriginalDoubleSize\normalsizeXIyesCE
\renewcommand\normalsizeXIyesCE{\OriginalDoubleSize\baselineskip19.5pt}
\begin{document}
\raggedbottom
\RUNAUTHOR{Hu}

\RUNTITLE{A \(5/3\) Guarantee for Echelon Stock \((R,nQ)\) Policies}

\TITLE{\Large A \(5/3\) Guarantee for Echelon Stock \((R,nQ)\) Policies \\in Two-Stage Stochastic Serial Systems}
\ARTICLEAUTHORS{\AUTHOR{Ming Hu$^{*}$} \AFF{Rotman School of Management, University of Toronto,
Toronto, Ontario M5S 3E6, Canada\\
\EMAIL{ming.hu@rotman.utoronto.ca}}}

\ABSTRACT{Fixed shipment costs encourage large orders, while inventory and customer backlog costs favor frequent replenishment. In a serial system, the stages must also coordinate their lot sizes and material availability. We study the classical echelon stock \((R,nQ)\) policy: the downstream stage requests a fixed lot, the upstream stage orders an integer multiple of that lot, and transfers contain complete lots. For unit Poisson demand, full backlogging, and at both stages, deterministic lead times, and setup and holding costs, we prove that the optimized cost in this class is at most \(5/3\approx1.667\) of the optimal admissible system cost, uniformly over all nonnegative cost rates, demand rates, and lead times. The guarantee holds even when the selected policy pays a setup charge for every lot while the optimal benchmark pays only once per aggregate shipment. To obtain the guarantee, we use a common inventory calculation that allows independently chosen replenishment frequencies while respecting upstream material availability. We compare its optimized value, an upper bound on the best integer-ratio policy cost, with a novel lower bound on the optimal system cost that combines allocated single-stage costs with the minimum inventory cost when setup costs are zero. We also show that no uniform factor below \((1+\sqrt{3})/2\approx1.366\) is possible for the classical class. Existing numerical experiments place best-found classical policies within \(22.7\%\) of an evaluated lower bound across three broad grids. The results quantify both the reliability and the limits of a simple operating rule and explain how shipment consolidation, replenishment timing, and flexibility affect its value.}

\maketitle
\begingroup
\renewcommand{\thefootnote}{\fnsymbol{footnote}}%
\footnotetext[1]{The author used AI tools (such as ChatGPT, Claude, and Refine.ink) to assist with derivations, numerical verification, and editing. The author developed and verified the substantive research questions, theoretical framework, analysis, and results. The author takes sole responsibility for the paper's content. Propositions~\ref{prop:rnq-account-main} and~\ref{prop:two-scale-lower} and Theorems~\ref{thm:five-thirds} and~\ref{thm:impossibility}, including the $5/3$ guarantee, are formally verified in Lean~4; the formal proofs are public on Prove2Me at \provemeurl. Their formal statements, including the definitions of the model and of the policies, match those of the paper, and these proofs rely only on Lean's three standard axioms.}%
\endgroup

\vspace{-1.5em}

\section{Introduction}\label{sec:introduction}

A distribution center replenishes a local warehouse in truckloads, while the center itself replenishes from a supplier. A production facility supplies an assembly plant in batches while purchasing material in larger quantities upstream. In both settings, a fixed charge makes frequent replenishment expensive, but holding too much inventory ties up resources and exposes the system to storage costs. Customer backlog creates a third consideration. Replenishment lead times complicate these tradeoffs: orders must anticipate demand during transit, and upstream shortages can further delay downstream deliveries. The locations must, therefore, coordinate both order quantities and timing.

The classical echelon stock \((R,nQ)\) policy provides a compact way to coordinate these decisions. The downstream stage requests a fixed lot of \(Q\) units. The upstream stage orders \(nQ\) units, where \(n\) is a positive integer. Two reorder points determine when the requests are made. When material is available, transfers contain complete downstream lots. This policy is attractive because the operating instructions are easy to communicate, the replenishment quantities can be standardized, and the number of controls stays small. Its restrictions also create a legitimate concern: how much performance can be lost by committing to this structure?

The answer is not apparent from two separate lot-sizing calculations. One stage may prefer a large lot because its setup cost is high, while the other prefers a small lot because inventory is expensive there. The preferred upstream lot can even be smaller than the preferred downstream lot. For example, upstream replenishments may be inexpensive releases under a supplier agreement, while downstream transfers require a dedicated truck. The integer-multiple restriction then prevents the two separate preferred lot sizes from being used together. Moreover, a downstream order may have to wait for material even when its nominal inventory position has triggered replenishment. A cost comparison must account for that waiting and for the additional stock held while enough material accumulates to complete a transfer lot.

A central lesson of inventory theory is that simple operating rules can perform well even in complex systems. The \(98\%\)-effective policies of \citet{Roundy1985} demonstrate this principle for deterministic one-warehouse multiretailer systems: coordinating replenishment intervals through powers of two preserves nearly all the value of optimal control. Such a guarantee establishes the value of simplicity without requiring a complete description of an optimal policy. We pursue the same goal for a stochastic serial system with setup costs and lead times at both stages, where the optimal policy is unknown. Our \(5/3\) guarantee quantifies the performance that the classical echelon \((R,nQ)\) policy can retain across all parameter choices.

We study a two-stage continuous-review serial inventory system with unit Poisson demand, full backlogging, and at both stages, nonnegative deterministic lead times, and setup and holding costs. Shipment quantities are integer. Our main result is a uniform guarantee:
\begin{equation}
\inf_{\Pi\in\mathcal P_{\RnQ}}\Clot(\Pi)
\le \frac53\OPTsh.
\label{eq:intro-main}
\end{equation}
Here \(\mathcal P_{\RnQ}\) denotes classical rules with integer controls \(Q,n\ge1\) and a compatible initialization, as defined in Appendix~\ref{app:stationary-construction}. The left side charges a setup cost for every complete lot moved, including lots dispatched together. The benchmark \(\OPTsh\) allows all admissible integer shipment policies and charges once for each positive aggregate dispatch. Thus, the comparison gives the classical policy a conservative cost convention while allowing the benchmark to obtain shipment-consolidation savings. It implies a \(5/3\) guarantee when both policies use either convention.

The factor is independent of the relative setup costs, holding costs, backlog penalty, demand rate, and lead times. Both lead times may be positive or zero. At zero cost coefficients, the result is stated for the infimum of long-run average costs, allowing finite policies that approach the guarantee arbitrarily closely. With positive holding and backlog costs, the construction yields an attained finite policy. The theorem is a worst-case guarantee across environments, which does not mean that a well-chosen classical policy typically incurs a cost premium of 2/3.

We establish a limit on how far such a uniform guarantee can be improved. We construct instances in which even the best classical policy costs over \(36\%\) more than a feasible alternative outside the classical class. This alternative orders one unit at a time from the external supplier and combines available units into downstream shipments of varying sizes. Thus, no uniform factor below
\begin{equation}
\frac{1+\sqrt3}{2}\approx1.366
\label{eq:intro-impossibility}
\end{equation}
can hold for the classical class. This example explains why the classical link between the two lot sizes can be costly when inexpensive unit upstream replenishment can support infrequent downstream transfers. The upper guarantee \(5/3\) and the impossibility example place the worst possible loss from the classical restriction in a nontrivial range. The sharpness of a guarantee remains open.

We prove \eqref{eq:intro-main} using a simple coordination argument. We first imagine choosing the two replenishment cycle sizes separately while still accounting for the material link between the stages. This intermediate calculation has more freedom than an integer-ratio policy. Nevertheless, every pair of integer cycle sizes can be converted into a feasible classical policy at no greater cost. The conversion groups upstream cycle positions according to their remainder after division by the downstream lot size. Each group corresponds to a physical upstream lot that is an integer multiple of the downstream lot. The weights of the groups keep average setup cost no higher than in the intermediate calculation, so at least one of the resulting integer-ratio policies costs no more than the intermediate calculation. Such an integer-ratio policy itself does not randomize among the groups. Optimizing the intermediate calculation, therefore, gives an upper bound on the best integer-ratio policy cost.

It remains to compare this optimized intermediate value with the optimal system cost. We construct two lower bounds on the optimal system cost by allocating holding and backlog costs between two single-stage problems and the setup-costs-free system. These lower-bound formulas evaluate costs at doubled or quadrupled setup charges. To connect those calculations to the intermediate calculation that provides an upper bound on the best integer-ratio policy cost, we compare any pair of replenishment cycles with integer cycles of approximately half their respective sizes. Separating the even-numbered and odd-numbered positions quantifies the tradeoff between more frequent setups and lower average inventory costs. Applying this comparison to the two lower bounds gives estimates in terms of the original intermediate value and the minimum inventory cost when setups are free. The latter cost enters the estimates with opposite signs. A weighted average of the two lower bounds cancels it and shows that the optimal system cost is at least \(3/5\) of the optimized intermediate value. Combining this lower bound with the policy cost bound above gives the \(5/3\) guarantee.

The paper makes three contributions. First, it establishes a uniform guarantee for the exact integer classical class under both cost conventions. With either per-lot or per-shipment setup costs, a classical policy retains a guaranteed share of the best attainable performance even when independent stage calculations would recommend incompatible replenishment sizes. Second, we provide a way to turn independently chosen integer cycle sizes into implementable integer-ratio lots without adding an uncontrolled rounding cost. Third, the impossibility construction shows how a policy outside the classical class can reduce costs by combining frequent unit orders from the external supplier into downstream shipments of varying sizes.

The lower bound on the optimal system cost used to prove \eqref{eq:intro-main} also has a practical interpretation. A manager can compute two benchmarks using simpler inventory problems that optimize order quantities and stock levels for one stage at a time. These calculations also use the minimum holding and backlog cost of the full two-stage system when neither stage pays a setup charge. Their maximum is a valid lower bound on the optimal cost. Dividing a proposed policy's cost by the larger of these two benchmarks gives an instance-specific upper bound on its ratio to optimal cost, which can be much smaller than the uniform factor. We establish the uniform guarantee using only an equal allocation of the relevant costs and two choices of scale. Optimizing the cost allocation and scale can tighten the lower bound for a system with a particular set of primitives.

Our numerical study illustrates the performance of the classical policy. Across three broad parameter grids, the best-found classical policies cost at most \(22.7\%\) above the evaluated lower bound; the median gap is \(12.1\%\). A deliberately asymmetric stress grid produces larger ratios, reaching about \(1.60\). At the grid point with the largest policy-to-lower-bound ratio, changing the setup-accounting convention from charging per lot to charging per shipment improves the best evaluated classical cost by about \(8.7\%\). A finite search allowing incomplete downstream lots finds no improvement over the best classical policy found there, whose unit downstream lots cannot be split further because shipment quantities are integers.

For management, the results separate three decisions that are often combined. The first is whether to standardize replenishment quantities. The theorem quantifies a guarantee for doing so, while the impossibility example shows why standardization has a real potential cost. The second is how to coordinate the standardized quantities across stages. The construction permits common lots or larger upstream multiples and keeps the reorder-point difference available to adjust relative timing. The third is how to price a shipment that carries several lots. A fixed charge per dispatch and a fixed charge per lot can produce different preferred policies even for identical physical flows. Evaluating these decisions separately helps identify whether an observed cost gap calls for different lot sizes, different timing, or a different operating agreement.

Figure~\ref{fig:policy-choice} summarizes two ways to coordinate fixed replenishment quantities. Common lots use the same quantity at both stages; a larger upstream multiple spreads one upstream setup over several downstream lots. In either case, the reorder points remain available to adjust inventory protection and replenishment timing. We prove the constant guarantee by choosing these controls jointly.

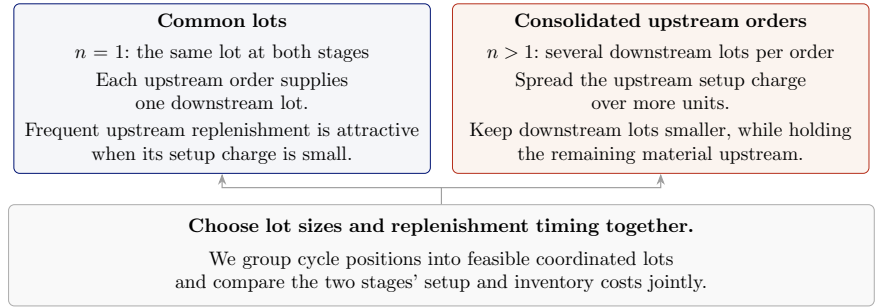
\begin{figure}[!htbp]
\centering
\begingroup\SingleSpaced
\resizebox{0.7\linewidth}{!}{%
\begin{tikzpicture}[x=1cm,y=1cm,every node/.style={font=\small,align=center},>=Stealth]
\node[draw=DarkBlue,fill=DarkBlue!4,rounded corners=3pt,text width=7.15cm,minimum height=2.95cm,inner sep=5pt,anchor=north] (com) at (-3.95,0) {\textbf{Common lots}\\[4pt]\(n=1\): the same lot at both stages\\[2pt]Each upstream order supplies\\one downstream lot.\\[2pt]Frequent upstream replenishment is attractive\\when its setup charge is small.};
\node[draw=BrickRed,fill=BrickRed!4,rounded corners=3pt,text width=7.15cm,minimum height=2.95cm,inner sep=5pt,anchor=north] (mult) at (3.95,0) {\textbf{Consolidated upstream orders}\\[4pt]\(n>1\): several downstream lots per order\\[2pt]Spread the upstream setup charge\\over more units.\\[2pt]Keep downstream lots smaller, while holding\\the remaining material upstream.};
\path (com.south) -- (mult.south) coordinate[midway] (base);
\node[draw=gray!65,fill=gray!5,rounded corners=3pt,text width=15.1cm,inner sep=7pt,anchor=north] (joint) at ([yshift=-0.55cm]base) {\textbf{Choose lot sizes and replenishment timing together.}\\[5pt]We group cycle positions into feasible coordinated lots\\and compare the two stages' setup and inventory costs jointly.};
\draw[->,semithick,gray!75] (joint.north) -- ([yshift=-0.25cm]base) -| (com.south);
\draw[->,semithick,gray!75] ([yshift=-0.25cm]base) -| (mult.south);
\end{tikzpicture}
}
\endgroup
\caption{Two ways to coordinate fixed lots while retaining control over replenishment timing.}
\label{fig:policy-choice}
\end{figure}

\section{Related Research}\label{sec:literature}

\citet{ClarkScarf1960} introduce the echelon-stock approach to serial inventory systems. An echelon counts a location and the material committed to downstream locations. This accounting makes the dependence between locations explicit and supports a decomposition when fixed ordering costs are absent. Positive fixed costs also require coordination of shipment timing and quantities \citep{ClarkScarf1962}. \citet{DeBodtGraves1985} study continuous-review policies with linked lot sizes. \citet{Chen} formalize the echelon stock \((R,nQ)\) class and derive its evaluation formulas. \citet{ChenZheng1998} and \citet{Chen2000} develop its multistage and batch-ordering structure. We retain this classical policy definition, including integer ratios and complete-lot transfers.

A second line of research develops tractable approximations and lower bounds. \citet{ShangSong2003} use single-stage comparisons to obtain newsvendor bounds and heuristics for serial base-stock systems. \citet{ShangSong2007} study serial systems with economies of scale. \citet{ShangSongZipkin2009} consider coordination in decentralized batch-ordering systems, and \citet{ShangZhou2010} study echelon policies that combine integer-ratio quantities with periodic review and fixed ordering and review costs. These contributions help managers select and coordinate replenishment controls. We seek a uniform guarantee for exact integer-ratio lots relative to all admissible policies.

Our lower-bound foundation is the serial-system analysis of \citet{ChenZheng1994LB}. Their induced-penalty method represents downstream shortages as a cost in an upstream inventory problem and includes setup costs at both stages. The methods of \citet{FedergruenZheng1992} and the structural results in \citet{Zheng1992} support the evaluation of single-stage \((r,Q)\) models. Our lower bound combines allocated single-stage costs with the zero-setup serial value. The $5/3$ argument uses an equal allocation and two setup scales to connect this bound to the joint inventory calculation underlying the policy upper bound. The numerical study evaluates the broader family of allocations and setup scales.

Several performance results show how guarantees depend on the policies allowed and the costs used for comparison. For a two-stage system with zero upstream lead time, \citet{Chen1999} establishes a \(94\%\)-effective policy within a broader continuous-quantity cyclic framework. That framework allows the first downstream shipment in a cycle to differ from later shipments. \citet{HuYang2014} study modified echelon \((r,Q)\) policies that permit unrestricted lot ratios and partial transfers. Under their continuous-quantity convention, an induced upstream lot at least as large as the induced downstream lot gives a \(5/4\) guarantee. Their general bound depends on the setup-cost ratio and their setup accounting charges shipments. Our theorem retains integer quantities, complete-lot transfers, both lead times, and a constant independent of the setup-cost ratio. It also charges the selected policy for each lot. \citet{ZhuChenHuYang2021} extend modified echelon policies to stochastic distribution systems and give an asymptotic guarantee as fixed-cost ratios scale.

Integer-ratio coordination has a long history in deterministic inventory theory. The power-of-two policies of \citet{Roundy1985} coordinate replenishment intervals on a common time scale. Related approximation algorithms address deterministic warehouse-retailer and joint replenishment models \citep{LeviEtAl2008,BienkowskiEtAl2014,Segev2025}. Their common managerial question is whether a simpler replenishment schedule sacrifices much efficiency. Our argument shares the idea of controlling setup and inventory costs through coordinated changes of scale. We separate the even-numbered and odd-numbered cycle positions to compare costs at two scales, while the implemented policy permits arbitrary integer lot ratios and responds to realized demand. Filling a downstream request also requires material already available upstream. Our grouping argument preserves this material constraint when converting independently selected cycle sizes into integer-ratio lots, and the \(5/3\) guarantee follows by comparing their joint cost with system lower bounds.

Other stochastic models use different benchmarks. \citet{ChuShen2010} study power-of-two ordering intervals for a warehouse-retailer system, and \citet{NajyEtAl2025} develop an approximation scheme for a related ordering-interval model. These studies use a service-level and safety-stock formulation and compare performance within that formulation or against the best policy in the corresponding class. \citet{LeviEtAl2017} provide a balancing approach for finite-horizon multiechelon systems without the fixed shipment charges studied here. Asymptotic performance guarantees for dual-sourcing and lost-sales inventory models concern different restrictions \citep{XinGoldberg2018,Xin2022}. \citet{HuhJanakiraman2010} study serial systems in which unmet demand is lost, whereas our model fully backlogs unmet demand.

Our contribution combines an upper and a lower statement about the same classical restriction. The upper statement shows that the infimum cost of classical integer-ratio lots is within \(5/3\) of the admissible optimum. The lower statement shows that the best uniform constant cannot be arbitrarily close to one, even when each classical policy uses its best controls. We prove the \(5/3\) guarantee using a valid system lower bound; the lower example compares every classical policy with a specific feasible policy outside the restriction.
\section{The System and Its Cost Conventions}\label{sec:model}

\subsection{Material Flow and Decisions}

Stage~1 serves customers and replenishes from Stage~2. Stage~2 replenishes from an external supplier with unlimited supply. Customer demand arrives one unit at a time according to a Poisson process of rate \(\lambda\ge0\). The time from a Stage-2 dispatch to its arrival at Stage~1 is the deterministic lead time \(L_1\); upstream orders reach Stage~2 after deterministic lead time \(L_2\). Both times are finite and nonnegative. A customer whose demand cannot be filled immediately waits in backlog. There is no limit on order or shipment size other than material availability and integrality.

This is a model of a serial supply chain with stable demand characteristics. It permits substantial uncertainty in the number and timing of demands during a lead time, but it does not model capacity congestion, perishability, lost sales, or changing demand regimes. A policy can use the demand and inventory history to choose the timing and quantity of each dispatch, but cannot use future demand. A request to the upstream stage is distinct from a physical shipment: requesting material does not make it available downstream immediately, because the material must first be available upstream and then traverse the downstream lead time.

Let \(K_i\ge0\) be the fixed setup charge associated with a Stage-\(i\) dispatch. For every \(\lambda\ge0\), write \(k_i=\lambda K_i\), so \(k_i=0\) when demand is zero. This normalization makes the setup-cost rate of a fixed lot of \(Q_i\) units equal to \(k_i/Q_i\), since the stage ultimately supplies demand at rate \(\lambda\). The parameters \(h_1,h_2,p\ge0\) are the two echelon holding rates and the customer backlog rate. Throughout the paper, a cost rate means long-run expected dollars per unit of time.

An echelon counts a location and the physical material already committed to locations downstream of it. Let \(I_1(t)\) denote physical inventory at Stage~1, let \(I_2(t)\) denote physical inventory in the Stage-2 echelon, and let \(\mathsf B(t)\) be customer backlog. The Stage-2 echelon includes stock at Stage~2, material traveling to Stage~1, and stock at Stage~1. Consequently, a unit at Stage~1 carries the combined holding rate \(h_1+h_2\), while material still upstream carries the applicable upstream rate. The inventory and backlog cost rate is
\begin{equation}
c_I(t)=h_1I_1(t)+h_2I_2(t)+p\mathsf B(t).
\label{eq:instant-cost}
\end{equation}
The echelon rates are incremental cost rates \citep{Zipkin2000}. This convention avoids charging an upstream unit as though it had already incurred the full downstream holding cost.

Net echelon inventory subtracts customer backlog from physical echelon inventory. We denote it by \(\mathit{IL}_i(t)=I_i(t)-\mathsf B(t)\). The actual echelon inventory position \(\mathit{IP}_i(t)\) adds material dispatched toward that echelon but still in its incoming lead time. For Stage~1, a pending request that has not left Stage~2 is excluded. The key material restriction is that the actual Stage-1 position cannot exceed the available net Stage-2 inventory. Appendix~\ref{app:admissibility} gives the exact flow equations and the treatment of simultaneous arrivals, demand, and dispatches when a lead time is zero.

We require a policy to be feasible, to use only current and past information, and to keep expected costs finite on finite horizons. We also require long-run flow balance: customer deliveries and the quantities passing through both stages have average rate \(\lambda\). Terminal inventory, pipeline, and end-of-horizon unfinished-cycle effects must be negligible relative to the observation horizon. These conditions rule out satisfying an apparent cost bound by indefinitely postponing a positive fraction of demand or by moving a large cost beyond the reporting horizon. Their precise form is in Appendix~\ref{app:admissibility}. Every finite stationary classical policy constructed in this paper satisfies them. Our optimal benchmark uses the same admissible class throughout.

The nonnegative parameter boundaries deserve an operational interpretation. With \(\lambda=0\), empty initialization and no orders give zero cost. When a holding or backlog rate is zero, a sequence of increasingly large or increasingly delayed finite policies may approach an infimum without attaining it. We, therefore, formulate the full-space guarantee using cost infima. This convention preserves service and flow requirements even with no backlog charge.

\subsection{Shipment-Based and Lot-Based Setup Costs}

A truck may carry several replenishment lots. Whether it incurs one fixed charge or several depends on the activity represented by the charge. A vehicle dispatch fee is naturally assessed once per shipment. A cleaning, inspection, or handling charge incurred for each production lot can remain payable even when lots travel together. This distinction determines consolidation savings.

Under \emph{shipment accounting}, every positive aggregate dispatch at Stage~\(i\) incurs \(K_i\). The dispatch can contain any positive integer number of units available to move. Write \(Z_i^{\rm sh}(T)\) for its number of positive dispatches up to time \(T\). For a policy \(\Pi\), define
\begin{equation}
\Csh(\Pi)=\limsup_{T\to\infty}\frac1T\E_\Pi\left[\int_0^T c_I(t)\,dt+\sum_{i=1}^2 K_i Z_i^{\rm sh}(T)\right].
\label{eq:shipment-cost}
\end{equation}
The benchmark \(\OPTsh\) is the infimum of this cost over all admissible integer shipment policies. These policies may vary shipment quantities with the observed state.

Under \emph{lot accounting}, a policy first chooses an integer lot size at each stage and subsequently dispatches complete lots of those sizes. Each lot incurs \(K_i\), including every lot in a dispatch containing several lots. Let \(Z_i^{\rm lot}(T)\) count these lots with multiplicity, and let \(\Clot(\Pi)\) be defined as in \eqref{eq:shipment-cost} with these counts. Write \(\OPTlot\) for the infimum over the corresponding admissible fixed-lot policies. For three lots dispatched together, shipment accounting charges \(K_i\); lot accounting charges \(3K_i\).

The admissible fixed-lot class is broader than the classical \((R,nQ)\) class. A fixed-lot policy may choose any integer lot size at each stage and may time orders and dispatches using any admissible history-based rule. For example, downstream lots of four units and upstream lots of six units define a fixed-lot policy, but cannot define a classical \((R,nQ)\) policy because six is not an integer multiple of four. A classical \((R,nQ)\) policy also uses fixed reorder points. Therefore, $
\{\text{classical }(R,nQ)\text{ policies}\}
\subseteq
\{\text{admissible fixed-lot policies}\}.
$

Removing the lot labels from a fixed-lot policy preserves its material flows and cannot increase its shipment-accounted cost. Thus, 
\begin{equation}
\Csh(\Pi)\le\Clot(\Pi),\qquad \OPTsh\le\OPTlot.
\label{eq:cost-order}
\end{equation}
Our main guarantee compares the larger classical-policy cost with the smaller benchmark. It, therefore, remains valid when the manager and the optimal comparison use the same convention. By contrast, the impossibility result in Section~\ref{sec:impossibility} is measured against \(\OPTsh\) only. This benchmark permits the variable shipment quantities used in the construction, which lie outside the fixed-lot class that defines \(\OPTlot\).

The conventions change setup accounting and can, therefore, change the optimal policy, but they do not change which physical actions are feasible. All policies operate under the same demand process, lead times, material availability constraints, and holding and backlog rates. The classical \((R,nQ)\) policy offers a simple way to coordinate replenishment using only a few fixed controls.

\section{The Classical Rule and Its Exact Cost}\label{sec:policy}

\subsection{Four Controls with Distinct Operational Roles}\label{sec:event-rule}

A classical policy chooses two integer reorder points \(R_1,R_2\), a positive integer downstream lot size \(Q\), and a positive integer multiplier \(n\). Stage~1 requests a lot of \(Q\) units when its nominal echelon inventory position crosses its reorder point. Stage~2 orders \(nQ\) from the external supplier when its echelon position crosses its own reorder point. A downstream request waits if a complete lot is unavailable. Once material is available, the stage sends the pending complete lots and the transportation lead time begins. All four controls are fixed before future demands are observed.

To make the timing convention explicit, we initialize the complete-lot rule with a whole number of downstream lots held at Stage~2 and the compatible pending requests and pipelines defined in Appendix~\ref{app:stationary-construction}. We use two counters to record Stage~1's replenishment requests and Stage~2's orders to the external supplier; both are driven by the same cumulative demand count. Its starting point within the replenishment cycle can be any fixed remainder, without affecting the long-run cost. Appendix~\ref{app:initial-remainder} shows that allowing an additional upstream remainder smaller than \(Q\) cannot improve the cost optimized over the four controls.

The two reorder points are allowed to be negative. With backlogging, a negative reorder point can be an intentional response to a low backlog penalty or expensive inventory. The quantities remain positive integers, and a shipment can never contain material that has not arrived. Appendix~\ref{app:policy-cost} constructs the rule from demand and shipment counters and verifies these requirements directly.

For a fixed multiplier \(n\), it is useful to think in terms of three remaining decisions: \(Q\), the downstream cycle endpoint \(s=R_1+Q\), and the reorder-point difference \(\ell=R_2-R_1\). The lot size \(Q\) determines the downstream request rate and the setup rate under lot accounting. The endpoint \(s\) sets the upper limit of the nominal downstream inventory-position cycle. The difference \(\ell\) determines the relative replenishment timing of the stages. Holding \(Q\) and \(s\) fixed, positive \(\ell\) shifts upstream orders earlier and negative \(\ell\) shifts them later compared with \(\ell=0\). Raising both reorder points by the same amount increases \(s\) while keeping \(\ell\) fixed, adding inventory protection and preserving the stages' relative replenishment timing.

Figure~\ref{fig:rnq-regimes} separates the quantity decision from the reorder-point decision. Panels (a) and (b) use common lots with different alignments. Panels (c) and (d) put three downstream lots in each upstream order, with opposite signs of \(R_2-R_1\). These are the four best-found configurations reported in Table~\ref{tab:configurations}. The columns show nominal inventory-position ranges; their lower edges locate the reorder points. A higher upstream reorder point provides more inventory protection upstream, while a larger multiplier changes the amount ordered each time.

\begin{figure}[t]
\centering
\begingroup\SingleSpaced
\definecolor{RegimeBlue}{HTML}{2C6E9B}
\definecolor{RegimeOrange}{HTML}{C46D27}
\definecolor{RegimePurple}{HTML}{6A1B9A}
\setlength{\fboxsep}{0.35pt}
\setlength{\unitlength}{0.01\linewidth}
\scalebox{0.756}{%
\begin{picture}(100,74)
  \put(0,0){\includegraphics[
    width=100\unitlength,
    height=74\unitlength
  ]{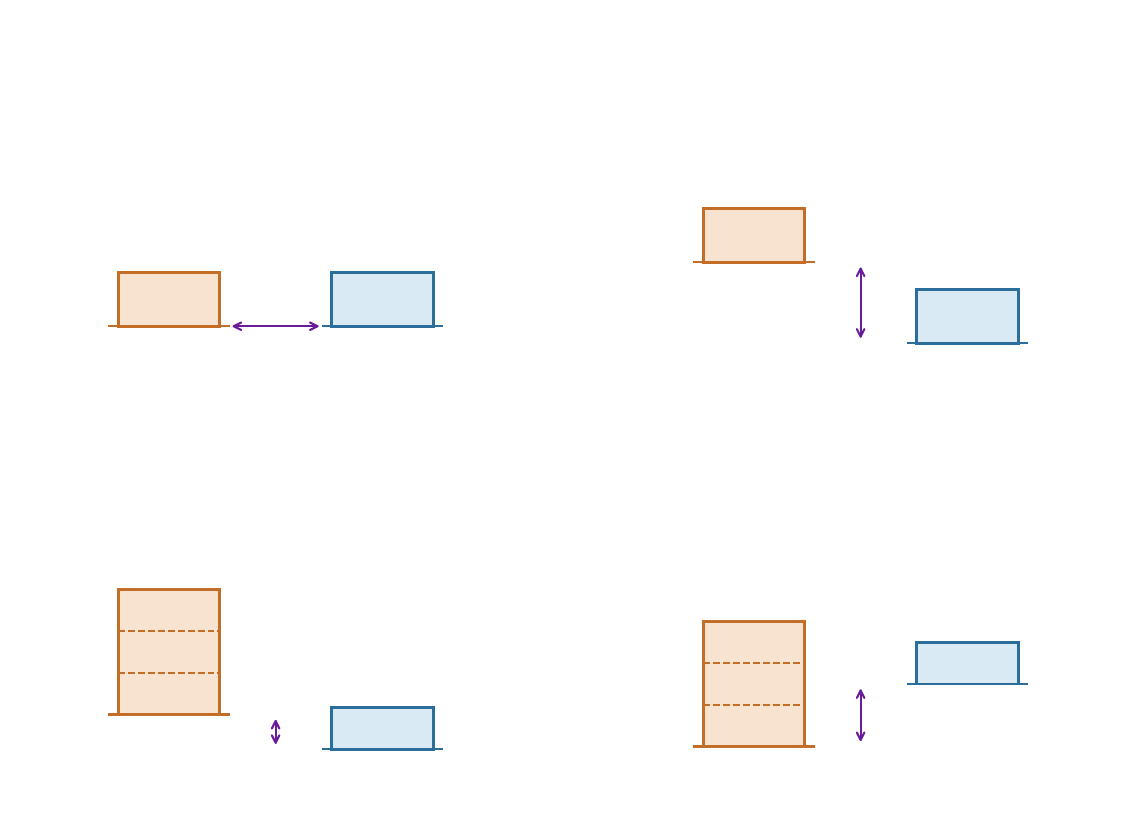}}

  \put(1,68.0){\makebox(0,0)[lb]{\small\bfseries
    (a) \(n=1,\ R_2=R_1\)}}
  \put(15,62.0){\makebox(0,0)[b]{\footnotesize\bfseries\color{RegimeOrange}
    \shortstack{Stage 2\\[-0.4ex](upstream)}}}
  \put(34,62.0){\makebox(0,0)[b]{\footnotesize\bfseries\color{RegimeBlue}
    \shortstack{Stage 1\\[-0.4ex](downstream)}}}
  \put(15,47.4){\makebox(0,0){\small\(nQ=4\)}}
  \put(34,47.4){\makebox(0,0){\small\(Q=4\)}}
  \put(9.2,45.0){\makebox(0,0)[r]{\small\textcolor{RegimeOrange}{\(R_2\)}}}
  \put(40.8,45.0){\makebox(0,0)[l]{\small\textcolor{RegimeBlue}{\(R_1\)}}}
  \put(24.5,42.5){\makebox(0,0){\colorbox{white}{\footnotesize
    \textcolor{RegimePurple}{\(R_2-R_1=0\): aligned}}}}
  \put(1,39.2){\makebox(0,0)[lb]{\footnotesize
    \((R_1,R_2,Q,n)=(0,0,4,1)\)}}
  \put(1,37.2){\parbox[t]{46\unitlength}{\raggedright\footnotesize
    Common lots and equal reorder points.}}

  \put(53,68.0){\makebox(0,0)[lb]{\small\bfseries
    (b) \(n=1,\ R_2>R_1\)}}
  \put(67,62.0){\makebox(0,0)[b]{\footnotesize\bfseries\color{RegimeOrange}
    \shortstack{Stage 2\\[-0.4ex](upstream)}}}
  \put(86,62.0){\makebox(0,0)[b]{\footnotesize\bfseries\color{RegimeBlue}
    \shortstack{Stage 1\\[-0.4ex](downstream)}}}
  \put(67,53.1){\makebox(0,0){\small\(nQ=4\)}}
  \put(86,45.9){\makebox(0,0){\small\(Q=4\)}}
  \put(61.2,50.7){\makebox(0,0)[r]{\small\textcolor{RegimeOrange}{\(R_2\)}}}
  \put(92.8,43.5){\makebox(0,0)[l]{\small\textcolor{RegimeBlue}{\(R_1\)}}}
  \put(76.5,42.5){\makebox(0,0){\colorbox{white}{\footnotesize
    \textcolor{RegimePurple}{\(R_2-R_1=6\)}}}}
  \put(53,39.2){\makebox(0,0)[lb]{\footnotesize
    \((R_1,R_2,Q,n)=(2,8,4,1)\)}}
  \put(53,37.2){\parbox[t]{46\unitlength}{\raggedright\footnotesize
    Common lots; higher upstream reorder point.}}

  \put(1,32.0){\makebox(0,0)[lb]{\small\bfseries
    (c) \(n=3,\ R_2>R_1\)}}
  \put(15,26.0){\makebox(0,0)[b]{\footnotesize\bfseries\color{RegimeOrange}
    \shortstack{Stage 2\\[-0.4ex](upstream)}}}
  \put(34,26.0){\makebox(0,0)[b]{\footnotesize\bfseries\color{RegimeBlue}
    \shortstack{Stage 1\\[-0.4ex](downstream)}}}
  \put(15,16.03){\makebox(0,0){\small\(nQ=18\)}}
  \put(34,9.25){\makebox(0,0){\small\(Q=6\)}}
  \put(9.2,10.48){\makebox(0,0)[r]{\small\textcolor{RegimeOrange}{\(R_2\)}}}
  \put(40.8,7.40){\makebox(0,0)[l]{\small\textcolor{RegimeBlue}{\(R_1\)}}}
  \put(17,8.94){\makebox(0,0){\colorbox{white}{\footnotesize
    \textcolor{RegimePurple}{\(R_2-R_1=5\)}}}}
  \put(1,3.2){\makebox(0,0)[lb]{\footnotesize
    \((R_1,R_2,Q,n)=(0,5,6,3)\)}}
  \put(1,1.2){\parbox[t]{46\unitlength}{\raggedright\footnotesize
    Three downstream lots per upstream order.}}

  \put(53,32.0){\makebox(0,0)[lb]{\small\bfseries
    (d) \(n=3,\ R_2<R_1\)}}
  \put(67,26.0){\makebox(0,0)[b]{\footnotesize\bfseries\color{RegimeOrange}
    \shortstack{Stage 2\\[-0.4ex](upstream)}}}
  \put(86,26.0){\makebox(0,0)[b]{\footnotesize\bfseries\color{RegimeBlue}
    \shortstack{Stage 1\\[-0.4ex](downstream)}}}
  \put(67,13.20){\makebox(0,0){\small\(nQ=6\)}}
  \put(86,15.05){\makebox(0,0){\small\(Q=2\)}}
  \put(61.2,7.65){\makebox(0,0)[r]{\small\textcolor{RegimeOrange}{\(R_2\)}}}
  \put(92.8,13.20){\makebox(0,0)[l]{\small\textcolor{RegimeBlue}{\(R_1\)}}}
  \put(76.5,5.9){\makebox(0,0){\colorbox{white}{\footnotesize
    \textcolor{RegimePurple}{\(R_2-R_1=-3\)}}}}
  \put(53,3.2){\makebox(0,0)[lb]{\footnotesize
    \((R_1,R_2,Q,n)=(5,2,2,3)\)}}
  \put(53,1.2){\parbox[t]{46\unitlength}{\raggedright\footnotesize
    Three lots; lower upstream reorder point.}}
\end{picture}}
\endgroup
\caption{Common and consolidated lots with different reorder-point alignments.}
\label{fig:rnq-regimes}
\begin{minipage}{0.98\linewidth}
\SingleSpaced\footnotesize\emph{Notes.} Columns show nominal echelon inventory-position ranges above the reorder points. Their heights are \(Q\) downstream and \(nQ\) upstream; dashed lines divide an upstream lot into downstream lots. Panels (a) and (b) share a vertical scale; the other panels use separate scales. Panels (a)--(d) show the best-found policies for environments A--D in Table~\ref{tab:configurations}.
\end{minipage}
\end{figure}
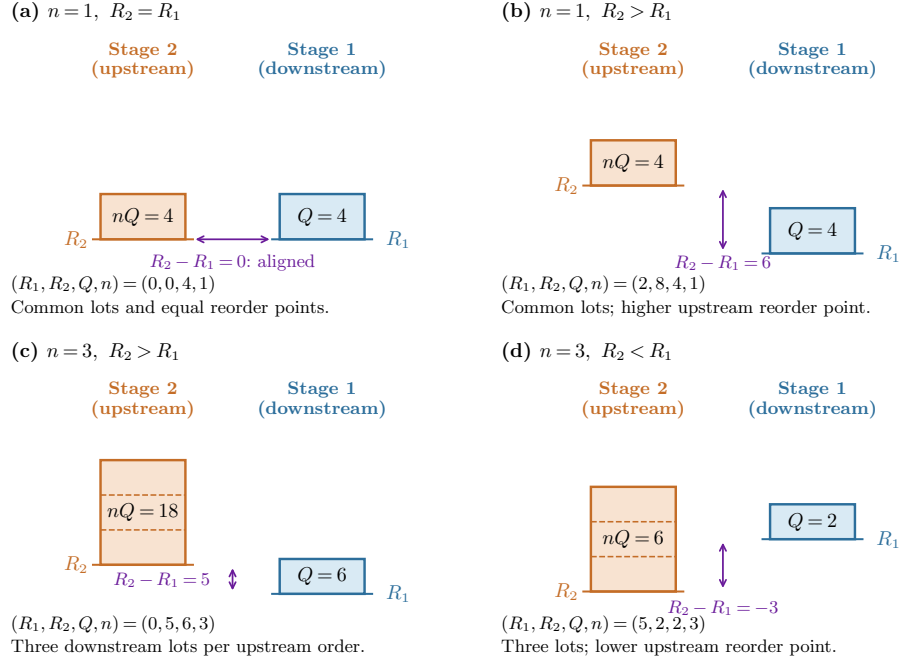

For example, choosing \(Q=4\) and \(n=3\) makes each upstream order provide three downstream lots. The center can send one lot when the first request is fillable and retain the other two lots for later requests. Raising the common endpoint adds protection against customer demand. Increasing \(\ell\), holding the downstream endpoint \(s\) fixed, brings more material into the upstream echelon relative to the downstream request schedule. The latter can reduce requests waiting for material, but increases stock held upstream. Neither adjustment changes the setup rates under lot accounting. Under shipment accounting, timing can change how many complete lots share a dispatch charge.

\subsection{Averaging the Positions Within a Replenishment Cycle}

The exact cost calculation can be understood by observing a system at a randomly selected stationary time. Unit Poisson demand moves the position through the \(Q\) points of a downstream cycle, with each point receiving the same weight. Likewise, an upstream order contains \(n\) equally weighted component lots. To express these averages, we imagine selecting one of the \(Q\) downstream positions and one of the \(n\) upstream component lots independently, with each position and each lot equally likely. The operating policy remains deterministic.

Let \(D_i\) be Poisson demand during lead time \(L_i\), with \(D_1,D_2\) independent. They represent demand over disjoint time intervals in the cost calculation. This does not assume that lead-time demands observed at successive decision epochs are independent. For nonnegative rates \(a,b\), define \(\ph_{a,b}(x)=ax^++bx^-\), where \(x^+=\max(x,0)\) and \(x^-=\max(-x,0)\). The downstream cost function is
\begin{equation}
G_0(x)=H_0+\E\ph_{h_1+h_2,p}(x-D_1),\qquad H_0=h_2\lambda L_1.
\label{eq:G0-main}
\end{equation}
The constant \(H_0=h_2\lambda L_1\) is the holding cost per unit of time for the average \(\lambda L_1\) units in transit downstream. The remaining term captures the expected holding cost of stock left over and the backlog cost of unmet demand after the lead time \(L_1\).

For a function \(F\) and \(j\in\mathbb N=\{1,2,\ldots\}\), let
\begin{equation}
\mathcal L_j F(x)=\frac1j\sum_{v=0}^{j-1}F(x-v),\qquad
\mathcal H_j(F)=\inf_{x\in\mathbb R}\mathcal L_jF(x).
\label{eq:lattice-window-main}
\end{equation}
Thus, \(\mathcal L_jF\) is the average over a block of \(j\) consecutive inventory positions; \(\mathcal H_j(F)\) places that block as cheaply as possible. With integer demand, the convex functions used here are linear between consecutive integer points. Their minima have integer placements whenever attained, so the calculation keeps the physical integer controls.

For downstream endpoint \(s\), difference \(\ell\), and lot sizes \(Q,nQ\), the inventory and backlog cost is
\begin{equation}
I_{n,Q}(s,\ell)=\frac1n\sum_{j=0}^{n-1}\E\left[
\mathcal L_QG_0\bigl(s-(\ell+jQ-D_2)^-\bigr)
+h_2(\ell+jQ-D_2)^+\right].
\label{eq:rnq-account-main}
\end{equation}
The expression has two operational cases. If enough material is available upstream to fill the request, the positive part charges the units waiting upstream. Otherwise, the negative part shifts the entire downstream cycle to lower positions, accounting for a request that must wait. Thus, the formula accounts for material delays in filling requests.

\begin{proposition}[{\sc Exact Cost of the Classical Rule}]\label{prop:rnq-account-main}
For \(\lambda>0\) and integer controls \(R_1,R_2,Q,n\), with \(Q,n\ge1\), the classical rule with canonical initialization is feasible, with
\begin{equation}
\Clot(\Pi)=\frac{k_1}{Q}+\frac{k_2}{nQ}+I_{n,Q}(R_1+Q,R_2-R_1),
\qquad \Csh(\Pi)\le\Clot(\Pi).
\label{eq:rnq-exact-main}
\end{equation}
\end{proposition}

The two setup terms follow from demand balance: every \(Q\) units require a downstream lot, and every \(nQ\) units require an upstream lot. The inventory term averages the corresponding cycle positions, then demand during each transportation lead time. A proof from the physical shipment counters is in Appendix~\ref{app:policy-cost}. The proposition expresses the total expected cost rate as the sum of setup, holding, and backlog costs.

An alternative form helps choose reorder points. Let \(V_Q\) be uniform on \(0,\ldots,Q-1\), and let \(J_n\) be uniform on \(0,\ldots,n-1\), independently of each other and of demand. Define
\begin{equation}
D_{\rm eff}=D_1+V_Q+(D_2-\ell-QJ_n)^+.
\label{eq:effective-demand}
\end{equation}
Then \eqref{eq:rnq-account-main} can be written as
\begin{equation}
I_{n,Q}(s,\ell)=H_0+h_2\E(\ell+QJ_n-D_2)^+
+\E\ph_{h_1+h_2,p}(s-D_{\rm eff}).
\label{eq:rnq-intuition}
\end{equation}
Effective demand includes customer demand \(D_1\), the position within the downstream cycle \(V_Q\), and material unavailable upstream \((D_2-\ell-QJ_n)^+\). At fixed \(Q,n,\ell\), only the last term depends on \(s\). The cheapest endpoint balances the probability of excess stock against the probability of backlog. For positive holding and backlog rates, it is an integer quantile at probability \(p/(h_1+h_2+p)\). This reduces the choice of \(s\) to a simple one-dimensional calculation.

Larger \(n\) spreads each upstream setup over more downstream lots, but may increase inventory waiting upstream. Adjusting the reorder-point difference \(\ell\) changes how early that material arrives relative to those requests. Choosing \(n\) only from an upstream setup calculation perspective would miss this adjustment. Conversely, a common lot, \(n=1\), eliminates partially used upstream orders and may be attractive when upstream setups are cheap. The theorem will guarantee that some coordinated choice is affordable without requiring common lots \(n=1\) in every environment.
\section{A Combined Lower Bound on the Optimal System Cost}\label{sec:benchmark}

To prove the performance guarantee, we combine two lower bounds on every feasible policy's cost. One allocates holding and backlog costs between two single-stage problems with setup charges \(K_1\) and \(K_2\); the other gives the minimum inventory cost when setup charges are zero. Rescaling the setup charges allows different shares of inventory cost to enter these two comparisons. Two choices of allocation and scale connect the resulting bound to the joint inventory calculation used to construct a classical policy.

\subsection{The Cost of Operating Without Setup Charges}

Imagine first that both fixed charges are zero. The manager can order and transfer individual units whenever doing so helps. Demand uncertainty and transportation time still create holding and backlog costs. Let \(V_0\) be the optimal cost of this system. With \(D_i\) denoting independent Poisson demand over lead time \(L_i\), its expression is
\begin{equation}
 V_0=\inf_{r,y\in\mathbb R}\E\left[G_0\bigl(r-(y-D_2)^-\bigr)+h_2(y-D_2)^+\right].
 \label{eq:lb-protected-zero-value}
\end{equation}
Here \(r\) is a downstream target and \(r+y\) is the upstream echelon target. When \(D_2\le y\), enough material is available to replenish the downstream position to its target \(r\), leaving \(y-D_2\) units upstream. When supply is insufficient, the downstream position falls by the shortage \((y-D_2)^-\). The two terms, therefore, evaluate the same realized availability situation.

The infimum in \eqref{eq:lb-protected-zero-value} is attained by integer controls in the regular positive-cost case. At a zero-cost boundary, it can be approached with finite integer controls. More importantly for the comparison, the inventory and backlog component of every admissible policy has a long-run lower limiting cost of at least \(V_0\). Appendix~\ref{app:zero-setup} proves this directly from the material restriction and two lead-time identities. It establishes the comparison on finite horizons before taking a long-run limit, including cases in which separate cost components do not converge.

\subsection{Dividing Cost Without Counting It Twice}

The second benchmark divides the physical holding and backlog charges between two single-stage problems. Let \(\eta\in[0,1]\) be the share of upstream echelon holding assigned to the first problem, and let \(\zeta\in[0,1]\) be its share of backlog cost. The remaining shares go to the second problem. Each problem retains its own stage's fixed setup charge. Define
\begin{align}
 F_1^{\eta,\zeta}(x)&=\E\phi_{h_1+\eta h_2,\zeta p}(x-D_1),\label{eq:allocation-loss-first}\\
 F_2^{\eta,\zeta}(x)&=\E\phi_{(1-\eta)h_2,(1-\zeta)p}(x-D_1-D_2).
 \label{eq:allocation-losses}
\end{align}
The longer demand window \(L_1+L_2\) in the second problem reflects the time needed to reach the customer from the external supplier. The losses in \eqref{eq:allocation-loss-first} and \eqref{eq:allocation-losses} are obtained by applying material feasibility to the allocated physical charges. Each single-stage problem includes only its allocated holding and backlog costs and its own stage's setup charge.

For a loss function \(F\), write
\begin{equation}
 \mathcal H_j(F)=\inf_{s\in\mathbb R}\frac1j\sum_{v=0}^{j-1}F(s-v),\qquad
 \mathcal C_F(k)=\inf_{j\in\mathbb N}\{k/j+\mathcal H_j(F)\}.
 \label{eq:cycle-cost-main}
\end{equation}
The first expression places a cycle of \(j\) consecutive inventory positions. The second chooses its integer size and pays the corresponding setup rate. Poisson arrivals see the same expected state cost as a randomly selected time. Expected inventory cost can, therefore, be written as a sum of costs evaluated at the inventory levels immediately before each demand arrival, divided by the demand rate. The position path splits this sum into completed cycles plus a negligible unfinished part. For each completed cycle of \(j\) demand arrivals, the costs evaluated at those arrivals plus the scaled setup charge \(k=\lambda K\) total at least \(j\mathcal C_F(k)\). This makes \(\mathcal C_F\) a lower bound for arbitrary admissible shipment decisions, including decisions with variable quantities (see Appendix~\ref{app:lower-bounds}). %

The optimized allocation bound is
\begin{equation}
 \mathcal A(k_1,k_2)=\sup_{0\le\eta,\zeta\le1}
 \{\eta H_0+\mathcal C_{F_1^{\eta,\zeta}}(k_1)+\mathcal C_{F_2^{\eta,\zeta}}(k_2)\},
 \qquad H_0=h_2\lambda L_1.
 \label{eq:allocation-bound}
\end{equation}
Every choice of the two shares, \(\eta\) and \(\zeta\), provides a valid lower bound. Taking their supremum strengthens the comparison. The constant \(H_0\) represents the upstream echelon holding associated with material traveling downstream at the demand rate. Assigning each physical charge only once avoids double counting when the two single-stage costs are added.

\subsection{Combining the Allocation and Zero-Setup Bounds}

On any finite horizon, let a policy have expected average inventory cost \(I\) and expected average setup cost \(\bar S\). For any \(0<u\le1\), its total cost can be written as
\[
 I+\bar S=(1-u)I+u(I+\bar S/u).
\]
This identity reserves a fraction \(1-u\) of inventory cost for the zero-setup comparison. The remaining fraction sees setup charges increased by \(1/u\), so that the original total setup cost is preserved. Applying the two lower bounds to the same physical policy gives
\begin{equation}
 \LBsp=\sup_{0<u\le1}\{(1-u)V_0+u\mathcal A(k_1/u,k_2/u)\}
 \le\OPTsh\le\OPTlot.
 \label{eq:sp-bound-main}
\end{equation}
The superscript SP refers to preserving the setup-cost contribution during this rescaling. The finite-horizon identity validates it under Section~\ref{sec:model}'s admissibility conditions.

The uniform guarantee theorem uses only \(u=1\) and \(u=1/2\) in \eqref{eq:sp-bound-main}, with equal allocation shares \(\eta=\zeta=1/2\) in \eqref{eq:allocation-loss-first} and \eqref{eq:allocation-losses}. With equal allocation, we obtain simpler lower bounds by replacing downstream lead-time demand \(D_1\) in the second problem by its mean. Optimizing the inventory level absorbs this constant shift, leaving only \(D_2\) in that problem's loss function. Define the downstream and upstream single-stage benchmark costs
\begin{equation}
 B_1(k)=H_0+\mathcal C_{\E\phi_{h_1+h_2,p}(x-D_1)}(k),\qquad B_2(k)=\mathcal C_{\E\phi_{h_2,p}(x-D_2)}(k).
 \label{eq:BT-main}
\end{equation}
These quantities retain the full relevant holding and backlog rates at each stage. They enter through equal allocation of upstream holding and backlog charges, together with a comparison of demand windows proved in Appendix~\ref{app:joint-cycles}.

\begin{proposition}[{\sc Two Sufficient Lower Bounds}]\label{prop:two-scale-lower}
For every instance,
\begin{equation}
 L_{\rm two}=\max\left\{\frac{B_1(2k_1)+B_2(2k_2)}2,\quad
 \frac{V_0}2+\frac{B_1(4k_1)+B_2(4k_2)}4\right\}
 \le\LBsp\le\OPTsh.
 \label{eq:two-scale-lower}
\end{equation}
\end{proposition}

\proof{Proof.} Set both allocation shares to one half. The Stage-1 loss is at least \((1/2)\E\phi_{h_1+h_2,p}(x-D_1)\). Replacing the independent downstream lead-time demand \(D_1\) in the Stage-2 loss by its mean gives at least \((1/2)\E\phi_{h_2,p}(x-\E D_1-D_2)\) by the convexity of \(\phi_{h_2,p}\). The optimizing cycle placement absorbs the shift \(\E D_1\). Including \(H_0/2\) gives \(\mathcal A(k_1,k_2)\ge\{B_1(2k_1)+B_2(2k_2)\}/2\), as detailed in \eqref{appjc:equal-allocation}. This calculation uses only the convexity of the loss functions, nonnegative cost coefficients, and infima, so it also holds when a coefficient is zero. Now use \(u=1\) and \(u=1/2\) in \eqref{eq:sp-bound-main}, respectively:
\[
 \thickmuskip=2mu \LBsp\ge\mathcal A(k_1,k_2)\ge\tfrac12\{B_1(2k_1)+B_2(2k_2)\},\enspace \LBsp\ge\tfrac12 V_0+\tfrac12\mathcal A(2k_1,2k_2)\ge\tfrac12 V_0+\tfrac14\{B_1(4k_1)+B_2(4k_2)\}.
\]
Their maximum is \(L_{\rm two}\), and Appendix~\ref{app:zero-setup} proves \(\LBsp\le\OPTsh\).\Halmos\endproof

The first bound in the composite bound \(L_{\rm two}\) uses all cost in the allocation comparison. The second reserves half the inventory cost for \(V_0\). Their usefulness depends on how high the zero-setup cost is relative to the cost of replenishment cycles. When evaluating a specific policy, one can also optimize \(u,\eta,\zeta\) to obtain a stronger numerical benchmark. We establish the uniform factor using only the two displayed candidates, without solving that additional optimization.
\section{A \texorpdfstring{$5/3$}{5/3} Guarantee for the Classical Policy}\label{sec:guarantee}

\begin{theorem}[{\sc The $5/3$ Guarantee}]\label{thm:five-thirds}
For unit Poisson demand with rate \(\lambda\ge0\), finite deterministic lead times, and all finite nonnegative cost rates,
\begin{equation}
 \inf_{\Pi\in\mathcal P_{\RnQ}}\Clot(\Pi)
 \le\frac53 L_{\rm two}\le\frac53\LBsp\le\frac53\OPTsh\le\frac53\OPTlot.
 \label{eq:main-five-thirds}
\end{equation}
The same guarantee holds with \(\Clot\) replaced by \(\Csh\). If \(\lambda,h_1,h_2,p>0\), there is a finite deterministic classical policy with \(\Clot\le(5/3)L_{\rm two}\). At other nonnegative boundaries, finite deterministic classical policies approach the stated bound arbitrarily closely.
\end{theorem}

The theorem permits any imbalance between the two setup charges, lead times, and holding rates. It also retains exact integer lot sizes. The upper factor $5/3$ is a guarantee on the best classical policy in the stated system. It is not a claim that every choice of the four controls performs well. The construction below provides a feasible choice satisfying the guarantee.

The difficult economic tradeoff is familiar: reducing the frequency of shipments saves setups but adds inventory. The additional challenge in a serial system is that the two stages may prefer different frequencies. Treating their costs separately can charge too much for the same supply shortage. Our argument keeps both stages in one inventory calculation while changing their cycle sizes. Remainder groups then convert this calculation into coordinated lots.

\subsection{Choosing Cycles Before Coordinating Lot Sizes}

Consider an auxiliary calculation with independent positive integer cycle sizes \(m_1\) downstream and \(m_2\) upstream, without requiring an integer ratio. Let \(V_{m_1,m_2}\) be their joint inventory cost, with the downstream position constrained by available upstream material. The two cycle positions are independent uniforms on \(\{0,\ldots,m_1-1\}\) and \(\{0,\ldots,m_2-1\}\). In this static calculation, the availability limit depends on upstream lead-time demand and the upstream cycle position. The downstream position can depend on these quantities, while downstream lead-time demand remains unobserved. Appendix~\ref{app:joint-cycles} defines the calculation and shows that the best downstream choice is the smaller of a fixed target and the resulting upper limit on the downstream inventory position.

The auxiliary calculation minimizes the combined setup and inventory cost over both cycle sizes:
\begin{equation}
 \mathcal J(k_1,k_2)=\inf_{(m_1,m_2)\in\mathbb N^2}
 \{k_1/m_1+k_2/m_2+V_{m_1,m_2}\}.
 \label{eq:joint-main}
\end{equation}
This quantity keeps the two inventory costs together while temporarily allowing independent cycle sizes. Two properties connect it to actual policies and to the lower bounds:
\begin{equation}
 \inf_{\Pi\in\mathcal P_{\RnQ}}\Clot(\Pi)\le\mathcal J(k_1,k_2)
 \le B_1(k_1)+B_2(k_2),\qquad V_0\le\mathcal J(k_1,k_2).
 \label{eq:joint-sandwich-main}
\end{equation}
The first comparison is particularly useful. An auxiliary choice of independent sizes can be converted to a feasible classical policy without increasing its cost. This is possible because the policy can choose its reorder-point difference as well as its lot sizes.

To see the conversion, fix the downstream size \(m_1\). Divide the \(m_2\) upstream positions into groups according to their remainders after division by \(m_1\). For example, \(m_1=3,m_2=8\) produce the groups \(\{0,3,6\}\), \(\{1,4,7\}\), and \(\{2,5\}\), with sizes \(3,3,2\) satisfying \(m_2=3+3+2=8\). Figure~\ref{fig:cycle-groups}(a) illustrates these remainder groups and the resulting policies; Appendix~\ref{app:joint-policy} develops the example in detail. These groups partition the average inventory cost in the auxiliary calculation. For each group, we consider the average inventory cost conditional on the upstream cycle position belonging to that group.

\begin{figure}[!htbp]
\centering
\begingroup\SingleSpaced
\definecolor{GroupingInk}{HTML}{202831}
\definecolor{GroupingGray}{HTML}{525E68}
\definecolor{GroupingLine}{HTML}{D1D6DB}
\definecolor{GroupingBlue}{HTML}{0072B2}
\definecolor{GroupingOrange}{HTML}{C75600}
\definecolor{GroupingGreen}{HTML}{00856A}
\definecolor{GroupingPaleBlue}{HTML}{EAF4FA}
\definecolor{GroupingPaleOrange}{HTML}{FFF1E8}
\definecolor{GroupingPaleGreen}{HTML}{E8F5F0}
\definecolor{GroupingPanel}{HTML}{F4F6F8}
\resizebox{.81\linewidth}{!}{%
\begin{tikzpicture}[x=1bp,y=1bp,every node/.style={inner sep=0pt,outer sep=0pt,anchor=base west,text=GroupingInk,font=\fontsize{8.5}{10}\selectfont},group dot/.style={circle,anchor=center,minimum size=14.4bp,line width=.75bp,font=\fontsize{8}{9}\selectfont},group arrow/.style={draw=GroupingGray,line width=.7bp,-{Stealth[length=4bp,width=3.6bp]}},group rule/.style={draw=GroupingLine,line width=.6bp}]
\path[use as bounding box] (0,0) rectangle (468,298);
\draw[group rule] (237,12) -- (237,292);

\node[font=\bfseries\fontsize{11}{13}\selectfont] at (1,284) {(a) Construct a classical policy};
\node[text=GroupingGray,font=\fontsize{9}{11}\selectfont] at (1,265) {Downstream \(m_1=3\); upstream \(m_2=8\)};
\foreach \groupindex/\groupcolor in {0/Blue,1/Orange,2/Green,3/Blue,4/Orange,5/Green,6/Blue,7/Orange}{
  \node[group dot,draw=Grouping\groupcolor,fill=GroupingPale\groupcolor,text=Grouping\groupcolor] at ({10+29*\groupindex},243) {\(\groupindex\)};
}
\node[font=\fontsize{8.7}{10}\selectfont] at (1,221) {Group by remainder after division by \(3\).};
\draw[group rule] (1,211) -- (225,211);
\node[anchor=base,font=\bfseries\fontsize{8.5}{10}\selectfont] at (39,194) {Group positions};
\node[anchor=base,font=\bfseries\fontsize{8.5}{10}\selectfont] at (146,194) {Lot sizes \((Q_1,Q_2)\)};
\node[anchor=base,font=\bfseries\fontsize{8.5}{10}\selectfont] at (211,194) {Weight};
\foreach \groupy/\groupcolor/\groupstart/\groupsize/\grouplot in {171/Blue/0/3/9,139/Orange/1/3/9,107/Green/2/2/6}{
  \pgfmathtruncatemacro{\grouplast}{\groupsize-1}
  \foreach \groupoffset in {0,...,\grouplast}{
    \pgfmathtruncatemacro{\groupposition}{\groupstart+3*\groupoffset}
    \node[group dot,draw=Grouping\groupcolor,fill=GroupingPale\groupcolor,text=Grouping\groupcolor] at ({12+26*\groupoffset},\groupy) {\(\groupposition\)};
  }
  \draw[group arrow] (83,\groupy) -- (107,\groupy);
  \fill[GroupingPale\groupcolor,rounded corners=3bp] (120,{\groupy-11}) rectangle (172,{\groupy+11});
  \node[anchor=center,text=Grouping\groupcolor,font=\fontsize{10}{12}\selectfont] at (146,\groupy) {\((3,\grouplot)\)};
  \node[anchor=center,font=\fontsize{9.5}{11}\selectfont] at (211,\groupy) {\(\groupsize/8\)};
}
\node[text=GroupingGray,font=\fontsize{8.3}{10}\selectfont] at (1,82) {Each group sets its reorder-point difference.};
\fill[GroupingPanel,rounded corners=4bp] (1,13) rectangle (225,64);
\node[font=\bfseries\fontsize{8.8}{11}\selectfont] at (9,45) {Choose the cheapest candidate policy.};
\node at (9,28) {Weighted average cost does not increase.};

\node[font=\bfseries\fontsize{11}{13}\selectfont] at (249,284) {(b) Compare auxiliary costs};
\node[text=GroupingGray] at (249,269) {Separate even and odd positions.};
\node[text=GroupingGray] at (249,256) {Remove the offset; halve the spacing.};
\node[font=\bfseries\fontsize{9}{11}\selectfont] at (249,243) {Original downstream size \(m_1=3\) (odd)};
\node[font=\bfseries\fontsize{9}{11}\selectfont] at (249,123) {Original upstream size \(m_2=8\) (even)};
\foreach \groupy/\groupcolor/\groupparity/\groupsize/\groupcycle/\groupstage in {220/Blue/0/2/3/1,181/Orange/1/1/3/1,103/Blue/0/4/8/2,64/Orange/1/4/8/2}{
  \ifnum\groupparity=0\def\groupname{Even}\else\def\groupname{Odd}\fi
  \node[anchor=west,align=left,text=Grouping\groupcolor,font=\fontsize{8}{9}\selectfont] at (249,\groupy) {\groupname\\positions};
  \pgfmathtruncatemacro{\grouplast}{\groupsize-1}
  \foreach \groupoffset in {0,...,\grouplast}{
    \pgfmathtruncatemacro{\groupposition}{\groupparity+2*\groupoffset}
    \node[group dot,minimum size=13.6bp,draw=Grouping\groupcolor,fill=GroupingPale\groupcolor,text=Grouping\groupcolor] at ({291+26*\groupoffset},\groupy) {\(\groupposition\)};
    \node[group dot,minimum size=12.4bp,draw=Grouping\groupcolor,fill=GroupingPale\groupcolor,text=Grouping\groupcolor] at ({400+13*\groupoffset},\groupy) {\(\groupoffset\)};
  }
  \draw[group arrow] (380,\groupy) -- (390,\groupy);
  \node[font=\fontsize{8.8}{10}\selectfont] at (291,{\groupy-17}) {\(g=\groupsize:\quad (\groupsize/\groupcycle)(\ifnum\groupsize=1 k_{\groupstage}\else k_{\groupstage}/\groupsize\fi)=k_{\groupstage}/\groupcycle\)};
}
\fill[GroupingPanel,rounded corners=3bp] (249,135) rectangle (467,156);
\node[font=\fontsize{8.7}{10}\selectfont] at (257,142) {Weighted setup: \(k_1/3+k_1/3=2k_1/3\)};
\fill[GroupingPanel,rounded corners=3bp] (249,13) rectangle (467,34);
\node[font=\fontsize{8.7}{10}\selectfont] at (257,20) {Weighted setup: \(k_2/8+k_2/8=k_2/4\)};
\end{tikzpicture}}
\endgroup
\caption{Two uses of the auxiliary cycles \(m_1=3,m_2=8\): (a) constructing a classical policy from remainder groups; (b) comparing auxiliary costs by separating even and odd positions.}
\label{fig:cycle-groups}
\end{figure}
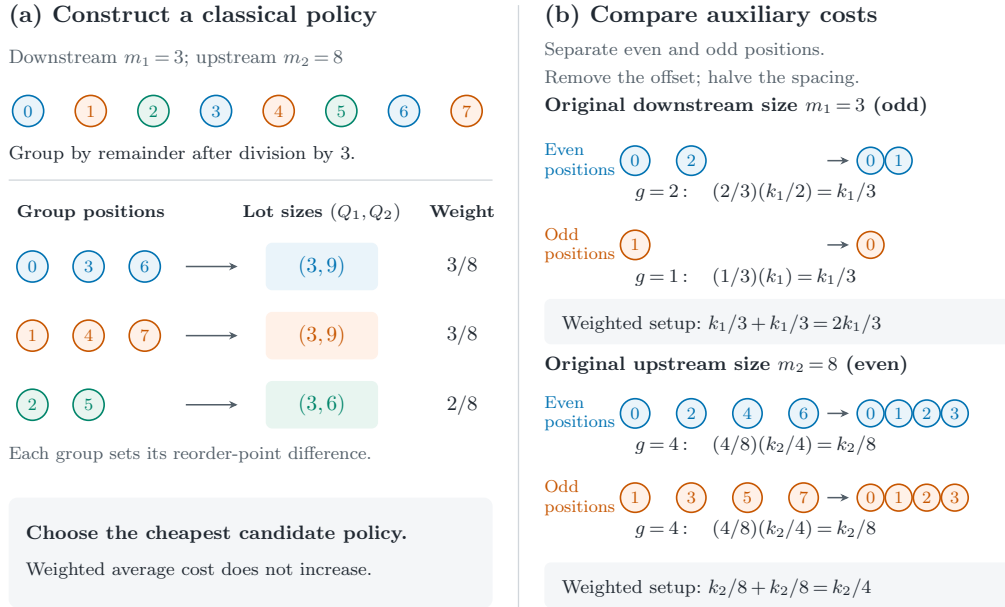

Within each group, consecutive positions differ by exactly \(m_1\). Let \(g\) be the number of positions in a group. This group, therefore, gives a feasible upstream lot of \(gm_1\) units and a downstream lot of \(m_1\) units, corresponding to \(Q_1=m_1\) and policy multiplier \(n=g\). We choose each policy's reorder-point difference so that its average inventory cost matches the conditional average for its group in the auxiliary calculation. Weight the resulting policies by their group sizes. Their weighted average inventory cost equals \(V_{m_1,m_2}\), the original joint inventory cost. The policy associated with a group of \(g\) positions orders upstream lots of \(gm_1\) units, giving an upstream setup cost rate of \(k_2/(gm_1)\). Its weight \(g/m_2\), therefore, gives a contribution of \(k_2/(m_1m_2)\) to the weighted average. With at most \(m_1\) nonempty groups, the weighted average upstream setup cost is at most \(m_1\cdot k_2/(m_1m_2)=k_2/m_2\). Every constructed policy retains the downstream lot size \(m_1\), so its downstream setup cost remains \(k_1/m_1\). Adding the inventory and both setup costs gives a weighted average total cost of at most \(k_1/m_1+k_2/m_2+V_{m_1,m_2}\), the cost of the auxiliary calculation for these cycle sizes. The cheapest classical policy constructed from the remainder groups costs no more than this weighted average. In Figure~\ref{fig:cycle-groups}(a), its lot sizes are, therefore, \((Q_1,Q_2)=(3,9)\) or \((3,6)\). The two candidates with lot sizes \((3,9)\) have different reorder-point differences and can have different costs.

\FloatBarrier

\subsection{Comparing Small and Large Cycles Together}

Optimizing the auxiliary cycle sizes gives \(\mathcal J(k_1,k_2)\) as an upper bound on the best classical policy cost, i.e., the first comparison in \eqref{eq:joint-sandwich-main}. We now compare this auxiliary value \(\mathcal J(k_1,k_2)\) across setup scales to connect it to \(L_{\rm two}\). Figure~\ref{fig:cycle-groups}(b) starts again from the original auxiliary sizes \(m_1=3,m_2=8\), forming smaller auxiliary pairs \((2,4)\) and \((1,4)\). Appendix~\ref{app:joint-doubling} develops this comparison in detail, including the case of a cycle of size one.

For a single realized inventory shortage or surplus, the holding-and-backlog loss is convex: the cost of an intermediate inventory position cannot exceed the corresponding average cost of two positions. If each downstream position is no greater than its corresponding upstream availability limit, averaging both sides preserves this inequality. These facts imply the convexity of the joint inventory value when the two cycle positions are scaled together by a common nonnegative factor. Halving both cycle positions, therefore, gives a joint inventory cost at most the average of the original cost and \(V_0\). Rearranging gives the inventory comparison underlying
\begin{equation}
 \mathcal J(4k_1,4k_2)\ge2\mathcal J(k_1,k_2)-V_0.
 \label{eq:joint-four-main}
\end{equation}

We explain \eqref{eq:joint-four-main} in the following three paragraphs. To apply this comparison with integer cycle sizes, separate each cycle's even-numbered and odd-numbered positions into two sets. Within each set, subtract the first position and divide by two to obtain consecutive integer positions. Apply this separation at both stages and weight each set by its share of the original positions. Figure~\ref{fig:cycle-groups}(b) uses the auxiliary sizes \(m_1=3,m_2=8\). The downstream size \(m_1=3\) illustrates the odd case: the even positions \(\{0,2\}\) give a set of size two, and the odd position \(\{1\}\) gives a set of size one. The resulting smaller cycles have setup costs \(k_1/2\) and \(k_1\), with weights \(2/3\) and \(1/3\). Each set contributes \(k_1/3\), so the weighted average is \((2/3)(k_1/2)+(1/3)k_1=2k_1/3\). The upstream size \(m_2=8\) illustrates the even case: the sets \(\{0,2,4,6\}\) and \(\{1,3,5,7\}\) both have size four. Each has weight \(4/8\) and setup cost \(k_2/4\), giving a weighted average of \((4/8)(k_2/4)+(4/8)(k_2/4)=k_2/4\).

For either stage \(i\in\{1,2\}\), the original cycle size \(m_i\) counts positions before separation, whereas \(g\) counts the positions in one set and becomes the smaller cycle size. When \(m_i\ge2\), both the even and odd sets are nonempty. An odd size \(m_i\ge3\) gives unequal set sizes, as in the downstream example; an even size gives equal set sizes, as in the upstream example. A set of size \(g\) has weight \(g/m_i\) and setup cost \(k_i/g\), so each set contributes \((g/m_i)(k_i/g)=k_i/m_i\). Adding the two contributions gives a weighted average setup cost of \(k_i/m_i+k_i/m_i=2k_i/m_i\). An original cycle of size \(m_i=1\) forms one set with weight one and unchanged setup cost.

The convexity of the joint inventory value gives an original joint inventory cost of at least twice the weighted average smaller-cycle inventory cost minus \(V_0\). To compare total costs, we must, therefore, also count the smaller cycles' setup costs twice, requiring at most \(2\cdot(2k_i/m_i)=4k_i/m_i\) at stage \(i\). In the running example, twice the weighted setup total is \(2(2k_1/3+k_2/4)=4k_1/3+4k_2/8\), exactly the setup cost of the original sizes three and eight at coefficients \((4k_1,4k_2)\). This explains the quadrupled setup coefficients in \eqref{eq:joint-four-main}. All cycle sizes remain integer, and subtracting \(V_0\) avoids counting the underlying lead-time inventory cost twice.

For given downstream and upstream cycle sizes \((m_1,m_2)\), the cost \(k_1/m_1+k_2/m_2+V_{m_1,m_2}\) consists of linear setup terms and an inventory term independent of \((k_1,k_2)\). Taking the infimum over \((m_1,m_2)\) makes \(\mathcal J(k_1,k_2)\) concave in the setup coefficients. Since \(2=(2/3)1+(1/3)4\), the concavity of \(\mathcal J\) gives the first inequality below, and \eqref{eq:joint-four-main} gives the second:
\begin{equation}
 \mathcal J(2k_1,2k_2)\ge\frac23\mathcal J(k_1,k_2)+\frac13\mathcal J(4k_1,4k_2)\ge\frac43\mathcal J(k_1,k_2)-\frac13V_0.
 \label{eq:joint-two-main}
\end{equation}
This step turns a comparison of cycle sizes into a comparison of setup costs. It applies after optimizing both sizes, so the optimizer can change when setup charges change.

\subsection{The Cancellation That Gives the Constant $3/5$}

The first lower bound in \eqref{eq:two-scale-lower}, together with \eqref{eq:joint-sandwich-main} and \eqref{eq:joint-two-main}, gives
\begin{equation}
 L_{\rm two}\ge\tfrac12\mathcal J(2k_1,2k_2)
 \ge\tfrac23\mathcal J(k_1,k_2)-\tfrac16V_0.
 \label{eq:comparison-low-zero-main}
\end{equation}
The second lower bound in \eqref{eq:two-scale-lower}, using \eqref{eq:joint-four-main}, gives
\begin{equation}
 L_{\rm two}\ge\tfrac12V_0+\tfrac14\mathcal J(4k_1,4k_2)
 \ge\tfrac12\mathcal J(k_1,k_2)+\tfrac14V_0.
 \label{eq:comparison-high-zero-main}
\end{equation}
The comparison in \eqref{eq:comparison-low-zero-main} is stronger when the cost of operating without setups is relatively small. The comparison in \eqref{eq:comparison-high-zero-main} is stronger when that cost is relatively large. Take \(3/5\) of \eqref{eq:comparison-low-zero-main} and \(2/5\) of \eqref{eq:comparison-high-zero-main}. The two coefficients of \(V_0\) cancel exactly, leaving
\begin{equation}
 L_{\rm two}\ge\frac35\mathcal J(k_1,k_2).
 \label{eq:cancellation-main}
\end{equation}
Combine this with the remainder-group policy construction to prove \eqref{eq:main-five-thirds}. Appendix~\ref{app:joint-boundary} proves finite attainment in the positive-cost case and provides limiting arguments for zero-cost \mbox{boundaries}.

The cancellation leading to \eqref{eq:cancellation-main} explains why we retain the system's zero-setup cost. That cost can help or weaken a particular comparison, depending on the price scale being used. Keeping it in both comparisons allows its contribution to cancel. Bounding each stage's cost separately by its worst-case factor can hide the shared \(V_0\) term and the cancellation it permits. The $5/3$ constant follows from using the same joint inventory calculation in both comparisons.

\subsection{Interpreting the Two Cost Bounds}

Figure~\ref{fig:proof-comparison} displays the calculation when \(J=\mathcal J(k_1,k_2)>0\). Let \(v=V_0/J\), the share of the joint cost that remains even if replenishment setups are free. Since \(0\le V_0\le J\), this share lies between zero and one. The first comparison guarantees \(L_{\rm two}/J\ge2/3-v/6\); the second guarantees \(L_{\rm two}/J\ge1/2+v/4\). Their maximum is at least \(3/5\) throughout, with equality at \(v=2/5\).

\begin{figure}[t]
\centering
\includegraphics[width=.576\linewidth]{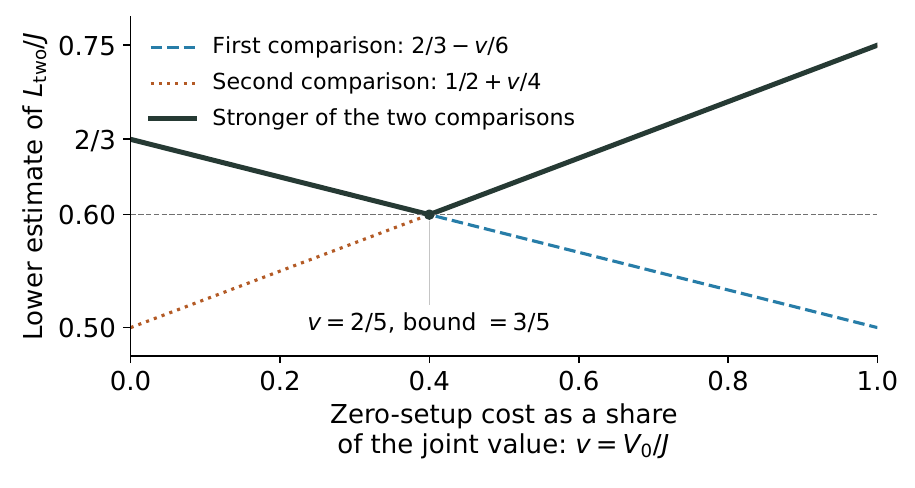}
\caption{The two comparisons cover different shares of unavoidable inventory cost. Their larger value is always at least \(3/5\). The curves plot the two proved inequalities.}
\label{fig:proof-comparison}
\end{figure}

The picture explains why retaining the zero-setup cost is useful even though it disappears from the final constant. When \(v=V_0/J\) is small, subtracting \(v/6\) leaves the first bound \(2/3-v/6\) close to \(2/3\). As \(v\) increases, the positive term \(v/4\) raises the second bound \(1/2+v/4\). At the crossing, either argument provides exactly the same estimate. The weighted cancellation is an economical way to cover all possible shares without identifying the system type in advance. The curves describe these algebraic bounds without indicating how often either is stronger in practice.

The fractions in \eqref{eq:two-scale-lower} follow from the chosen cost splits. Equal allocation, \(\eta=\zeta=1/2\), gives the factor \(1/2\). Taking \(u=1/2\) also reserves half the inventory cost for \(V_0\), leaving the factor \((1/2)(1/2)=1/4\). The choices \(u=1\) and \(u=1/2\) match the setup scales two and four in the cycle-size comparison. Optimizing \(u,\eta,\zeta\) in \eqref{eq:sp-bound-main} may strengthen the bound for a particular instance. Both linear estimates equal \(3/5\) at \(v=2/5\); changing their final mixing weights cannot improve the uniform estimate. The weights \(3/5\) and \(2/5\) attain that estimate. One route to a factor below \(5/3\) is a sharper comparison between \(\LBsp\) in \eqref{eq:sp-bound-main} and \(\mathcal J(k_1,k_2)\), raising the guaranteed share \(3/5\) by a fixed positive amount across all instances. This calls for a better estimate near \(v=2/5\), where the two linear bounds in \eqref{eq:comparison-low-zero-main} and \eqref{eq:comparison-high-zero-main} meet. Our analysis does not determine the best uniform factor obtainable from the full family in \eqref{eq:sp-bound-main}.

The crossing also shows the distinction between a proof's least favorable estimate and a policy class's worst performance. To attain the factor \(5/3\), a real instance would have to make the conversion from auxiliary cycles to a classical policy, the cycle scaling, the allocation lower bound, and the final crossing sufficiently tight at the same time. The argument establishes their inequalities individually. It does not exhibit an instance attaining all of them. The impossibility construction establishes an actual loss from the classical restriction, whose factor falls below \(5/3\).

This distinction has a practical counterpart. When a lower-bound calculation gives a loose comparison, it can be tempting to interpret its largest difference as the value of a more flexible replenishment policy. A policy improvement requires an actual feasible operating rule with a lower evaluated cost. Section~\ref{sec:impossibility} provides such a rule for a particular family, while Section~\ref{sec:numerics} evaluates selected alternatives in other environments. Keeping the two types of evidence separate prevents a conservative guarantee from being read as an expected operating loss.

\subsection{A Guaranteed Policy-Selection Procedure}

With \(\lambda,h_1,h_2,p>0\), the argument gives the following procedure for selecting a classical policy.
\begin{enumerate}[1.]
\item Find a global minimum of \eqref{eq:joint-main} over integer cycle sizes and integer inventory placements. Retain the sizes \(m_1,m_2\), the downstream target \(r\), and the upstream endpoint \(z\). Appendix~\ref{app:joint-boundary} gives finite search bounds and a quantile calculation for \(r\).
\item Partition the upstream positions \(0,\ldots,m_2-1\) by their remainders after division by \(m_1\). For each nonempty group with remainder \(b\) and \(g_b\) positions, construct the controls
\[
 (R_1,R_2,Q_1,Q_2)=(r-m_1,\ z-b-g_bm_1,\ m_1,\ g_bm_1).
\]
The policy multiplier is \(n=g_b\).
\item Evaluate each candidate's complete lot-accounted cost and select the cheapest, denoted by \(\widehat\Pi\). Implement its fixed lot sizes and reorder points.
\end{enumerate}
The grouping comparison and \eqref{eq:cancellation-main} give \(\Clot(\widehat\Pi)\le\mathcal J(k_1,k_2)\le(5/3)L_{\rm two}\le(5/3)\OPTsh\). Since \(\Csh(\widehat\Pi)\le\Clot(\widehat\Pi)\), the selected policy also satisfies the shipment-cost guarantee. At zero-cost boundaries, Appendix~\ref{app:joint-boundary} provides finite policies within any positive additive tolerance of the bound.

The procedure checks finitely many candidates, although the search can be large when holding or backlog costs are small. We recommend using the resulting policy as a starting candidate and retaining any cheaper classical policy found under the same setup-accounting convention, preserving the \(5/3\) guarantee. Section~\ref{sec:numerics} evaluates a separate direct search over classical controls.
\section{A Limit to Uniform Improvement}\label{sec:impossibility}

The uniform guarantee raises a natural question: how close to one could the approximation factor become if the policy parameters were chosen perfectly? The classical \((R,nQ)\) class itself imposes a positive gap. We construct finite Poisson systems in which the best policy costs more than $1.36$ times the unrestricted shipment optimum. Along a sequence of such systems, the proved lower bound on this cost ratio approaches $(1+\sqrt{3})/2=1.3660254\ldots$. The comparison uses an explicit feasible competing policy and the infimum over every classical integer lot size and reorder point. It, therefore, establishes a limitation of the policy class under shipment accounting, independently of the analytic lower bound used to prove the upper guarantee $5/3$.

The mechanism comes from a separation between upstream replenishment and downstream transfers. Upstream orders are free, transportation downstream is immediate, and storing a unit upstream or downstream has the same cost. A flexible policy can replace demand one unit at a time at the external supplier while accumulating those units upstream for a later transfer. Its total stock remains close to a base-stock target, even though downstream transfers occur infrequently. The classical rule links the two lot sizes through $Q_2=nQ_1$. In the instances below, infrequent transfers consequently require a substantial upstream lot, which adds inventory as the upstream order is consumed. Optimizing the four classical controls cannot remove this additional cycle inventory.

\begin{theorem}[{\sc A Lower Limit for the Classical Policy Class}]\label{thm:impossibility}
Let $\rho_\circ=(1+\sqrt{3})/2$. For every $r<\rho_\circ$, there is a finite instance of the model such that
\begin{equation}
 \inf_{\Pi\in\mathcal P_{\RnQ}}\Csh(\Pi)
 >r\,\OPTsh.
 \label{eq:impossibility-ratio}
\end{equation}
The same conclusion holds with $\Clot$ in the numerator. The instances can be chosen with all cost coefficients, the demand rate, and both lead times strictly positive. Thus, no uniform factor of $1.366$ or below holds against $\OPTsh$ under either numerator convention.
\end{theorem}

Consider $h_1=K_2=L_1=0$, unit demand and backlog rates, $L_2>0$, and $h_2>0$. Write $k=K_1$ and $D_2\sim\operatorname{Poisson}(L_2)$, and define
\begin{equation}
 F_{\mathrm{imp}}(z)=\mathbb E\{h_2(z-D_2)^+ +(z-D_2)^-\},\qquad
 J_{\mathrm{imp}}(k)=\inf_{\substack{m\in\mathbb N\\ \sigma \in\mathbb Z}}
 \left\{\frac{k}{m}+\frac1m\sum_{j=0}^{m-1}F_{\mathrm{imp}}(\sigma +j)\right\}.
 \label{eq:impossibility-scalar}
\end{equation}
Here $\sigma $ is the lowest inventory position in a cycle. On this boundary of the parameter space, $J_{\mathrm{imp}}(k)$ equals the optimized classical cost under both accounting conventions. To see why both setup-accounting conventions give the same answer, observe that a classical downstream dispatch contains at most $Q_2$ units. At an upstream receipt, any previously unfilled request implies that less than one $Q_1$-lot was available immediately before that receipt; the arriving $Q_2=nQ_1$ units can, therefore, release at most $n$ complete lots. At a demand epoch, at most one new $Q_1$-request is created. Rate balance then forces an actual shipment setup cost of at least $k/Q_2$. The upstream inventory position cycles over $Q_2$ positions; subtracting the independent lead-time demand gives the inventory part of \eqref{eq:impossibility-scalar}. A common-lot policy attains this cycle cost under both conventions.

The competing policy chooses an integer $S$ and places one upstream order at every demand. Let $N$ count demands and set $D(a,b]=N(b)-N(a)$. Its total net echelon inventory is $Y(t)=S-D(t-L_2,t]$. It holds received units upstream while downstream net inventory is nonnegative. When a demand creates downstream backlog, it dispatches all available upstream stock; when a receipt finds downstream backlog, it dispatches that unit immediately. Every decision uses observed demand and available material. Because stock is transferred whenever it could eliminate backlog, downstream backlog equals $Y(t)^-$. Its inventory cost $I$ and actual downstream shipment rate $\nu$ satisfy
\begin{equation}
 I=F_{\mathrm{imp}}(S),\qquad
 \frac1S\le\nu\le\frac{1+L_2}{S}
 +\left(1-\frac1S\right)\Pr(D_2\ge S).
 \label{eq:impossibility-policy-bounds}
\end{equation}
This feasible policy has total cost $I+k\nu$, so $\OPTsh\le I+k\nu$. The rate estimate includes receipt-triggered emergency shipments. We obtain this estimate by bounding selected shipment events by all demand and receipt events, without assuming that demand observed at a selected shipment epoch has an unconditional Poisson distribution.

Take $S\to\infty$ through integers, $L_2=1/S$, $h_2=L_2^S/S!$, and $k=\alpha h_2S^2$, where $0<\alpha\le1/2$. The Poisson probabilities imply $I/(h_2S)\to1$ and $S\nu\to1$. For the competing policy, the normalized setup cost is $k\nu/(h_2S)=\alpha S\nu\to\alpha$, so its normalized total cost tends to $1+\alpha$. For classical policies, balancing the leading setup and cycle-inventory terms $k/m+h_2m/2$ gives $\sqrt{2kh_2}=\sqrt{2\alpha}\,h_2S$. The bounds in \eqref{eq:imp-global-cycle-bounds} justify this balance over all integer lots and placements, giving $J_{\mathrm{imp}}(k)/(h_2S)\to1+\sqrt{2\alpha}$. Dividing both costs by $h_2S$, therefore, gives
\begin{equation}
 \frac{J_{\mathrm{imp}}(k)}{I+k\nu}
 =\frac{J_{\mathrm{imp}}(k)/(h_2S)}{I/(h_2S)+\alpha S\nu}
 \longrightarrow\frac{1+\sqrt{2\alpha}}{1+\alpha}.
 \label{eq:impossibility-limit}
\end{equation}
The maximum is $\rho_\circ$, attained at $\alpha=2-\sqrt{3}$. Establishing the best classical policy's limiting cost requires excluding very large lots and excessively low reorder positions. Appendix~\ref{app:impossibility} gives the global exclusions, admissibility proof, and extension to positive coefficients and lead times.

The conclusion can also be verified in an explicit finite example. Set $S=1000$, $L_2=1/1000$, $h_2=L_2^{1000}/1000!$, and $k=267949h_2$. Exact bounds give $J_{\mathrm{imp}}(k)/h_2\ge1729$ and $(I+k\nu)/h_2<1269$. Their ratio exceeds $1.36$. Every positive primitive in this boundary example is rational, and the comparison uses analytic Poisson bounds and integer arithmetic. Its extremely small holding rate is permitted by a uniform statement over the full primitive space. The example identifies an environment that a uniform guarantee must cover, but it does not describe the frequency or magnitude of losses in ordinary applications.

The result survives small positive changes to the zero coefficients. For a fixed finite instance with a strict gap, the competing policy can retain its dispatch decisions and allow each downstream shipment a small positive travel time. The resulting extra backlog cost is bounded by the quantity in that short pipeline. Positive downstream holding and upstream setup charges likewise add arbitrarily small amounts to its cost. The classical lower bound $J_{\mathrm{imp}}(k)$ defined in \eqref{eq:impossibility-scalar} continues to hold because, for fixed classical controls, the dispatch-size bound and upstream inventory cycle are unchanged from the case $h_1=K_2=L_1=0$. This establishes the claim for strictly positive primitives without interchanging optimization and a continuity limit.

For a manager, the relevant flexibility is the ability to replenish the system frequently while transferring accumulated stock downstream less frequently. This becomes valuable when upstream setup costs are small and the storage location has little effect on holding cost. The lower bound quantifies a limit to the benefit obtainable by retuning classical reorder points and lots alone. Its benchmark is specifically $\OPTsh$. The competing policy uses variable downstream shipment quantities, so the argument does not establish the same impossibility bound relative to the fixed-batch benchmark $\OPTlot$. Nor does the construction identify the exact worst-case factor of the classical class: it leaves a gap between the lower limit $\rho_\circ$ and the uniform upper guarantee $5/3$.
\section{Using the Result in an Operating Decision}\label{sec:managerial}

To decide whether to change replenishment practice, a manager needs the current policy's cost, a valid lower bound on the same system's optimal cost, and a description of the proposed operating change. The results support three steps: bound the savings available, identify the source of improvement, and account for implementation costs.

\vspace{-6pt}
\subsubsection*{A bound on the improvement still available.}

Suppose a manager has evaluated a feasible classical policy with cost rate \(C\) and a valid lower bound \(L>0\) on the same system's optimal shipment cost. Then \(L\le\OPTsh\le C\), and
\begin{equation}
 0\le\frac{C-\OPTsh}{C}\le1-\frac LC.
 \label{eq:manager-saving}
\end{equation}
The right side bounds the fraction of current cost that any replenishment policy could save. The evaluated policy need not have optimal classical controls. For example, if annual policy cost is 100 and a valid lower bound is 85, any policy change saves at most 15. A feasible alternative costing 92 realizes a saving of eight and leaves at most seven more available. These hypothetical values illustrate how the bound limits further savings.

For a policy satisfying the uniform factor \(5/3\), the optimal policy can save at most 40\% of its cost. The theorem constructs such a policy with positive holding and backlog costs and gives a limiting sequence at zero-cost boundaries. In the impossibility construction, the explicit flexible policy saves a fraction approaching \(2-\sqrt3\), about 26.8\%, of the best classical policy's cost. That example shows that substantial savings are possible even after all classical controls have been optimized; their size in a particular operation requires its own evaluation.

\vspace{-6pt}
\subsubsection*{Three changes that should be evaluated separately.}

Table~\ref{tab:managerial-changes} distinguishes three changes to current practice. Separating them helps identify whether an improvement comes from better controls, shared dispatch charges, or more flexible quantities.

\begin{table}[t]
\centering
\caption{Operating changes and the comparison needed to evaluate them.}
\label{tab:managerial-changes}
\begin{tabular}{>{\raggedright\arraybackslash}p{.24\linewidth}>{\raggedright\arraybackslash}p{.32\linewidth}>{\raggedright\arraybackslash}p{.34\linewidth}}
\toprule
Proposed change & What becomes adjustable & Appropriate evaluation\\
\midrule
Retune classical controls & Lot sizes, upstream multiplier, and both reorder points & Recompute the same exact policy cost under the current setup convention\\
Consolidate complete lots & Number of lot charges paid at a common dispatch & Evaluate shipment counts and reoptimize controls under shipment accounting\\
Allow variable deliveries & Quantity transferred at each observed state & Evaluate a feasible policy with the same material, information, and lead-time restrictions\\
\bottomrule
\end{tabular}
\end{table}

Retune lot sizes and both reorder points jointly. A new upstream multiplier changes when material becomes available, so keeping the old reorder points can overstate coordination costs. Section~\ref{sec:guarantee} explains how the policy construction adjusts these controls together.

For consolidation, the charged activity determines the appropriate comparison. A vehicle dispatch can carry several complete lots for one setup charge; an activity performed separately for each lot still incurs each lot's charge. Reoptimizing controls under shipment accounting may, therefore, lower cost within the classical class. Section~\ref{sec:numerics} illustrates this opportunity with unit downstream lots. With these lots, every positive integer shipment is a complete-lot shipment.

Variable deliveries can separate upstream replenishment and downstream transfer frequencies. The impossibility example demonstrates their value under the same shipment-accounting convention after optimizing all classical controls. The example is particularly relevant when upstream setups are inexpensive and upstream storage adds little cost. Evaluate any flexible rule under the same material and information restrictions as the current rule.

\vspace{-6pt}
\subsubsection*{Scope and practical use of the model.}

The comparison assumes stationary Poisson unit demand, deterministic lead times, full backlogging, and the modeled holding, backlog, and setup costs. Within these assumptions, all cost and lead-time ratios are allowed. Changing a supplier's lead time, adding transport capacity, or changing the service promise requires a new comparison: \eqref{eq:manager-saving} measures policy improvement within a fixed system.

The cost bound also gives a screen for implementation expense. Additional recurring administrative or implementation costs, measured on the same time basis as \(C\) and \(L\), cannot be justified by modeled operating savings if they exceed \(C-L\). When that margin is small, extensive policy redesign has limited value. When it is large, evaluating a feasible alternative or strengthening the lower bound can clarify whether the improvement opportunity warrants further effort.
\section{Numerical Evidence and Operating Implications}\label{sec:numerics}

The numerical study asks how the four controls respond to differences between stages, how much an instance-specific cost bound improves on the uniform guarantee, and whether shipment accounting changes the preferred classical policy. We evaluate classical policies under Poisson demand using the exact cost formula and compare their costs with a lower bound computed by searching over cost allocations and scales in \eqref{eq:sp-bound-main}. This numerical search explores the broader combined family; the \(5/3\) guarantee uses only the two simplified bounds defining \(L_{\rm two}\) in \eqref{eq:two-scale-lower}.

\subsection{Evaluation Method and Interpretation}

For each candidate downstream lot size \(Q\), multiplier \(n\), and integer difference \(\ell=R_2-R_1\), we form the effective demand in \eqref{eq:effective-demand}. An integer \(p/(h_1+h_2+p)\)-quantile chooses the best downstream endpoint \(s=R_1+Q\). The exact formula in Proposition~\ref{prop:rnq-account-main} then gives the lot-accounted policy cost. We retain the lowest evaluated cost in the finite search menu. We call this the best-found policy, because the search does not provide a global exclusion of every larger upstream multiplier.

For the denominator, we evaluate feasible members of \eqref{eq:sp-bound-main}, including the zero-setup value, a selected allocation and scale, and a candidate at \(u=1\). Each selected allocation is reevaluated using integer-cycle losses. The maximum of these evaluated candidates is denoted by \(\widehat L_{\rm SP}\). In exact arithmetic, \(\widehat L_{\rm SP}\le\LBsp\le\OPTsh\), even if the search does not find the best allocation or scale. A more extensive policy search can lower the numerator; a more extensive lower-bound search can raise the denominator. Both can improve the comparison.

The reported ratio is the best-found lot-accounted cost divided by \(\widehat L_{\rm SP}\). Its distance from one combines any remaining policy inefficiency with the weakness of the evaluated lower bound. A large ratio does not by itself identify which source is responsible. This differs from Theorem~\ref{thm:impossibility}, whose denominator is the cost of a feasible competing policy and, therefore, establishes an actual limitation of the classical class.

Expectations are evaluated in floating point after a Poisson upper-tail truncation with target probability \(10^{-13}\). The implementation places the remaining probability mass at the largest retained demand value and removes extremely small weights in the loss calculation. These choices support high-accuracy numerical evaluation, but they do not provide rigorous error bounds that account for numerical rounding. Accordingly, the reported ratios are numerical estimates. The uniform theorem and impossibility example are analytic results and do not depend on these computations. Appendix~\ref{app:computation} describes the search ranges, data checks, and limitations.

\subsection{Effects of Costs and Lead Times on Policy Controls}

Four examples in Table~\ref{tab:configurations}, drawn in Figure~\ref{fig:rnq-regimes}, show why all four controls matter. Common lots can accompany either equal or different reorder points. A larger upstream lot can be useful with either a positive or a negative reorder-point difference. The difference, therefore, governs relative replenishment timing as well as the inventory held to cover transportation. A rule that forces it to have one sign would exclude useful classical policies.

\begin{table}[t]
\centering
\caption{Representative best-found integer configurations illustrated in Figure~\ref{fig:rnq-regimes}, with \(\lambda=1\).}
\label{tab:configurations}
\small
\begin{tabular}{@{}lcccc@{}}
\toprule
 & A & B & C & D\\
\midrule
Lead times \((L_1,L_2)\) & \((1.5,.25)\) & \((1.5,4)\) & \((1,4)\) & \((5,.1)\)\\
Setups \((K_1,K_2)\) & \((8,.05)\) & \((8,.05)\) & \((12,30)\) & \((.1,11)\)\\
Rates \((h_1,h_2,p)\) & \((1,2,5)\) & \((1,.5,20)\) & \((1,.2,10)\) & \((1,1,2)\)\\
Reorder points \((R_1,R_2)\) & \((0,0)\) & \((2,8)\) & \((0,5)\) & \((5,2)\)\\
Lot sizes \((Q_1,Q_2)\) & \((4,4)\) & \((4,4)\) & \((6,18)\) & \((2,6)\)\\
Lot-accounted cost & 10.2942 & 9.4300 & 9.5230 & 11.3746\\
\bottomrule
\end{tabular}
\end{table}

Set \(\lambda=1\), \(p=10\), and fix total setup, lead time, and echelon holding rates at six, two, and two. With \(\rho_K=K_1/K_2\), \(\rho_L=L_1/L_2\), and \(\rho_h=h_1/h_2\), the parameters are
\begin{equation}
 (K_1,K_2)=\frac6{1+\rho_K}(\rho_K,1),\quad
 (L_1,L_2)=\frac2{1+\rho_L}(\rho_L,1),\quad
 (h_1,h_2)=\frac2{1+\rho_h}(\rho_h,1).
 \label{eq:grid-normalization}
\end{equation}
This design changes where costs and delay arise while preserving their total scale. The downstream quantile level remains \(5/6\). Thus, these grids study asymmetry at a fixed service-cost tradeoff; the subsequent stress test examines a much smaller quantile level.

\FloatBarrier

\begin{figure}[t]
\centering
\includegraphics[width=.98\linewidth]{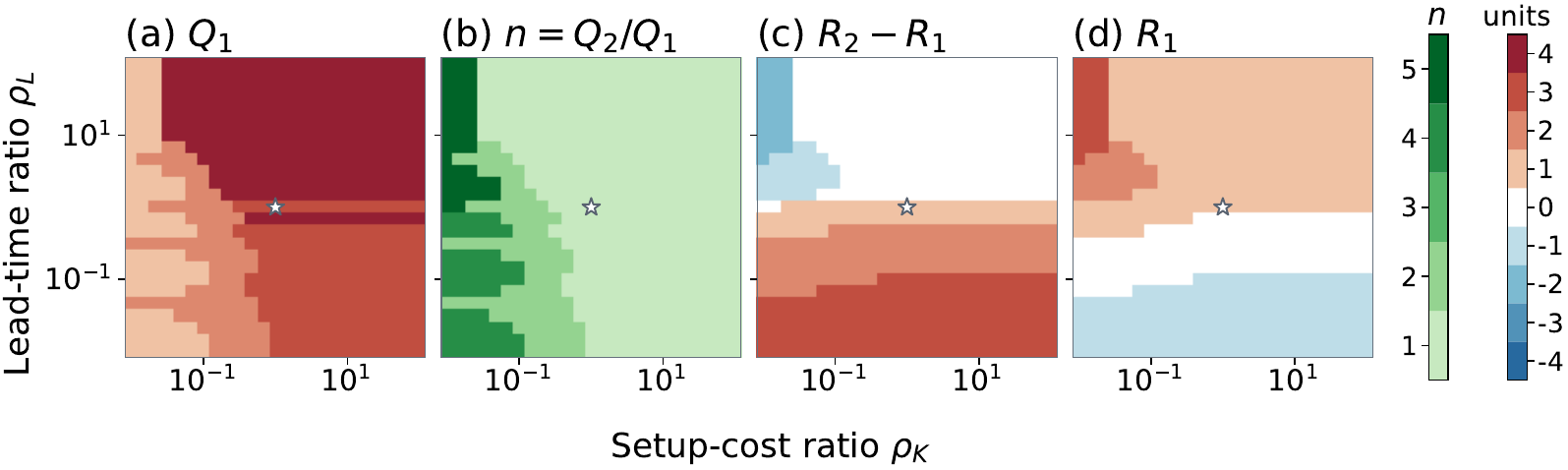}
\caption{Best-found lot sizes and reorder points as setup costs and lead times are redistributed between stages. The holding-rate ratio is one; the star marks equal division of all three parameter totals.}
\label{fig:rnq-response}
\end{figure}

Figure~\ref{fig:rnq-response} varies the setup and lead-time ratios over \([10^{-2},10^2]^2\), using 25 logarithmically spaced values per axis and \(\rho_h=1\). A small \(\rho_K\) makes downstream setups relatively inexpensive and tends to favor a smaller downstream lot and a larger upstream multiplier. A small \(\rho_L\) moves more transportation time upstream and tends to increase the reorder-point difference. The boundaries between selected controls are steps because lot sizes and reorder points are integers. At the symmetric setting marked with a star, the selected controls are \((R_1,R_2,Q_1,Q_2)=(1,2,3,3)\).

These patterns suggest different responses to changes in a supplier agreement. Reducing the upstream setup charge makes more frequent upstream orders attractive, but the appropriate downstream lot still depends on downstream shipment costs. Shortening upstream transportation time changes the inventory available to support downstream requests, so it can alter the reorder-point difference even when the preferred lot sizes remain fixed. The numerical patterns are conditional on the chosen normalization. They should not be interpreted as global monotonicity patterns.

\subsection{Comparisons with the Combined Lower Bound}

Figure~\ref{fig:sp-landscape} gives three two-dimensional views: setup versus lead-time ratios at \(\rho_h=1\), setup versus holding-rate ratios at \(\rho_L=1\), and lead-time versus holding-rate ratios at \(\rho_K=1\). The setup and lead-time ratios range from \(10^{-2}\) to \(10^2\); the holding-rate ratio ranges from \(10^{-3/2}\) to \(10^{3/2}\). Each panel contains \(25^2=625\) cells. The panels share intersections, leaving 1,801 distinct instances.

\begin{figure}[t]
\centering
\includegraphics[width=.98\linewidth]{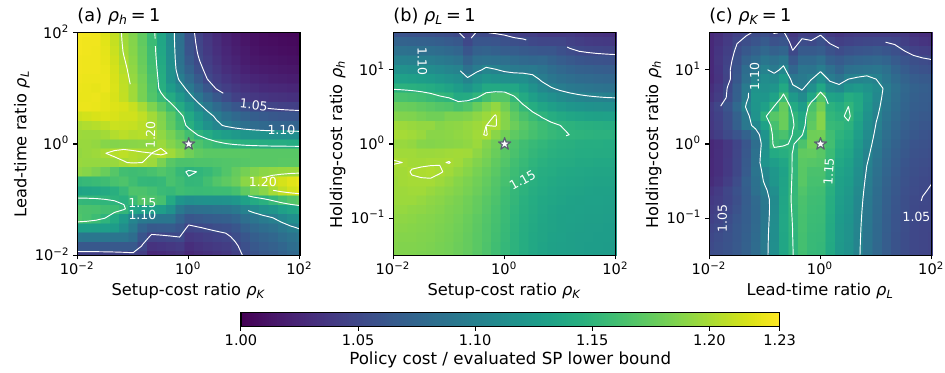}
\caption{Best-found lot-accounted cost divided by the evaluated combined lower bound. All three panels use the same color scale. The maximum over the 1,801 distinct instances is 1.2265.}
\label{fig:sp-landscape}
\end{figure}

The maximum computed ratio is \(1.2265\), the median is \(1.1208\), and the 90th percentile is \(1.1928\). The panel maxima are \(1.2265\), \(1.2021\), and \(1.1874\). The largest ratio occurs at \((\rho_K,\rho_L,\rho_h)=(.01,100,1)\), where the selected controls are \((3,1,1,5)\), their cost is \(8.9857\), and the evaluated lower bound is \(7.3261\). This point lies on the plotted boundary, so the reported maximum concerns the chosen grid. The median and percentile likewise summarize the experimental design rather than a probability distribution of business environments.

These comparisons are substantially below \(5/3\). They show that the evaluated bound can provide useful assurance beyond the uniform theorem on these grids. They also show why a manager should calculate an instance-specific comparison before attributing a large financial value to more complicated operating rules. At a ratio of \(1.12\), for example, an exact lower-bound evaluation would limit the best possible improvement to about \(10.7\%\) of the proposed policy's cost, because \(1-1/1.12\approx.107\). The ratio expresses the policy premium relative to the lower bound; the improvement percentage uses the current policy's cost as its denominator.

\subsection{An Asymmetric Stress Test and the Role of Shipment Accounting}

A separate grid places almost all holding cost downstream, makes backlog inexpensive, and makes upstream setups nearly free. Its fixed primitives are \(\lambda=h_2=1\) and
\[
 (L_2,K_1,h_1,p)=(58.3717,.501859,2.32828\times10^6,1.12801\times10^{-3}).
\]
The grid varies \(K_2/K_1\) from \(6.80\times10^{-7}\) to \(2.71\times10^{-5}\) and \(L_1/L_2\) from \(1.98\times10^{-6}\) to \(7.90\times10^{-5}\), again using 25 points per axis. The selected lot-accounted policy is \((-30,6,30,30)\) throughout. Ratios to the evaluated combined lower bound range from \(1.5445\) to \(1.6004\), with median \(1.5912\).

At the stress-grid instance with the largest ratio, \(K_2=3.41325\times10^{-7}\) and \(L_1=7.30770\times10^{-4}\). The common-lot policy costs \(0.060576\), while the evaluated lower bound is \(0.037852\). Because its multiplier is one, each downstream dispatch contains one lot, and the two setup conventions give this particular policy the same cost. Shipment accounting nevertheless permits a different classical policy to perform better: direct evaluation of the stationary cost formula over the unit-lot menu recorded in Appendix~\ref{app:computation} selects \((-1,9,1,33)\), with cost \(0.055303\), about \(8.7\%\) less, and ratio \(1.4610\) to the same bound.

This comparison involves different controls optimized under different accounting conventions. It illustrates how complete unit lots can be accumulated into a physical shipment that incurs one setup charge. A manager who treats the charge as payable per unit lot would avoid this choice, even when the actual activity is a single dispatch. Correctly identifying the charged activity can, therefore, change both the cost estimate and the recommended operating rule.

The existing simulation study also searches policies that may transfer fewer than a prescribed downstream lot when a complete lot is unavailable, following \citet{HuYang2014}. At the same instance, a separate menu of 1,785 control vectors with downstream lot sizes of at least two selects a candidate with independently estimated cost \(0.055608\). This exceeds the evaluated unit-lot classical cost by about \(0.6\%\). A separate search over classical policies with downstream lot sizes from one to eight selects a policy with one-unit downstream lots (\(Q_1=1\)). These finite searches provide no evidence of an additional gain from incomplete-lot transfers in this experiment. Policies outside these searches may achieve further savings.

The high ratios in the stress grid also do not show how much of the gap comes from a weak lower bound and how much comes from policy inefficiency. The stress experiment, therefore, complements the analytical impossibility example. The latter proves that flexible shipment quantities can yield a substantial improvement in some environments. The numerical study shows that a particular expansion of the shipment rules need not produce such an improvement in every difficult-looking environment. Evaluating the cost convention, the controls, and the benchmark separately gives a more useful diagnosis than interpreting the policy-to-bound ratio alone.
\section{Concluding Remarks}\label{sec:conclusion}

The classical echelon stock \((R,nQ)\) policy coordinates two stages with four integer controls and complete-lot transfers. We prove that its optimized cost is at most \(5/3\) of the optimal admissible shipment cost throughout the nonnegative parameter space. We charge the classical policy a setup cost for each complete lot, even when several lots are shipped together, while the optimal admissible policy pays once per shipment. A simple standardized operating structure, therefore, retains a uniform guarantee even against policies that can change their shipment quantities over time.

The guarantee has a substantive limit. No factor below \((1+\sqrt3)/2\) can hold for the classical class against the shipment optimum. The construction separates inexpensive unit upstream replenishment from less frequent, variable downstream transfers. Linking the upstream lot to a fixed downstream lot introduces additional cycle inventory that reorder-point optimization cannot eliminate. The result holds for strictly positive finite primitives as well as the simpler boundary construction used to explain it. The exact worst-case factor remains between \(1.366\ldots\) and \(1.6667\ldots\).

For managers, the results support evaluating standardization and flexibility through the actual activities that create cost. Standardized lots simplify replenishment, and the theorem limits their worst possible loss in the stated model. Shipment consolidation can improve performance within the classical class when several lots share one dispatch charge. Further flexibility can be valuable when upstream replenishment and downstream delivery should operate at different frequencies. The numerical comparisons help assess these possibilities for particular parameter choices, while their remaining gaps require separate attention to policy search and lower-bound strength.

Extensions require additional analysis. More than two stages impose several integer-multiple relationships, so a conversion that works for one adjacent pair must also preserve the next pair's feasibility. Compound demand removes the consecutive unit-demand cycle structure used here. Renewal demand, which allows nonexponential times between successive arrivals, changes the independence argument over lead-time windows. Lost sales change both inventory balances and the information relevant to optimal decisions. These issues make the joint-cycle approach a starting point for broader models, while the present theorem concerns the two-stage Poisson system.

\label{end:manuscript}\clearpage
\renewcommand\normalsizeXIyesCE{\OriginalDoubleSize\baselineskip17.5pt}

\DoubleSpaced
\setcounter{page}{1}
\renewcommand\thesection{\Alph{section}}
\renewcommand\theequation{\Alph{section}.\arabic{equation}}
\renewcommand\thedefinition{\Alph{section}.\arabic{definition}}
\renewcommand\thelemma{\Alph{section}.\arabic{lemma}}
\renewcommand\theproposition{\Alph{section}.\arabic{proposition}}
\renewcommand\thetheorem{\Alph{section}.\arabic{theorem}}
\renewcommand\theremark{\Alph{section}.\arabic{remark}}
\makeatletter
\@addtoreset{equation}{section}\@addtoreset{definition}{section}\@addtoreset{lemma}{section}\@addtoreset{proposition}{section}\@addtoreset{theorem}{section}\@addtoreset{remark}{section}
\makeatother
\setcounter{section}{0}\setcounter{equation}{0}\setcounter{lemma}{0}\setcounter{proposition}{0}\setcounter{theorem}{0}\setcounter{remark}{0}
\begin{center}
\textbf{\large Online Appendices to\\``A \(5/3\) Guarantee for Echelon Stock \((R,nQ)\) Policies\\in Two-Stage Stochastic Serial Systems"}
\end{center}
\noindent This electronic companion gives the complete arguments supporting the paper. We first establish material feasibility, exact policy costs, and the lower bounds. We then prove the joint-cycle comparison, including integer implementation and zero-cost cases, construct the impossibility example, and describe the numerical evaluation. We recall notation as needed.

\noindent The table summarizes the main cost quantities. Appendix~\ref{app:joint-cycles} uses \(m_1,m_2\) for the independent downstream and upstream cycle sizes and \(n=Q_2/Q_1\) for the classical multiplier.
\begin{center}
\small
\begin{tabular}{>{\raggedright\arraybackslash}p{.26\linewidth}>{\raggedright\arraybackslash}p{.65\linewidth}}
\toprule
Quantity & Meaning and where it is established\\
\midrule
\(k_i=\lambda K_i\) & Setup charge multiplied by the demand rate; a lot of size \(q\) pays rate \(k_i/q\).\\
\(H_0=h_2\lambda L_1\) & Upstream echelon holding for mean material in the downstream pipeline.\\
\(V_0\) & Optimal inventory cost with both setup charges zero; Appendix~\ref{app:zero-setup}.\\
\(\mathcal C_F(k)\) & Best average loss over consecutive integer positions plus setup cost; Appendix~\ref{app:lower-bounds}.\\
\(\mathcal A,\LBsp\) & Cost allocation and its combination with the zero-setup value; Appendices~\ref{app:lower-bounds} and~\ref{app:zero-setup}.\\
\(V_{m_1,m_2},\mathcal J\) & Joint inventory value and its optimized total cost; Appendix~\ref{app:joint-cycles}.\\
\(J_{\mathrm{imp}}(k)\) & Optimized classical cost on the impossibility family; Section~\ref{sec:impossibility} and Appendix~\ref{app:impossibility}.\\
\bottomrule
\end{tabular}
\end{center}

\section{Policy Costs and Admissibility}\label{app:policy-cost}

This appendix proves the physical-policy and lower-bound facts used in Sections~\ref{sec:policy} and~\ref{sec:benchmark}. Demand has rate \(\lambda>0\), lead times are finite and nonnegative, and all cost coefficients are finite and nonnegative. With zero demand, empty initial inventories and pipelines, and no orders, the cost is zero.

\subsection{The Admissible Comparison Class}\label{app:admissibility}

Let \(N\) count unit customer demands and let \(U_i\) count units physically dispatched into echelon \(i\). Internal requests enter \(U_1\) when material is dispatched. A policy uses only information available at the time of each decision and may use independent private randomization. Its cumulative dispatch quantities are nondecreasing integer processes, with only finitely many dispatches on each finite time interval. Compatible initial histories determine constants \(x_i\) and the pipelines; their initial inventory levels and pipelines have finite first moments. Fixed initial states and the stationary constructions below satisfy this convention. The balances and material constraint are
\begin{align}
 \mathit{IP}_i(t)&=x_i+U_i(t)-N(t),&
 \mathit{IL}_i(t)&=x_i+U_i(t-L_i)-N(t),\\
 P_i(t)&=U_i(t)-U_i(t-L_i)\ge0,&
 \mathit{IP}_1(t)&\le\mathit{IL}_2(t).
 \label{eq:app-physical-balances}
\end{align}
Thus, \(\mathit{IL}_i(t+L_i)=\mathit{IP}_i(t)-[N(t+L_i)-N(t)]\). Customer backlog is filled as soon as downstream stock is available, so \(I_1=\mathit{IL}_1^+\), \(\mathsf B=\mathit{IL}_1^-\), and \(I_2=\mathit{IL}_2+\mathsf B\). These conventions agree with Section~\ref{sec:model}.

The comparison class requires finite expected costs on finite horizons and the following stability conditions. Write \(U_0\) for customer deliveries, \(L_\Sigma=L_1+L_2\), and \(c_I=h_1I_1+h_2I_2+p\mathsf B\). As \(T\to\infty\),
\begin{gather}
 \E U_j(T)/T\longrightarrow\lambda\quad(j=0,1,2),\\
 T^{-1}\E\{|\mathit{IL}_i(T)|+P_i(T)\}\longrightarrow0\quad(i=1,2),\\
 T^{-1}\E\int_T^{T+L_\Sigma}c_I(t)\,dt\longrightarrow0,
 \qquad T^{-1}\E[U_i(T+L_\Sigma)-U_i(T)]\longrightarrow0.
 \label{eq:app-stability}
\end{gather}
Finally, consider each chronological inventory-position path, recording a unit demand decrease before any coincident dispatch increase. Apply chronological loop erasure as follows: whenever the path returns to a state already in the retained path, remove the segment from that earlier visit through the return and keep the repeated state once. Record each removed segment as a completed cycle. The retained path has distinct states. Let \(r_{i,T}\) count its remaining downward demand transitions through time \(T\). We require \(\E r_{i,T}=o(T)\), which we call the unfinished-path condition. Lemma~\ref{lem:lb-unit-stage} uses the same decomposition. These conditions concern the physical processes, so they are unchanged when fixed setup coefficients are rescaled. They also exclude permanent non-service when backlog is unpenalized.

The unfinished-path condition has a useful sufficient criterion. A loop-erased integer path has distinct states, so its remaining demand decreases are at most its state range. Thus, an expected range of \(o(T)\) suffices. In particular, every stationary inventory-position process \(X\) with \(\E|X(0)|<\infty\) satisfies this criterion. To see this, let \(D_j=N(j+1)-N(j)\). Since dispatches only raise the position, every state in the chronological path during \((j,j+1]\) lies between \(X(j)-D_j\) and \(X(j+1)+D_j\). For any identically distributed nonnegative integrable sequence \(Y_j\), the bound \(\max_{0\le j\le n}Y_j\le a+\sum_{0\le j\le n}Y_j\mathbf1_{\{Y_j>a\}}\) gives \(n^{-1}\E\max_{0\le j\le n}Y_j\to0\), by first sending \(n\) and then \(a\) to infinity. Apply this to \(|X(j)|\) and the Poisson counts \(D_j\). No independence between the sampled positions is required.

At coincident events, first receive shipments whose positive lead times expire, then process customer demand, and then perform the upstream-order and complete-lot matching decisions. A shipment with zero lead time is received at the decision epoch. Complete all consequences of that event before advancing time. Each finite-lot construction below has only finitely many such consequences at an epoch; simultaneous complete lots may share a physical dispatch.

\subsection{Stationary Complete-Lot Construction}\label{app:stationary-construction}

Fix \(Q,n\in\mathbb N\), \(\ell\in\mathbb Z\), and an integer endpoint \(s\), with downstream lot size \(Q_1=Q\) and upstream lot size \(Q_2=nQ\). Let \(N_0\) be the signed counter of a two-sided Poisson process, with \(N_0(0)=0\): for \(t<0\), \(N_0(t)\) is minus the number of demands in \((t,0]\). Choose \(N(0)\) independently and uniformly on \(\{0,\ldots,nQ-1\}\), and put \(N(t)=N_0(t)+N(0)\). Thus, increments of \(N\) count demand, and its initial value specifies the position within the replenishment cycle. Define
\begin{equation}
 A(t)=\left\lfloor\frac{N(t)}Q\right\rfloor,
 \quad O(t)=\left\lfloor\frac{N(t)+\ell+(n-1)Q}{nQ}\right\rfloor,
 \quad \mathsf S(t)=nO(t-L_2),
 \quad M(t)=\min\{A(t),\mathsf S(t)\}.
 \label{eq:app-counters}
\end{equation}
A sublot is one \(Q\)-unit piece of an upstream \(nQ\)-unit lot. The counters in \eqref{eq:app-counters} record downstream requests, upstream orders, received sublots, and matched sublots, respectively. Their integer origins include the compatible past history. Absolute counter values may, therefore, be negative; their nonnegative increments over a forward time interval count the corresponding requests or lots. Increments of \(QM\) are downstream dispatches; increments of \(nQO\) are supplier dispatches. The reorder points and lot sizes are \((R_1,R_2,Q_1,Q_2)=(s-Q,s-Q+\ell,Q,nQ)\).

The resulting state is
\begin{align}
 \mathit{IP}_2&=s-N+nQO,& \mathit{IL}_2&=s-N+Q\mathsf S,\\
 \mathit{IP}_1&=s-N+QM,& \mathit{IL}_1(t)&=s-N(t)+QM(t-L_1).
 \label{eq:app-counter-states}
\end{align}
All counters are nondecreasing. Pending requests and unused sublots are \((A-\mathsf S)^+\) and \((\mathsf S-A)^+\), respectively. Hence, matching uses available complete lots, \(\mathit{IP}_1\le\mathit{IL}_2\), and both pipelines are nonnegative. Also \(\mathit{IL}_1\le\mathit{IP}_1\le\mathit{IL}_2\), which implies nonnegative physical inventories. This is a causal implementation of the classical policy described by \citet{Chen}; the calculation below establishes its cost directly.

A state is settled after all ordering, matching, and associated zero-lead-time receipts at an epoch have been completed. A compatible initialization is a settled integer state satisfying \eqref{eq:app-physical-balances}, with an integer number \(\beta\ge0\) of pending downstream requests. The nominal downstream position adds the quantities of those pending requests to the actual position: \(\widehat{\mathit{IP}}_1=\mathit{IP}_1+Q\beta\in\{R_1+1,\ldots,R_1+Q\}\). The Stage-2 position satisfies \(\mathit{IP}_2\in\{R_2+1,\ldots,R_2+nQ\}\). The initial pipelines contain complete lots of sizes \(Q\) and \(nQ\), due within their respective lead times. Upstream stock \(W=\mathit{IL}_2-\mathit{IP}_1\) is nonnegative, and \(\beta>0\) requires \(W<Q\). All initial quantities have finite first moments. The canonical convention additionally requires \(W\in Q\mathbb Z\). \(\mathcal P_{\RnQ}\) denotes classical rules with integer controls \(Q,n\ge1\) and a compatible initialization; by Lemma~\ref{lem:rnq-remainder}, its optimized cost equals that of the canonical convention.

Both request counters in \eqref{eq:app-counters} use the same \(N\). Thus, the controls fix their relative timing, and changing the initial remainder of \(N\) moves both counters together. Physical stock held at Stage~2 is \(\mathit{IL}_2-\mathit{IP}_1=Q(\mathsf S-M)\), a multiple of \(Q\). We call this the canonical initialization convention. The remainder of physical upstream stock modulo \(Q\) is a separate invariant, addressed in Appendix~\ref{app:initial-remainder}. Uniformly randomizing the demand-counter remainder preserves this physical-stock remainder.

Write \(G_{\mathrm{adj}}(x)=\E\phi_{h_1,p+h_2}(x-D_1)\), \(G_0(x)=h_2x+G_{\mathrm{adj}}(x)=H_0+\E\phi_{h_1+h_2,p}(x-D_1)\), and \(H_0=h_2\lambda L_1\). The stationary lead-time identity and the independence of future Poisson increments give expected inventory cost \(\E[G_{\mathrm{adj}}(\mathit{IP}_1)+h_2\mathit{IL}_2]\). Equivalently, this is \(\E[G_0(\mathit{IP}_1)+h_2(\mathit{IL}_2-\mathit{IP}_1)]\).

\begin{proposition}[{\sc Exact Stationary Cost}]\label{prop:app-policy-cost}
Proposition~\ref{prop:rnq-account-main} holds, with cost
\begin{equation}
 \Clot=\frac{k_1}{Q}+\frac{k_2}{nQ}
 +\frac1n\sum_{j=0}^{n-1}\E\left[
 \mathcal L_QG_0\bigl(s-(\ell+jQ-D_2)^-\bigr)
 +h_2(\ell+jQ-D_2)^+\right],
 \label{eq:app-exact-cost}
\end{equation}
where \(\mathcal L_QF(x)=Q^{-1}\sum_{v=0}^{Q-1}F(x-v)\). Shipment accounting gives at most this cost.
\end{proposition}

\proof{Proof.} Conditional on \(D_2=d=N(t)-N(t-L_2)\), the residue \(U=N(t)\bmod nQ\) is uniform: the stationary residue at \(t-L_2\) is uniform and independent of the later increment \(d\). Put \(w'=\ell+(n-1)Q-d\). Reduce \(U+w'\) modulo \(nQ\), then divide that remainder by \(Q\): write \((U+w')\bmod nQ=\kappa Q+u\), where \(0\le \kappa<n\) and \(0\le u<Q\). The quotient \(\kappa\) identifies one of the \(n\) sublots, and the remainder \(u\) identifies one of its \(Q\) positions. Conditional on \(d\), these indices are independent and uniform. Set \(j=n-1-\kappa\) to index the sublots in reverse order, and let \(w=\ell+jQ-d\) be the corresponding reorder-point offset after lead-time demand. Finally, \(v=U\bmod Q\) is the downstream cycle residue. The counter identities give
\[
 QA=N-v,\qquad Q\mathsf S=N+w'-((U+w')\bmod nQ)=N+w'-(\kappa Q+u),\qquad w'-\kappa Q=w.
\]
Hence, \(v+w=Q(\mathsf S-A)+u\) and \(\mathit{IP}_1=s-v+Q\min\{\lfloor(v+w)/Q\rfloor,0\}\). For fixed \(\kappa\) and \(d\), varying \(u\) makes \(v\) run through every downstream residue, while \(w\) remains fixed. If \(w\ge0\), the floor is nonnegative, \(\mathit{IP}_1=s-v\), and the average of \(Q(\mathsf S-A)\) is \(w\). Indeed, both \(v\) and \((v+w)\bmod Q\) are uniform on \(\{0,\ldots,Q-1\}\), so their equal means cancel in \(Q\lfloor(v+w)/Q\rfloor=v+w-((v+w)\bmod Q)\). The conditional cost is, therefore, \(\mathcal L_QG_0(s)+h_2w\). If \(w<0\), the floor is nonpositive, and \(\mathit{IP}_1=\mathit{IL}_2=s+w-((v+w)\bmod Q)\). The conditional cost is \(\mathcal L_QG_0(s+w)\). Averaging over \(j\) and \(D_2\) gives the inventory term in \eqref{eq:app-exact-cost}.

The request and upstream-order counters have expected rates \(\lambda/Q\) and \(\lambda/(nQ)\). Since \(A-M\) is stationary with finite mean, matched lots also have expected rate \(\lambda/Q\). Charging each matched lot and each upstream lot gives the two setup terms. Every positive dispatch contains at least one lot, so its single shipment charge is bounded by the sum of its lot charges on each sample path.\Halmos\endproof

\begin{lemma}[{\sc Admissibility of Finite Controls}]\label{lem:rnq-admissible}
Every finite integer construction above satisfies \eqref{eq:app-stability} and the unfinished-path condition.
\end{lemma}

\proof{Proof.} Inventory levels, pending lots, and pipelines are bounded in absolute value by fixed constants plus demands in fixed lead-time windows. They have finite second moments in stationarity. Terminal expectations and fixed-window costs are consequently \(O(1)\). The counter rates give the two dispatch rates; \(U_0(T)=N(T)-N(0)-\mathsf B(T)+\mathsf B(0)\) gives the customer-delivery rate.

A loop-erased integer path has distinct states, so its number of surviving demand transitions is at most its state range. The settled Stage-2 position occupies the \(nQ\) states \(R_2+1,\ldots,R_2+nQ\). Immediately before an upstream order, a demand also takes it to \(R_2\). Thus, the full chronological path takes values in \(nQ+1\) consecutive states and has range at most \(nQ\), giving \(r_{2,T}\le nQ\). For Stage~1, \eqref{eq:app-counters} gives \(s-Q+1-[N(t)-N(t-L_2)-\ell]^+ \le\mathit{IP}_1(t)\le s\). When \(M=A\), the lower bound follows from the downstream residue. When \(M=\mathsf S\), the inequality \(Q\mathsf S(t)>N(t-L_2)+\ell-Q\) gives \(\mathit{IP}_1(t)>s-Q-[N(t)-N(t-L_2)-\ell]\). Since the position is integer, this implies the stated lower bound. A moving window of length \(L_2>0\) intersects at most two intervals in a partition into intervals of that length. If their demands are \(W_1,\ldots,W_J\), where \(J=O(T)\), then \(\E\max_jW_j\le(\sum_j\E W_j^2)^{1/2}=O(\sqrt T)\). The Stage-1 range, including intermediate demand states, therefore, has expectation \(o(T)\). For \(L_2=0\), it is bounded.

An arbitrary fixed initial remainder of the common demand counter, with the same counters \eqref{eq:app-counters} and their compatible history, gives the same long-run cost. This statement keeps the relative timing of the two counters fixed. The Poisson counter modulo \(nQ\) converges to uniform: for \(k=1,\ldots,nQ-1\), the modulus of the expected complex exponential is \(\exp\{\lambda t[\cos(2\pi k/(nQ))-1]\}\), which approaches zero. Condition on the counter before the fixed lead-time windows. Future increments are independent, so expectations of bounded state costs converge to their stationary values. The state bounds above give a uniformly bounded second moment of \(c_I(t)\) once the initial lead-time histories have cleared. For any \(a>0\), \(\E[c_I(t)\mathbf1_{\{c_I(t)>a\}}]\le a^{-1}\E[c_I(t)^2]\), so the expected contribution above \(a\) approaches zero uniformly as \(a\to\infty\). Truncating the cost and then removing the truncation, therefore, gives convergence of expected inventory costs to their stationary values. Setup rates and pathwise range estimates remain unchanged.\Halmos\endproof

\subsection{An Initial Remainder of Upstream Stock}\label{app:initial-remainder}

The four controls alone need not determine the cost if one also varies the remainder of the initial physical stock at Stage~2. That remainder persists because every upstream arrival adds \(nQ\) units and every downstream dispatch removes a multiple of \(Q\). The following comparison shows why the canonical convention suffices when the reorder points are optimized.

\begin{lemma}[{\sc Removing an Upstream Remainder}]\label{lem:rnq-remainder}
Fix integer classical controls \((R_1,R_2,Q,n)\) and a compatible initialization. Let \(\rho\in\{0,\ldots,Q-1\}\) be the remainder, modulo \(Q\), of the physical stock held at Stage~2. Removing \(\rho\) units from that stock and replacing \(R_2\) by \(R_2-\rho\) gives a feasible classical rule with zero upstream remainder. Every request, order, dispatch, and customer delivery occurs at the same epoch and has the same quantity. Under either setup convention, its long-run cost is lower by \(h_2\rho\). Consequently, allowing these additional initialization remainders leaves both optimized classical cost infima unchanged.
\end{lemma}

\proof{Proof.} Recall that \(W=\mathit{IL}_2-\mathit{IP}_1\) is the physical stock held at Stage~2. It is nonnegative and changes only by multiples of \(Q\), so \(W\ge\rho\) at every event and between events. Remove \(\rho\) units from the initial upstream stock, retain all initial pipelines and pending requests, and use the same demand realization. The proposed transformed states are \(\mathit{IP}'_1=\mathit{IP}_1\), \(\mathit{IL}'_1=\mathit{IL}_1\), \(\mathit{IP}'_2=\mathit{IP}_2-\rho\), and \(\mathit{IL}'_2=\mathit{IL}_2-\rho\). At Stage~2, subtracting \(\rho\) from both its position and its reorder point leaves every upstream-order decision unchanged. At Stage~1, the position and pending requests are unchanged, and the number of available complete lots is preserved because \(\lfloor(W-\rho)/Q\rfloor=\lfloor W/Q\rfloor\). Thus, each complete-lot matching decision is unchanged. This verifies the transformed states by induction over events. It also covers simultaneous receipts and requests under the specified event order. Each common upstream receipt adds the same multiple of \(Q\) to the two upstream stocks, and each common matched dispatch subtracts the same multiple of \(Q\). Their difference remains \(\rho\), and their numbers of available complete lots remain equal before the next decision. All shipment quantities, lead times, and pipelines coincide, as do customer fulfillment and backlog. The only physical difference is the removal of \(\rho\) units of upstream stock.

The transformed inventory cost is, therefore, \(c'_I=c_I-h_2\rho\) at every time, while both shipment counts and lot counts coincide. The balance and boundary conditions are preserved; shifting the Stage-2 inventory-position path by a constant also preserves all its repeated states and unfinished demand transitions. Hence, admissibility is preserved, and the claimed cost difference follows on every horizon before passage to the long-run limit.

We next verify that every compatible zero-remainder initialization is represented by the common counter. Write its upstream reorder point as \(\bar R_2\), put \(s=R_1+Q\), and let \(\beta\) be the number of pending downstream requests. Its nominal downstream position is \(\widehat{\mathit{IP}}_1=\mathit{IP}_1+Q\beta\). Because \(W\) is a multiple of \(Q\) and the incoming Stage-2 pipeline \(P_2\) is a multiple of \(nQ\), \(\mathit{IP}_2-\widehat{\mathit{IP}}_1=W+P_2-Q\beta\in Q\mathbb Z\). The settled nominal Stage-1 position lies in \(\{s-Q+1,\ldots,s\}\), and the settled Stage-2 position lies in \(\{\bar R_2+1,\ldots,\bar R_2+nQ\}\). Set \(r=(s-\mathit{IP}_2(0))\bmod nQ\). The canonical construction with \(N(0)=r\) and controls \((R_1,\bar R_2,Q,n)\) gives the unique Stage-2 position in \(\{\bar R_2+1,\ldots,\bar R_2+nQ\}\) congruent to \(s-r\) modulo \(nQ\), so it equals \(\mathit{IP}_2(0)\). Its nominal downstream position is \(s-(r\bmod Q)\), the unique point in \(\{s-Q+1,\ldots,s\}\) with the required residue, so it equals \(\widehat{\mathit{IP}}_1(0)\). Under the same future demand, all nominal requests and upstream orders consequently agree. Their cumulative counters are \(A,O\) in \eqref{eq:app-counters}, with the compatible choices of their integer origins.

After matching, pending requests and a complete available downstream lot cannot coexist. If \(\beta>0\), then \(W<Q\); since \(W\) is a multiple of \(Q\), it is zero and \(\mathit{IP}_1=\mathit{IL}_2<\widehat{\mathit{IP}}_1\). If \(\beta=0\), then \(\mathit{IP}_1=\widehat{\mathit{IP}}_1\le\mathit{IL}_2\). Thus, \(\mathit{IP}_1=\min\{\widehat{\mathit{IP}}_1,\mathit{IL}_2\}\). Matching, therefore, fixes the number dispatched once nominal requests and available sublots are specified; with the counter origins above it is precisely \(M=\min\{A,\mathsf S\}\). FIFO changes no quantity or epoch because every pending request has size \(Q\). Any difference between compatible initial lead-time histories is transient. For \(t\ge L_2\), the lead-time identity expresses \(\mathit{IL}_2(t)\) through \(\mathit{IP}_2(t-L_2)\) and subsequent demand, so the two Stage-2 levels coincide. Hence, the actual Stage-1 positions also coincide. For \(t\ge L_1+L_2\), the Stage-1 levels coincide as well. All later dispatches coincide. Finite expected cost on this fixed initial interval makes its contribution approach zero after division by the horizon. Lemma~\ref{lem:rnq-admissible}, therefore, gives the canonical long-run cost for every such compatible initialization.

With upstream remainder \(\rho\), a compatible stationary realization is obtained by replacing \(\ell\) in the upstream-order counter by \(\ell-\rho\), retaining the downstream counter, and adding \(\rho\) to both Stage-2 states in \eqref{eq:app-counter-states}. Its settled Stage-2 position lies in \(\{R_2+1,\ldots,R_2+nQ\}\), and its upstream stock is \(\rho+Q(\mathsf S-M)\). The removal comparison and the coverage argument just given identify its long-run cost, and that of every compatible initialization with this fixed \(\rho\), as
\begin{equation}
 \Clot=\frac{k_1}{Q}+\frac{k_2}{nQ}
       +I_{n,Q}(R_1+Q,R_2-R_1-\rho)+h_2\rho.
 \label{eq:app-remainder-cost}
\end{equation}
Every integer value \(R_2-\rho\) is already among the canonical controls. The canonical class is included among the possible initializations, and the comparison maps every other initialization to one of its policies at weakly smaller cost. This proves the equality of the infima. A randomized initialization gives a mixture of the same comparisons and cannot reduce their infimum.\Halmos\endproof

\section{Finite-Horizon Allocation Lower Bounds}\label{app:lower-bounds}

\subsection{The Integer-Cycle Inequality}

For a bounded-below continuous loss \(F\), define
\[
 \mathcal C_F(k)=\inf_{m\in\mathbb N,\ s\in\mathbb R}
 \left\{\frac{k}{m}+\frac1m\sum_{j=0}^{m-1}F(s-j)\right\}.
\]
This definition also covers a one-sided loss and an unattained infimum.

The cycle argument below applies to the chronological path of an adaptive policy. Its decisions may depend on time, earlier demand, outstanding orders, and independent randomization. The loss immediately before a demand is predictable under all of these choices. Unit demand fixes the downward jump size, while physical shipments may have arbitrary positive integer sizes. Thus, we can assign the actual charge of one shipment to its entire upward jump for any of these decision rules. Allowing all real placements in \(\mathcal C_F\) can only lower the cycle infimum when the physical states are constrained to a particular integer lattice, so the resulting lower bound remains valid.

\begin{lemma}[{\sc Unit-Demand Cycle Bound}]\label{lem:lb-unit-stage}
Let \(\lambda>0\), let \(N\) be the demand process, and write \(X(t)=X(0)+U(t)-N(t)\), with \(U(0)=N(0)=0\), where \(X(0)\) is integer and \(U\) is an adapted, right-continuous, nondecreasing integer process with finitely many jumps on every finite interval. Future demand increments are independent of the information used by the policy. At coincident events, record the unit demand decrease before any dispatch increase. Let \(Z(T)\) be any count at least as large as the number of positive jumps of \(U\) in \((0,T]\), and let \(K\ge0\). For a bounded-below continuous \(F\), suppose the chronological loop erasure defined in Appendix~\ref{app:admissibility} leaves \(r_T\) demand decreases and satisfies the unfinished-path condition \(\E r_T=o(T)\). Then
\[
 \liminf_{T\to\infty}\frac1T\E\left[\int_0^TF(X(t))\,dt+KZ(T)\right]
 \ge\mathcal C_F(\lambda K).
\]
If \(F\) is convex and grows without bound in both directions, the cycle infimum is attained.
\end{lemma}

\proof{Proof.} Subtract \(F_*=\inf F\) and put \(F_0=F-F_*\ge0\). Poisson compensation gives \(\E\sum_{\tau\le T}F_0(X(\tau-)) =\lambda\E\int_0^TF_0(X(t))\,dt\). The identity first applies to bounded predictable integrands; the monotone convergence theorem gives this form. At a simultaneous demand and dispatch, record the demand transition first. Assign \(F_0(x)\) to each downward transition \(x\to x-1\) and assign \(k=\lambda K\) to each positive upward jump. Apply the chronological loop erasure defined in Appendix~\ref{app:admissibility}. Each removed cycle visits distinct states until its final return to its starting state; this is a simple cycle. The retained path has no repeated state and contains \(r_T\) downward transitions.

Each erased simple cycle contains one upward jump. To see this, begin at its largest state. The initial downward steps visit consecutive states. Any upward jump must land among those already visited states, since its endpoint cannot exceed the largest state of the cycle. In a simple cycle, the only permitted repeated state is the final return to the starting state. Thus, the upward jump closes the cycle at its largest state. A cycle with \(m\) demand transitions, therefore, consists of \(m\) consecutive unit decreases followed by one upward jump of size \(m\), and pays at least \(m\mathcal C_{F_0}(k)\). The costs assigned to the retained path are nonnegative because \(F_0\ge0\) and \(k\ge0\). Dropping those costs gives
\[
 \lambda\E\left[\int_0^TF_0(X(t))\,dt+KZ(T)\right]
 \ge\mathcal C_{F_0}(k)\{\lambda T-\E r_T\}.
\]
Divide by \(\lambda T\), pass to the lower limit, and restore \(F_*\). If the expectation is infinite, the same nonnegative inequalities apply.

For attainment, choose \(A>k+F_*+1\) and a bounded interval outside which \(F\ge A\). A unit-spaced cycle has only a bounded number of states inside that interval. Consequently, all sufficiently large cycles cost more than the cycle of size one at a minimizer of \(F\). Only finitely many cycle sizes remain; the continuity and growth at infinity of the cycle cost give an optimal placement for each.\Halmos\endproof

\subsection{Lead-Time Shifts and the Cost Split}

For an admissible policy \(\Pi\) and \(\lambda>0\), let \(I_T^\Pi=T^{-1}\E_\Pi\int_0^T c_I(t)\,dt\) be its expected average inventory cost over \([0,T]\). At setup coefficients \(a_1,a_2\ge0\), its total expected average cost is \(C_T^\Pi(a_1,a_2)=I_T^\Pi+(\lambda T)^{-1}\sum_{i=1}^2a_i\E_\Pi Z_i^{\rm sh}(T)\).

The time-shift estimate below applies to \(h_2\mathit{IL}_2\), which can be negative because \(\mathit{IL}_2\) is net inventory, and remains valid at zero cost rates. For any fixed \(0\le d\le L_\Sigma\), the flow balances imply, for \(0\le t\le d\),
\[
 |\mathit{IL}_2(T+t)|\le|\mathit{IL}_2(T)|
 +U_2(T+t-L_2)-U_2(T-L_2)+N(T+t)-N(T).
\]
The terminal-state and fixed-window dispatch conditions, therefore, give \(\E\int_T^{T+d}|\mathit{IL}_2(t)|dt=o(T)\). The corresponding initial integral is finite. Thus, shifting the integration interval by a fixed lead time changes the expected integral of \(h_2\mathit{IL}_2\) by \(o(T)\). This argument uses physical stability and remains applicable when a cost coefficient becomes zero.

Fix \((\eta,\zeta)\in[0,1]^2\). Split physical cost into \(c_1=h_1I_1+\eta h_2I_2+\zeta p\mathsf B\) and \(c_2=(1-\eta)h_2I_2+(1-\zeta)p\mathsf B\), and assign setup \(K_i\) to component \(i\). Both components are nonnegative. Define \(F_1^{\eta,\zeta}(x)=\E\phi_{h_1+\eta h_2,\zeta p}(x-D_1)\) and \(F_2^{\eta,\zeta}(x)=\E\phi_{(1-\eta)h_2,(1-\zeta)p}(x-D_1-D_2)\). To see the first component directly, use \(I_2=\mathit{IL}_2+\mathsf B\) to write \(c_1(t)=\phi_{h_1,\zeta p+\eta h_2}(\mathit{IL}_1(t)) +\eta h_2\mathit{IL}_2(t)\). Conditional on the information at \(t\), the first term at \(t+L_1\) has expectation \(\widetilde G(\mathit{IP}_1(t))\), where \(\widetilde G(x)=\E\phi_{h_1,\zeta p+\eta h_2}(x-D_1)\). Orders issued during this lead time cannot enter \(\mathit{IL}_1(t+L_1)\), which is the reason the conditioning remains valid for an adaptive policy. The possibly negative term \(\eta h_2\mathit{IL}_2\) is shifted in time using the estimate above. Material feasibility then gives \(\eta h_2\mathit{IL}_2+\widetilde G(\mathit{IP}_1) \ge\eta h_2\mathit{IP}_1+\widetilde G(\mathit{IP}_1) =\eta H_0+F_1^{\eta,\zeta}(\mathit{IP}_1)\). Integrate through \(T-L_1\), retain the setup count through that time, and apply Lemma~\ref{lem:lb-unit-stage} with \(X=\mathit{IP}_1\) and \(F=\eta H_0+F_1^{\eta,\zeta}\). The first component has lower limiting rate at least \(\eta H_0+\mathcal C_{F_1^{\eta,\zeta}}(k_1)\).

For the second component, define the auxiliary inventory level \(Z_t=\mathit{IP}_2(t)-[N(t+L_\Sigma)-N(t)]\). Apply the material inequality at \(t+L_2\): all units already dispatched downstream by that time are available within echelon~2. Subtract demand over the remaining interval of length \(L_1\). This gives
\begin{align*}
 \mathit{IL}_1(t+L_\Sigma)
 &=\mathit{IP}_1(t+L_2)-[N(t+L_\Sigma)-N(t+L_2)]\\*
 &\le\mathit{IL}_2(t+L_2)-[N(t+L_\Sigma)-N(t+L_2)]
 =Z_t.
\end{align*}
The supplier-dispatch balance similarly gives \(\mathit{IL}_1(t+L_\Sigma)\le Z_t\) and \(\mathit{IL}_2(t+L_\Sigma)=Z_t+U_2(t+L_1)-U_2(t)\ge Z_t\). Hence, \(\mathsf B(t+L_\Sigma)\ge Z_t^-\) and \(I_2(t+L_\Sigma)\ge Z_t^+\). Conditional on the information at \(t\), the increment in \(Z_t\) is independent Poisson demand over \(L_\Sigma\). The conditional allocated loss is at least \(F_2^{\eta,\zeta}(\mathit{IP}_2(t))\). Integrating through \(T-L_\Sigma\) and applying the cycle lemma with \(X=\mathit{IP}_2\) and \(F=F_2^{\eta,\zeta}\) gives the second lower limiting rate \(\mathcal C_{F_2^{\eta,\zeta}}(k_2)\). The discarded initial physical costs are nonnegative. The original two components have finite lower bounds, so their lower limits can be added. We have proved, for every admissible policy,
\begin{equation}
 \liminf_{T\to\infty}C_T^\Pi(k_1,k_2)
 \ge\eta H_0+\mathcal C_{F_1^{\eta,\zeta}}(k_1)+\mathcal C_{F_2^{\eta,\zeta}}(k_2).
 \label{eq:app-allocation-valid}
\end{equation}
Taking the allocation supremum gives \(\mathcal A(k_1,k_2)\le\OPTsh\). We allow zero allocated slopes and zero setups without requiring the convergence of either separate cost component.

\section{The Zero-Setup Value and Setup-Preserving Rescaling}\label{app:zero-setup}

Define the constrained cost
\begin{equation}
 \mathcal E_0(x)=\inf_{v\le x}\{G_0(v)+h_2(x-v)\}
 =h_2x+\inf_{v\le x}G_{\mathrm{adj}}(v).
 \label{eq:app-zero-envelope}
\end{equation}
It is finite, convex, and at least \(H_0\). Its growth is at most linear. These facts follow either from the convex constrained minimization or from the explicit forms below.

\begin{proposition}[{\sc Finite-Horizon Zero-Setup Benchmark}]\label{prop:app-zero-value}
For all nonnegative cost rates, the value \(V_0\) in \eqref{eq:lb-protected-zero-value} satisfies
\[
 V_0=\inf_{z\in\mathbb R}\E\mathcal E_0(z-D_2)
 =\left.\OPTsh\right|_{K_1=K_2=0},\qquad
 \liminf_{T\to\infty}I_T^\Pi\ge V_0
\]
for every admissible policy. Here the optimum is evaluated with both setup charges set to zero. In that system, finite integer unit-lot policies approach \(V_0\) arbitrarily closely.
\end{proposition}

\proof{Proof.} The Stage-1 lead-time shift and the preceding time-shift estimate give, with \(T'=T-L_1\),
\[
 T I_T^\Pi\ge\E\int_0^{T'}
 [G_{\mathrm{adj}}(\mathit{IP}_1(t))+h_2\mathit{IL}_2(t)]\,dt-o(T)
 \ge\E\int_0^{T'}\mathcal E_0(\mathit{IL}_2(t))\,dt-o(T).
\]
The second inequality uses \(\mathit{IP}_1\le\mathit{IL}_2\). Since \(\mathcal E_0\ge0\), discard its initial interval of length \(L_2\). The Stage-2 lead-time identity then gives conditional cost \(\E\mathcal E_0(\mathit{IP}_2(t)-D_2)\), which is at least \(\inf_z\E\mathcal E_0(z-D_2)\) at every time. Division by \(T\) proves the finite-horizon lower bound.

If \(h_1>0\) and \(p+h_2>0\), the function \(G_{\mathrm{adj}}\) grows without bound in both directions; choose an integer minimizer \(r_0\). The constrained minimizer in \eqref{eq:app-zero-envelope} is \(\min\{r_0,x\}\), the smaller of the target \(r_0\) and the available level \(x\). We call the target \(r_0\) the cap. Consequently, with \(z=r_0+y\), \(\mathcal E_0(z-D_2) =G_0(r_0-(y-D_2)^-)+h_2(y-D_2)^+\). For any cap \(r\), the choice \(v=\min\{r,x\}\) is feasible in the constrained problem and has at least this cost. This proves the equality with the two-variable definition of \(V_0\).

If \(h_1=0\), \(G_{\mathrm{adj}}\) is nonincreasing and the constrained choice is \(v=x\). For each finite \(z\), the nonnegativity of \(D_2\) permits the finite cap \(r=z\), which implements this choice for every \(x=z-D_2\). If \(p+h_2=0<h_1\), then \(p=h_2=0\); the cap \(r=0\) gives zero cost because \(D_1\ge0\). These cases cover every remaining boundary, including all zero rates.

The functions just described are linear between successive integers. Their expectations over integer \(D_2\) retain that property, so restricting finite placements to integers does not change the infimum. Choose an integer placement within any prescribed tolerance, and the corresponding integer cap. Equation~\eqref{eq:app-exact-cost}, with \(Q=n=1\), realizes exactly that value when setups are zero. Lemma~\ref{lem:rnq-admissible} establishes admissibility. Thus, the policy infimum is at most \(V_0\), establishing the equality. When all relevant slopes are positive, the remaining scalar loss grows without bound in both directions, and the minimum is attained.\Halmos\endproof

Finally, for each \(0<u\le1\), the same physical policy satisfies the exact finite-horizon identity \(C_T^\Pi(k_1,k_2)=(1-u)I_T^\Pi +uC_T^\Pi(k_1/u,k_2/u)\). Taking lower limits, using Proposition~\ref{prop:app-zero-value} and \eqref{eq:app-allocation-valid}, and then taking the policy infimum gives \((1-u)V_0+u\mathcal A(k_1/u,k_2/u)\le\OPTsh\). The supremum over \(u\) is \(\LBsp\). Here SP stands for setup-preserving: multiplying by \(u\) restores the original setup coefficients after their rescaling by \(1/u\). The choice \(u=1\) gives \(\mathcal A\le\LBsp\); the nonnegativity of \(\mathcal A\) and \(u\downarrow0\) also give \(V_0\le\LBsp\). Removing lot labels from any fixed-lot policy preserves its material flows and weakly decreases setup charges. Therefore, \(\LBsp\le\OPTsh\le\OPTlot\), and \(\Csh\le\Clot\) for each constructed policy. Every comparison was established before taking long-run policy infima; no interchange of component limits is required.
\section{The Joint-Cycle Comparison}\label{app:joint-cycles}

This appendix proves the comparison used in Theorem~\ref{thm:five-thirds}. We first choose the two cycle sizes independently in an auxiliary inventory calculation. We then turn each such choice into a finite collection of feasible classical policies, one of which costs no more than that calculation. The remaining argument compares larger and smaller integer cycles and combines two valid lower bounds. This separation allows each stage to have its own preferred ordering frequency while ensuring that the implemented upstream order contains a whole number of downstream lots.

Throughout the first four subsections, assume \(\lambda>0\) and \(h_1,h_2,p>0\). Setup coefficients \(k_i=\lambda K_i\) may be zero, and lead times are finite and nonnegative. Appendix~\ref{app:joint-boundary} treats zero holding or backlog costs and zero demand. Let \(D_1,D_2\) be independent Poisson variables with means \(\lambda L_1,\lambda L_2\). Set \(a=h_1+h_2\), \(H_0=h_2\E D_1\), and
\begin{equation}
 P(x)=\E\ph_{a,p}(x-D_1),\qquad
 F_{\mathrm{up}}(x)=\E\ph_{h_2,p}(x-D_2),\qquad G_0(x)=H_0+P(x).
 \label{appjc:losses}
\end{equation}
For a positive integer \(j\), let \(J_j\) be uniform on \(\{0,\ldots,j-1\}\). Copies used for different cycles are independent of one another and of demand. Recall
\begin{equation}
 \mathcal H_j(F)=\inf_{y\in\mathbb R}\E F(y-J_j),\qquad
 \mathcal C_F(k)=\inf_{j\in\mathbb N}\{k/j+\mathcal H_j(F)\},
 \quad B_1(k)=H_0+\mathcal C_P(k),\quad B_2(k)=\mathcal C_{F_{\mathrm{up}}}(k).
 \label{appjc:single-cycle}
\end{equation}
Here \(k/j\) is the setup rate associated with a lot of \(j\) units, and \(\mathcal H_j\) chooses the placement of its inventory positions. We use \(V_0\) for the zero-setup value in \eqref{eq:lb-protected-zero-value}. Its interpretation and finite-horizon lower-bound property are established in Appendix~\ref{app:zero-setup}.

\subsection{A Common Inventory Calculation for Both Cycles}

Let \(\Xi_1,\Xi_2\) be independent random variables with finite support, independent of demand. Define
\begin{equation}
 \mathscr V(\Xi_1,\Xi_2)=\inf_{z,x}\E\left[
 H_0+\ph_{a,p}(x-D_1-\Xi_1)+h_2(z-D_2-\Xi_2-x)\right],
 \qquad x=x(D_2,\Xi_2)\le z-D_2-\Xi_2.
 \label{appjc:joint-value}
\end{equation}
The level \(z\) is deterministic. The choice of \(x\) can respond to \(D_2\) and \(\Xi_2\); it cannot observe \(D_1\) or \(\Xi_1\). The infimum is over measurable choices with \(\E|x|<\infty\). The constraint represents material availability, and its nonnegative difference is charged at rate \(h_2\). All inventory and backlog terms remain in the calculation. The flexibility in \(x\) is used to establish a useful intermediate value; the next subsection provides actual policy controls.

For a fixed distribution of \(\Xi_1\), put \(G_{\Xi_1}(v)=H_0+\E_{D_1,\Xi_1}\ph_{a,p}(v-D_1-\Xi_1)\), where the expectation averages over both \(D_1\) and \(\Xi_1\). The slopes of \(G_{\Xi_1}(v)-h_2v\) approach \(-(p+h_2)\) on the left and \(h_1\) on the right. This convex function, therefore, grows without bound in either direction and has a deterministic minimizer, denoted by \(r_{\Xi_1}\), chosen from the distribution of \(\Xi_1\) before its realization. Given \(y=z-D_2-\Xi_2\), minimizing \(G_{\Xi_1}(x)-h_2x+h_2y\) over \(x\le y\) in \eqref{appjc:joint-value} gives the choice capped at \(r_{\Xi_1}\)
\begin{equation}
 x=\min\{r_{\Xi_1},y\}.
 \label{appjc:cap-choice}
\end{equation}
Indeed, the unconstrained minimum is feasible when \(y\ge r_{\Xi_1}\). When \(y<r_{\Xi_1}\), the convexity of the objective makes it nonincreasing up to \(r_{\Xi_1}\), so the largest feasible choice is optimal. Formula~\eqref{appjc:cap-choice} is integrable and depends on \((D_2,\Xi_2)\) only through their sum. Thus, the separate observations permitted in \eqref{appjc:joint-value} do not change its value. In particular,
\begin{equation}
 \mathscr V(\Xi_1,\Xi_2)=\inf_z\E\left[
 G_{\Xi_1}\bigl(r_{\Xi_1}-(r_{\Xi_1}-z+D_2+\Xi_2)^+\bigr)
 +h_2(z-D_2-\Xi_2-r_{\Xi_1})^+\right].
 \label{appjc:cap-value}
\end{equation}
At \(\Xi_1=\Xi_2=0\), minimizing first over the constrained downstream choice gives the zero-setup expression in \eqref{eq:lb-protected-zero-value}. Hence, \(\mathscr V(0,0)=V_0\).

\begin{lemma}[{\sc Constant Shifts and Simultaneous Scaling}]\label{appjc:scale-lemma}
For constants \(c_1,c_2\), \(\mathscr V(\Xi_1+c_1,\Xi_2+c_2)=\mathscr V(\Xi_1,\Xi_2)\). Moreover, \(t\mapsto\mathscr V(t\Xi_1,t\Xi_2)\) is convex for \(t\ge0\), and
\begin{equation}
 \mathscr V(\Xi_1,\Xi_2)\ge V_0,\qquad
 \mathscr V(2\Xi_1,2\Xi_2)\ge2\mathscr V(\Xi_1,\Xi_2)-V_0.
 \label{appjc:inventory-doubling}
\end{equation}
\end{lemma}
\proof{Proof.} For constant shifts, replace \(x\) by \(x-c_1\) and \(z\) by \(z-c_1-c_2\). The two loss arguments and the availability constraint are unchanged. These reversible substitutions preserve the infima.

For the scaling comparison, fix the underlying variables \((D_1,D_2,\Xi_1,\Xi_2)\). Let \(\widetilde v(t)\) be the value of \eqref{appjc:joint-value} at \((t\Xi_1,t\Xi_2)\) when the downstream choice may depend on \((D_2,\Xi_2)\) for every \(t\ge0\). At scales \(t_0,t_1\), select feasible choices \((z_0,x_0(D_2,\Xi_2))\) and \((z_1,x_1(D_2,\Xi_2))\). For \(\theta\in[0,1]\), their weighted averages satisfy
\[
 (1-\theta)x_0+\theta x_1
 \le (1-\theta)z_0+\theta z_1-D_2-
       \{(1-\theta)t_0+\theta t_1\}\Xi_2.
\]
The convexity of \(\ph_{a,p}\) bounds the mismatch loss at these averages by the same average of the two losses, and the remaining cost is linear. Choosing solutions arbitrarily close to their infima proves that \(\widetilde v\) is convex. For \(t>0\), observing \(t\Xi_2\) is equivalent to observing \(\Xi_2\), so \(\widetilde v(t)=\mathscr V(t\Xi_1,t\Xi_2)\). At \(t=0\), the conditional objective given \((D_2,\Xi_2)\) does not involve \(\Xi_2\); for each fixed \(z\), the choice \(\min\{r_0,z-D_2\}\) is optimal, where \(r_0\) minimizes \(G_0(x)-h_2x\). Thus, \(\widetilde v(0)=\mathscr V(0,0)=V_0\). The midpoint convexity at scales zero, one, and two gives the second inequality in \eqref{appjc:inventory-doubling}.

For the first inequality, condition on \((\Xi_1,\Xi_2)\) and allow the endpoint and downstream choice to be selected separately for every conditioning value. Allowing these additional choices can only lower the infimum. The independence of demand and cycle positions preserves each conditional demand distribution, and shifting the inventory levels removes the fixed cycle positions. Each conditional value is, therefore, \(V_0\).\Halmos\endproof

The information restriction in \eqref{appjc:joint-value} is important for this argument. The downstream demand remains unknown when \(x\) is chosen. In the conditioning argument, revealing a cycle position enlarges the allowable choices and is used only to obtain a lower bound. Scaling both cycle positions by the same factor keeps the weighted endpoint and downstream choice feasible at the intermediate scale, where the convexity bound applies to the total inventory cost.

\subsection{Converting Independent Cycle Sizes into Feasible Policies}\label{app:joint-policy}

For positive integer sizes \(m_1,m_2\) with independent cycle positions, define
\begin{equation}
 V_{m_1,m_2}=\mathscr V(J_{m_1},J_{m_2}),\qquad
 \mathcal J(k_1,k_2)=\inf_{(m_1,m_2)\in\mathbb N^2}
       \{k_1/m_1+k_2/m_2+V_{m_1,m_2}\}.
 \label{appjc:J-def}
\end{equation}
The sizes \(m_1,m_2\) need not be multiples of each other. The following conversion preserves their cost comparison while producing an upstream lot that is an integer multiple of the downstream lot.

Fix \(m_1\), and write
\begin{equation}
 \begin{aligned}
 G_{m_1}(v)&=\E G_0(v-J_{m_1}),& r_{m_1}&\in\arg\min_v\{G_{m_1}(v)-h_2v\},\\
 \widetilde F_{m_1}(w)&=G_{m_1}(r_{m_1}-w^-)-G_{m_1}(r_{m_1})+h_2w^+,& F_{m_1}(y)&=\E \widetilde F_{m_1}(y-D_2).
 \end{aligned}
 \label{appjc:residual-loss}
\end{equation}
The function \(\widetilde F_{m_1}\) records the change in cost from the reference level \(r_{m_1}\). Its left slopes are those of \(G_{m_1}\) below \(r_{m_1}\). Hence, they are at most \(h_2\); its right slope is \(h_2\). It is convex, with limiting slopes \(-p\) and \(h_2\). The function \(F_{m_1}\) is also convex, with slopes bounded between \(-p\) and \(h_2\) and tending to these limits at the two ends. Averaging over the unbounded support of \(D_2\) can leave its right slope strictly below \(h_2\) at every finite argument. Integer demands and cycle positions allow an integer \(r_{m_1}\) and make \(F_{m_1}\) linear between consecutive integers. With \(w=z-r_{m_1}-D_2-J_{m_2}\), the integrand in \eqref{appjc:cap-value} is \(G_{m_1}(r_{m_1})+\widetilde F_{m_1}(w)\). Averaging over \(D_2\) and optimizing the endpoint \(z-r_{m_1}\), therefore, gives
\begin{equation}
 V_{m_1,m_2}=G_{m_1}(r_{m_1})+\mathcal H_{m_2}(F_{m_1}).
 \label{appjc:inventory-cycle-reduction}
\end{equation}
The objective \(y\mapsto\E F_{m_1}(y-J_{m_2})\) has limiting slopes \(-p<0\) and \(h_2>0\), so it grows without bound in both directions and attains its minimum at a finite endpoint. It is linear between consecutive integers, so a minimizing endpoint can be chosen as an integer. The same reasoning applies to the threshold objective, whose limiting slopes are \(-(p+h_2)<0\) and \(h_1>0\).

For a positive integer multiplier \(g\) and integer difference \(\ell\), Proposition~\ref{prop:rnq-account-main} gives the feasible controls
\begin{equation}
 (R_1,R_2,Q_1,Q_2)=(r_{m_1}-m_1,\ r_{m_1}-m_1+\ell,\ m_1,\ gm_1).
 \label{appjc:physical-controls}
\end{equation}
Since \(\mathcal L_{m_1}G_0=G_{m_1}\), the bracket in \eqref{eq:rnq-account-main} at \(s=r_{m_1}\) and \(w=\ell+jm_1-D_2\) equals \(G_{m_1}(r_{m_1})+\widetilde F_{m_1}(w)\). Averaging over \(j=0,\ldots,g-1\) and \(D_2\) gives their exact lot-accounted cost
\begin{equation}
 \frac{k_1}{m_1}+G_{m_1}(r_{m_1})+\frac{k_2}{gm_1}
       +\E F_{m_1}(\ell+m_1J_g).
 \label{appjc:physical-cost}
\end{equation}
Every control is integer, and every dispatch consists of complete lots. Appendix~\ref{app:policy-cost} proves construction and admissibility, including negative reorder-point differences.

Choose an integer \(z\) attaining the infimum in \eqref{appjc:cap-value} for \((J_{m_1},J_{m_2})\). Partition the positions \(0,\ldots,m_2-1\) according to their remainders after division by \(m_1\). A nonempty group with remainder \(b\) consists of \(b,b+m_1,\ldots,b+(g_b-1)m_1\) and has probability \(g_b/m_2\). Conditional on this group, the last term in \eqref{appjc:inventory-cycle-reduction} is \(\E F_{m_1}(z-r_{m_1}-b-m_1J_{g_b})\). The variables \(J_{g_b}\) and \(g_b-1-J_{g_b}\) have the same distribution. Therefore, setting
\begin{equation}
 \ell_b=z-r_{m_1}-b-m_1(g_b-1)
 \label{appjc:group-placement}
\end{equation}
makes this conditional cost exactly \(\E F_{m_1}(\ell_b+m_1J_{g_b})\), which appears in the physical cost formula. Thus, every group gives an explicit feasible policy.

\begin{lemma}[{\sc Choosing a Feasible Policy}]\label{appjc:policy-lemma}
For each \((m_1,m_2)\), a nonempty remainder group gives a deterministic classical policy with lot-accounted cost at most \(k_1/m_1+k_2/m_2+V_{m_1,m_2}\). Thus,
\begin{equation}
 \inf_{\Pi\in\mathcal P_{\RnQ}}\Clot(\Pi)\le\mathcal J(k_1,k_2).
 \label{appjc:policy-bound}
\end{equation}
\end{lemma}
\proof{Proof.} Assign group \(b\)'s policy the weight \(g_b/m_2\). The weighted inventory cost is exactly \(V_{m_1,m_2}\), and every group retains the downstream setup term \(k_1/m_1\). There are \(\min\{m_1,m_2\}\) nonempty groups. Their weighted upstream setup rate is
\begin{equation}
 \sum_b\frac{g_b}{m_2}\frac{k_2}{m_1g_b}
 =\frac{k_2\min\{m_1,m_2\}}{m_1m_2}\le\frac{k_2}{m_2}.
 \label{appjc:setup-group-average}
\end{equation}
At least one member of this finite collection costs no more than its weighted average. Taking the infimum over \((m_1,m_2)\) proves \eqref{appjc:policy-bound}.\Halmos\endproof

When \(m_2\ge m_1\), write \(m_2=qm_1+d\), with \(0\le d<m_1\). The possible group sizes are \(q\) and \(q+1\): there are \(m_1-d\) groups of the first size and \(d\) of the second. The conversion, therefore, uses only two adjacent upstream multiples, and \eqref{appjc:setup-group-average} holds with equality. When \(m_2<m_1\), each nonempty group has one position and gives equal lot sizes \(Q_1=Q_2=m_1\); its upstream setup rate is no greater than the auxiliary rate. Operationally, grouping preserves the downstream delivery size and adjusts the upstream order to a feasible multiple. The weighted average of the policies' total costs is at most the auxiliary cost, so at least one deterministic policy costs no more than that auxiliary cost.

\paragraph{\textbf{\textup{Running example (a): coordinating the lot sizes.}}} For \(m_1=3,m_2=8\) from Section~\ref{sec:guarantee} and Figure~\ref{fig:cycle-groups}(a), write \(r=r_3\), and let \(z\) be an optimal auxiliary endpoint. Put \(y=z-r\). The three groups are \(\{0,3,6\}\), \(\{1,4,7\}\), and \(\{2,5\}\). Equation~\eqref{appjc:group-placement} gives the controls below, all with \(Q_1=3\) and \(R_1=r-3\).
\begin{equation}
 \begin{array}{c|c|c|c|c|c}
 b&g_b&\text{weight}&\ell_b&R_2&Q_2\\ \hline
 0&3&3/8&y-6&z-9&9\\
 1&3&3/8&y-7&z-10&9\\
 2&2&2/8&y-5&z-8&6
 \end{array}
 \label{appjc:worked-groups}
\end{equation}
The reorder points are part of the conversion. In particular, the two rows with the same lot sizes place their upstream cycles differently. All three rows keep the downstream threshold \(r\), so the common contribution \(k_1/3+G_3(r)\) is unchanged. The residual inventory costs in the first row are the average of \(F_3(y-6),F_3(y-3),F_3(y)\). Those in the second row are the average of \(F_3(y-7),F_3(y-4),F_3(y-1)\). The final row averages \(F_3(y-5),F_3(y-2)\). Each original position, therefore, appears exactly once when the rows are combined with their stated weights.

Let \(C_b\) be row \(b\)'s complete lot-accounted cost. The setup and inventory calculations give
\begin{align}
 \frac38\frac{k_2}{9}+\frac38\frac{k_2}{9}+\frac28\frac{k_2}{6}
 &=\frac{k_2}{8},\\
 \frac38C_0+\frac38C_1+\frac28C_2
 &=\frac{k_1}{3}+\frac{k_2}{8}+G_3(r)
   +\frac18\sum_{j=0}^{7}F_3(y-j)
 =\frac{k_1}{3}+\frac{k_2}{8}+V_{3,8}.
 \label{appjc:worked-group-cost}
\end{align}
Thus, evaluating these three policies and retaining the cheapest completes the conversion for this auxiliary pair. The procedure needs no information about future demand: \(r,z\), and each row's four controls are fixed before operations begin. The example also shows why choosing just the nearest feasible upstream lot size would leave the comparison incomplete. Both the inventory positions and the setup rates enter the weighted identity. As a variation, take \(m_1=3,m_2=2\), so the auxiliary upstream cycle is smaller than the downstream cycle. The two groups each have one position and produce upstream lots of three units; their average upstream setup rate is \(k_2/3\le k_2/2\).\Halmos

\subsection{A Stagewise Upper Bound for the Joint Value}

The joint calculation is also bounded above by two separate cycle calculations. To see this, fix \((m_1,m_2)\), and choose an integer \(s\) minimizing \(G_{m_1}\). In \eqref{appjc:joint-value}, use the feasible choice \(x=\min\{s,z-D_2-J_{m_2}\}\). If \(w=z-s-D_2-J_{m_2}\), the resulting cost is \(G_{m_1}(s-w^-)+h_2w^+\). For \(w\ge0\), its excess over \(G_{m_1}(s)\) is \(h_2w\). For \(w<0\), every slope of \(G_{m_1}\) is at least \(-p\). Hence, \(G_{m_1}(s+w)-G_{m_1}(s)\le-pw\). The excess is at most \(\ph_{h_2,p}(w)\). Optimizing upstream placement gives
\begin{equation}
 V_{m_1,m_2}\le H_0+\mathcal H_{m_1}(P)+\mathcal H_{m_2}(F_{\mathrm{up}}),
 \qquad
 \mathcal J(k_1,k_2)\le B_1(k_1)+B_2(k_2).
 \label{appjc:separate-upper}
\end{equation}
The second inequality follows by adding the two setup rates and minimizing independently over the integer sizes. These separate minimizations define \(B_1(k_1)\) and \(B_2(k_2)\); the chosen sizes remain eligible for the group construction.

At equal allocation, \(F_1^{1/2,1/2}(x)\ge P(x)/2\). The convexity of \(\ph_{h_2,p}\) and the independence of \(D_1,D_2\) give \(F_2^{1/2,1/2}(x)\ge F_{\mathrm{up}}(x-\E D_1)/2\). For \(c>0\), the definition gives \(\mathcal C_{cF}(k)=c\mathcal C_F(k/c)\). Also, \(\mathcal C_{F(\cdot-d)}(k)=\mathcal C_F(k)\) for every real \(d\), because \(\mathcal H_j\) optimizes over real placements. Taking \(d=\E D_1\) and including the allocated constant \(H_0/2\), we obtain
\begin{equation}
 \begin{aligned}
 \mathcal A(k_1,k_2)
 &\ge \tfrac12H_0+\mathcal C_{P/2}(k_1)+\mathcal C_{F_{\mathrm{up}}/2}(k_2)\\
 &=\tfrac12\{B_1(2k_1)+B_2(2k_2)\}
 \ge\tfrac12\mathcal J(2k_1,2k_2).
 \end{aligned}
 \label{appjc:equal-allocation}
\end{equation}
Appendix~\ref{app:lower-bounds} establishes the validity of the allocation against every admissible shipment policy. Inequality~\eqref{appjc:equal-allocation} links this lower bound to the joint inventory value. The allocation comparison remains valid when some cost rates become zero.

\subsection{Even and Odd Positions and the Two-Scale Comparison}\label{app:joint-doubling}

We write \(tk=(tk_1,tk_2)\), keeping all holding, backlog, and demand parameters fixed. Separating the even and odd positions at both stages gives the following comparison.

\begin{lemma}[{\sc Joint Cycle Doubling}]\label{appjc:doubling-lemma}
For every nonnegative setup vector,
\begin{equation}
 \mathcal J(4k)\ge2\mathcal J(k)-V_0,\qquad
 \mathcal J(2k)\ge\frac43\mathcal J(k)-\frac13V_0.
 \label{appjc:J-doubling}
\end{equation}
\end{lemma}
\proof{Proof.} For an integer size \(j\ge2\), the even positions of \(J_j\) have the distribution \(2J_{\lceil j/2\rceil}\), and the odd positions have the distribution \(1+2J_{\lfloor j/2\rfloor}\). Each set has probability equal to its size divided by \(j\). Let \(S_j\) be the size of the set containing the realized position \(J_j\); this is a positive random variable. The two stages use independent copies, including when their cycle sizes are equal. Then
\begin{equation}
 \E(1/S_j)=2/j\quad(j\ge2).
 \label{appjc:parity-reciprocal}
\end{equation}
For \(j=1\), retain its single even position, set \(S_1=1\), and omit the empty set of odd positions. Its position is zero, and \(\E(1/S_1)=1\le2\). All retained sets, therefore, have positive size.

Apply the two independent partitions to \(J_{m_1},J_{m_2}\). Conditional on their even/odd labels, allow a separate endpoint \(z\) and downstream choice for each pair of labels. This enlarges the choices in the inventory problem. Conditional on the labels, the remaining cycle positions are still independent, and the downstream choice still cannot see \(D_1\) or the remaining downstream position. Shifting the endpoint and downstream choice removes the parity offsets, and Lemma~\ref{appjc:scale-lemma} gives
\begin{equation}
 V_{m_1,m_2}\ge\E\mathscr V(2J_{S_{m_1}},2J_{S_{m_2}})
       \ge2\E V_{S_{m_1},S_{m_2}}-V_0.
 \label{appjc:parity-inventory}
\end{equation}
The first inequality has this direction because separate choices after observing the labels can lower the optimized cost. It does not give the actual policy additional information.

Equation~\eqref{appjc:parity-reciprocal} and the convention for cycles of size one imply
\begin{equation}
 \frac{4k_1}{m_1}+\frac{4k_2}{m_2}
 \ge2\E\left[\frac{k_1}{S_{m_1}}+\frac{k_2}{S_{m_2}}\right].
 \label{appjc:parity-setups}
\end{equation}
Adding \eqref{appjc:parity-inventory} shows that every candidate defining \(\mathcal J(4k)\) costs at least \(2\mathcal J(k)-V_0\). Taking its infimum proves the first inequality in \eqref{appjc:J-doubling}. Next, \(t\mapsto\mathcal J(tk)\) is concave, since it is an infimum of affine functions of \(t\). As \(2=(2/3)1+(1/3)4\), \(\mathcal J(2k)\ge(2/3)\mathcal J(k)+(1/3)\mathcal J(4k)\ge(4/3)\mathcal J(k)-(1/3)V_0\). This proves the second inequality.\Halmos\endproof

\paragraph{\textbf{\textup{Running example (b): comparing cycle scales.}}} For the same \(m_1=3,m_2=8\), consider the separation into even and odd positions in Figure~\ref{fig:cycle-groups}(b). The downstream cycle of size three has the even-position set \(\{0,2\}\) and the odd-position set \(\{1\}\), of sizes two and one. Subtracting each set's first position and dividing by two gives smaller cycles of sizes two and one, with probabilities \(2/3\) and \(1/3\). The upstream cycle of size eight has sets \(\{0,2,4,6\}\) and \(\{1,3,5,7\}\), so both sets produce cycles of size four, each with probability \(1/2\). The independent partitions give four pairs of even/odd labels. Combining pairs that produce the same smaller sizes gives
\begin{equation}
 \begin{array}{c|c|c}
 (S_3,S_8)&(2,4)&(1,4)\\ \hline
 \text{probability}&2/3&1/3.
 \end{array}
 \label{appjc:worked-parity}
\end{equation}
For example, an even first position and an odd second position have the conditional distributions \(2J_2\) and \(1+2J_4\). Shifting the upstream endpoint removes the second offset. Revealing these two even/odd labels permits a separate endpoint for this conditional problem, which lowers its optimized value. The remaining downstream position and downstream lead-time demand retain the information restriction of \eqref{appjc:joint-value}. Apply the scale inequality to the four conditional problems, average, and combine the terms with the same smaller sizes: \(V_{3,8}\ge2\{(2/3)V_{2,4}+(1/3)V_{1,4}\}-V_0\). The coefficient of \(V_0\) is exactly one because the probabilities sum to one. Every division uses an exact set size.

The setup calculation uses the same set probabilities. Downstream, the even-position set contributes \((2/3)(k_1/2)=k_1/3\), and the odd-position set contributes \((1/3)k_1=k_1/3\). Upstream, each set contributes \((4/8)(k_2/4)=k_2/8\). Thus, the weighted setup costs are \(2k_1/3\) downstream and \(k_2/4\) upstream, as shown in Figure~\ref{fig:cycle-groups}(b). Counting both twice gives \(2\E\{k_1/S_3+k_2/S_8\}=4k_1/3+k_2/2\). Adding this equality to the inventory comparison yields twice a weighted average of the two smaller-cycle candidates for \(\mathcal J(k)\), less \(V_0\). Each candidate is at least \(\mathcal J(k)\). This verifies the lower estimate \(2\mathcal J(k)-V_0\) for the particular candidate with sizes three and eight at setup vector \(4k\).

For a cycle of size one, take \(m_1=1,m_2=8\). The downstream cycle's only position is zero and remains zero when scaled. Keep \(S_1=1\), and apply the same two upstream sets for the cycle of size eight. The doubled average setup is then \(2k_1+k_2/2\), whereas the original candidate at setup vector \(4k\) pays \(4k_1+k_2/2\). Their difference is \(2k_1\ge0\), exactly the allowance used in \eqref{appjc:parity-setups}. The inventory comparison still uses \(\mathscr V(0,2J_4)\ge2\mathscr V(0,J_4)-V_0\). Thus, the comparison also covers unit lots under nonnegative setup costs.\Halmos

The role of the zero-setup value can be seen directly. Dividing both cycles reduces inventory dispersion, but it changes the frequency of setups. The even--odd calculation accounts for both changes exactly. The convexity of the joint inventory value bounds the inventory improvement relative to the inventory cost that remains when both cycles have size one. Indeed, \(\mathcal J(0,0)=V_0\): Lemma~\ref{appjc:scale-lemma} bounds every cycle pair below by \(V_0\), and the pair \((1,1)\) attains it.

\proof{Theorem~\ref{thm:five-thirds}: comparison for positive holding and backlog rates.} Proposition~\ref{prop:two-scale-lower} gives
\begin{equation}
 L_{\mathrm{two}}=
 \max\left\{\frac{B_1(2k_1)+B_2(2k_2)}2,
 \frac{V_0}{2}+\frac{B_1(4k_1)+B_2(4k_2)}4\right\}
 \le\LBsp(k)\le\OPTsh.
 \label{appjc:two-scale-value}
\end{equation}
Combining \eqref{appjc:separate-upper} and \eqref{appjc:J-doubling}, their respective lower estimates are
\begin{align}
 L_{\mathrm{two}}&\ge\tfrac12\mathcal J(2k)
       \ge\tfrac23\mathcal J(k)-\tfrac16V_0,\label{appjc:first-scale}\\
 L_{\mathrm{two}}&\ge\tfrac12V_0+\tfrac14\mathcal J(4k)
       \ge\tfrac12\mathcal J(k)+\tfrac14V_0.\label{appjc:second-scale}
\end{align}
Take \(3/5\) of the first estimate and two fifths of the second. Their coefficients of \(V_0\) cancel, giving \(L_{\mathrm{two}}\ge3\mathcal J(k)/5\). Lemma~\ref{appjc:policy-lemma} now yields
\begin{equation}
 \inf_{\Pi\in\mathcal P_{\RnQ}}\Clot(\Pi)
 \le\mathcal J(k)\le\frac53L_{\mathrm{two}}
 \le\frac53\LBsp(k)\le\frac53\OPTsh.
 \label{appjc:final-comparison}
\end{equation}
For each constructed policy, \(\Csh\le\Clot\), so the shipment-accounted comparison follows as well. This proves the cost-infimum comparison under the standing positive-rate assumptions. Appendix~\ref{app:joint-boundary} proves finite attainment here and covers the remaining parameters.\Halmos\endproof

The two estimates complement each other: the first has a larger coefficient on the joint value and subtracts part of the unavoidable inventory cost; the second restores that cost. Their weighted combination removes the need to divide instances into cases according to the relative importance of setup and inventory costs. Equal allocation and these two scales suffice.

\subsection{Finite Attainment and Zero-Cost Cases}\label{app:joint-boundary}

On the positive holding and backlog region, an actual finite policy satisfies \eqref{appjc:final-comparison}. We first verify that the independent-cycle infimum is attained. The choice \(x=x(D_2,J_{m_2})\) is independent of \((D_1,J_{m_1})\), so conditional on \(x\), its expected mismatch cost is \(\E_{D_1,J_{m_1}}\ph_{a,p}(x-D_1-J_{m_1})\ge\mathcal H_{m_1}(P)\). Dropping the nonnegative surplus term in \eqref{appjc:joint-value}, therefore, gives \(V_{m_1,m_2}\ge H_0+\mathcal H_{m_1}(P)\). For the upstream cycle size, let \(d=z-D_2-J_{m_2}-x\ge0\). Since \(\ph_{h_2,p}\) has upper slope \(h_2\le a\), \(\ph_{a,p}(x-D_1-J_{m_1})+h_2d \ge\ph_{h_2,p}(z-D_2-J_{m_2}-D_1-J_{m_1})\). Replacing the independent \(D_1+J_{m_1}\) by its mean in this convex lower bound, followed by a shift of \(z\), therefore, gives \(V_{m_1,m_2}\ge H_0+\mathcal H_{m_2}(F_{\mathrm{up}})\). Both lower bounds grow with their own cycle sizes. Explicitly, \(\ph_{c,d}(x)\ge\min\{c,d\}|x|\), and for a position chosen uniformly from \(j\) consecutive integers, minimizing its expected absolute distance over all fixed real points gives a value of at least \((j-1)/4\). Replacing demand by its mean consequently gives
\begin{equation}
 V_{m_1,m_2}\ge H_0+\frac{\min\{a,p\}}4(m_1-1),\qquad
 V_{m_1,m_2}\ge H_0+\frac{\min\{h_2,p\}}4(m_2-1).
 \label{appjc:finite-lengths}
\end{equation}
If an auxiliary candidate has finite total cost \(C\), the nonnegative setup costs and \eqref{appjc:finite-lengths} place finite upper bounds on both sizes of every candidate costing at most \(C\). There are only finitely many such integer pairs. Their placement objectives grow without bound in both directions and attain finite integer minima, as established after \eqref{appjc:inventory-cycle-reduction}, so \(\mathcal J(k)\) has a minimizing pair and the group construction selects a finite deterministic policy satisfying the bound. Zero setup coefficients do not affect this argument.

\paragraph{A finite construction of the selected policy.} The preceding existence argument yields an explicit selection rule on the positive holding and backlog region. We describe its bounds because the policy comparison involves both cycle sizes and their placements. Set \(\mu_i=\E D_i\). In the joint problem, choose cycles of size one, \(z=0\), and \(x=-D_2\). This is feasible and gives the reference total cost
\begin{equation}
 C_*=k_1+k_2+H_0+p(\mu_1+\mu_2),\qquad B_*=C_*-H_0.
 \label{appjc:finite-reference}
\end{equation}
It is also the cost of the finite unit-lot policy with both reorder points equal to \(-1\). In particular, \(\mathcal J(k)\le C_*\). By \eqref{appjc:finite-lengths} and nonnegative setups, any auxiliary candidate costing at most \(C_*\) has
\begin{equation}
 1\le m_1\le m_{1,*}=\left\lfloor1+\frac{4B_*}{\min\{a,p\}}\right\rfloor,
 \qquad
 1\le m_2\le m_{2,*}=\left\lfloor1+\frac{4B_*}{\min\{h_2,p\}}\right\rfloor.
 \label{appjc:finite-lot-search}
\end{equation}
Every excluded pair costs strictly more than the reference.

For each retained downstream size, its threshold can also be found within a specified finite set. Put \(X_{m_1}=D_1+J_{m_1}\) and \(q_*=(p+h_2)/(a+p)\), which lies strictly between zero and one. The first difference of the convex threshold objective is
\begin{equation}
 [G_{m_1}(r+1)-h_2(r+1)]-[G_{m_1}(r)-h_2r]
 =(a+p)\Pr(X_{m_1}\le r)-(p+h_2).
 \label{appjc:threshold-difference}
\end{equation}
Choose the smallest nonnegative integer \(r_{m_1}\) for which \(\Pr(X_{m_1}\le r_{m_1})\ge q_*\). The previous difference is negative and the current difference is nonnegative, so this choice minimizes the objective. Equality in the current difference simply permits an additional minimizing threshold. To see that the search has a finite explicit stopping point, Markov's inequality gives \(\Pr(X_{m_1}>r)\le\E X_{m_1}/(r+1)\). It is, therefore, sufficient to inspect integers through \(r_{m_1}^{\max}=\lceil(a+p)(\mu_1+(m_1-1)/2)/h_1\rceil\). This also covers \(X_{m_1}=0\), when the selected threshold is zero. The distribution function of \(X_{m_1}\) is a finite average of Poisson distribution functions.

After computing \(r_{m_1}\), write the inventory cost at endpoint \(z\) as
\[
 \mathcal I_{m_1,m_2}(z)=G_{m_1}(r_{m_1})+\E F_{m_1}(z-r_{m_1}-J_{m_2}),\qquad
 \mu_{m_1,m_2}=\mu_1+\mu_2+\frac{m_1+m_2-2}{2}.
\]
The pointwise inequality preceding \eqref{appjc:finite-lengths}, now retaining the endpoint, and averaging inside the convex lower bound imply
\begin{equation}
 \mathcal I_{m_1,m_2}(z)\ge H_0+
 \E\ph_{h_2,p}(z-D_1-J_{m_1}-D_2-J_{m_2})
 \ge H_0+\ph_{h_2,p}(z-\mu_{m_1,m_2}).
 \label{appjc:finite-endpoint-bound}
\end{equation}
Consequently, every candidate that can improve the reference has an integer endpoint in
\begin{equation}
 \left\lceil\mu_{m_1,m_2}-B_*/p\right\rceil
 \le z\le
 \left\lfloor\mu_{m_1,m_2}+B_*/h_2\right\rfloor.
 \label{appjc:finite-endpoint-search}
\end{equation}
The lower bound controls backlog from a low placement; the upper bound controls holding from a high placement. An empty integer interval means that this pair cannot improve the reference and may be omitted. Using the same \(B_*\) for every pair makes the statement simple; subtracting that pair's setup terms from the available cost only tightens the interval.

Evaluate \(k_1/m_1+k_2/m_2+\mathcal I_{m_1,m_2}(z)\) on the finite collection in \eqref{appjc:finite-lot-search} and \eqref{appjc:finite-endpoint-search}, and retain a minimizing triple \((m_1,m_2,z)\). The unit pair has a candidate costing at most \(C_*\) inside these bounds. Every candidate outside them costs more than \(C_*\). Thus, the minimum over the retained collection is exactly \(\mathcal J(k)\). This conclusion includes possible ties in sizes, thresholds, and endpoints; any minimizing triple may be retained. It also establishes a finite exclusion of all omitted lot sizes and placements, independently of a numerical stopping convention.

Finally form every nonempty remainder group for that triple. For each \(b\), compute \(g_b\), set \(\ell_b=z-r_{m_1}-b-m_1(g_b-1)\), and evaluate the actual policy cost \eqref{appjc:physical-cost}. Retain the row with the smallest complete cost. This last comparison is necessary because groups of equal size can have different placements and different inventory costs, as \eqref{appjc:worked-groups} illustrates. The weighted cost comparison in Lemma~\ref{appjc:policy-lemma} proves that the selected row costs at most \(\mathcal J(k)\). It does not require identifying in advance which row satisfies the inequality. All controls are now explicit integers, and the stationary construction in Appendix~\ref{app:policy-cost} gives their feasible implementation.

The bounds describe a finite mathematical construction for each positive-cost instance. They can be large when one holding or backlog rate is small. Their purpose here is to justify completion of the search and the existence of a selected policy; a practical computation can use tighter verified bounds. If a cost rate reaches zero, the divisions used to bound a threshold, a lot size, or an endpoint may no longer be available. The direct boundary argument below then selects finite controls within a prescribed cost tolerance, preserving feasibility before taking a limit.

Now allow arbitrary nonnegative \(h_1,h_2,p\), still with \(\lambda>0\). With zero holding or backlog rates, a placement or lot-size infimum may be approached only along a sequence, so the comparison is stated in terms of infima. The conditional downstream optimization continues to have a finite threshold once an integer \(z\) and finite sizes are chosen. To verify this, put \(P_{m_1}(x)=\E P(x-J_{m_1})\) and consider \(P_{m_1}(x)-h_2x\), whose limiting slopes are \(-(p+h_2)\) and \(h_1\).

If \(h_1>0\) and \(p+h_2>0\), the function grows without bound in both directions, and an integer minimizing threshold exists. If \(h_1=0\), it is nonincreasing, so the constrained optimum is \(x=y\) at every available level \(y=z-D_2-J_{m_2}\). Because demand and the uncentered cycle positions are nonnegative, \(y\le z\). The finite threshold \(r=z\), therefore, implements this choice as \(x=\min\{r,y\}\) for every realization. Finally, if \(p+h_2=0<h_1\), then \(p=h_2=0\), and \(P_{m_1}(x)=0\) for \(x\le0\). The finite threshold \(r=0\) gives an optimum for every \(y\). The second case includes all-zero rates.

The optimized objective is convex and linear between consecutive integers in \(z\). Integer thresholds and loss breakpoints establish this in the first case. In the second, the objective is \(H_0+\E\ph_{h_2,p}(z-D_1-J_{m_1}-D_2-J_{m_2})\). In the third case, its inventory cost is zero. Thus, the infimum over real \(z\) equals the infimum over integers, including when it is approached as \(z\) tends to infinity. Given any positive tolerance, a finite integer \(z\) approaches this value, and the associated finite threshold realizes its conditional optimum. These choices are integrable, so the restriction in \eqref{appjc:joint-value} remains valid at the boundary.

With that threshold and endpoint, grouping preserves average inventory cost, and \eqref{appjc:setup-group-average} holds for nonnegative setups. For any \(\delta>0\), choose the finite endpoint within \(\delta\) of \(V_{m_1,m_2}\). Some group's policy then costs at most \(k_1/m_1+k_2/m_2+V_{m_1,m_2}+\delta\). Choose a finite pair with total cost within another \(\delta\) of \(\mathcal J(k)\). Letting \(\delta\downarrow0\) proves the infimum comparison \eqref{appjc:policy-bound}.

The zero-cycle value agrees with the boundary benchmark: \(\mathscr V(0,0)=\inf_z\E\mathcal E_0(z-D_2)=V_0\). Indeed, the conditional minimization in \eqref{appjc:joint-value} is exactly the constrained cost \(\mathcal E_0\) in \eqref{eq:app-zero-envelope}; Proposition~\ref{prop:app-zero-value} establishes its equality with \eqref{eq:lb-protected-zero-value} for all nonnegative rates. For each finite \(z\), the finite caps just described attain that conditional minimum. At scale zero, allowing the choice to observe the independent cycle position \(\Xi_2\) leaves its conditional cost at least \(\mathcal E_0(z-D_2)\), and the same finite cap attains this value without using \(\Xi_2\).

The shifts in Lemma~\ref{appjc:scale-lemma} preserve the feasible choices for nonnegative rates. Its convexity proof also applies: choose each initial feasible pair within \(\delta\) of its infimum, average the pairs, and let \(\delta\downarrow0\). Their averages remain integrable, and the preceding paragraph identifies the value at scale zero with \(V_0\). Conditioning on the finite-support cycle positions similarly gives \(\mathscr V(\Xi_1,\Xi_2)\ge V_0\). Thus, both inequalities in \eqref{appjc:inventory-doubling} remain valid. The parity calculation uses only these inequalities, nonempty conditional sets, and nonnegative setup coefficients; the concavity calculation uses only the definition as an infimum of affine functions. They prove \eqref{appjc:J-doubling} at the boundary as well.

For the stagewise upper bound, fix \((m_1,m_2)\) and \(\delta>0\). Choose finite integer placements \(s,y\) satisfying \(G_{m_1}(s)\le H_0+\mathcal H_{m_1}(P)+\delta\) and \(\E F_{\mathrm{up}}(y-J_{m_2})\le\mathcal H_{m_2}(F_{\mathrm{up}})+\delta\). Such choices exist because the placement functions are finite and linear between successive integers. Use \(z=s+y\) and \(x=\min\{s,z-D_2-J_{m_2}\}\) in \eqref{appjc:joint-value}. The same slope bound used in \eqref{appjc:separate-upper} gives
\[
 V_{m_1,m_2}\le G_{m_1}(s)+\E F_{\mathrm{up}}(y-J_{m_2})
 \le H_0+\mathcal H_{m_1}(P)+\mathcal H_{m_2}(F_{\mathrm{up}})+2\delta.
\]
Letting \(\delta\downarrow0\), adding the setup rates, and taking the infimum over the sizes proves both inequalities in \eqref{appjc:separate-upper}. The equal-allocation comparison \eqref{appjc:equal-allocation} uses only nonnegative loss coefficients and the convexity of the loss functions, so it also holds at these rates.

Appendices~\ref{app:lower-bounds} and \ref{app:zero-setup} establish the required lower-bound validity for these nonnegative costs. Therefore, \eqref{appjc:final-comparison} continues to hold for cost infima. More explicitly, for every \(\varepsilon>0\), there is a finite deterministic classical policy \(\Pi_\varepsilon\) with
\begin{equation}
 \Csh(\Pi_\varepsilon)\le\Clot(\Pi_\varepsilon)\le\frac53L_{\mathrm{two}}+\varepsilon\le\frac53\LBsp(k)+\varepsilon\le\frac53\OPTsh+\varepsilon.
 \label{appjc:boundary-epsilon}
\end{equation}
This formulation also covers a zero optimal value and does not claim attainment when only a sequence of finite lot sizes or placements approaches the bound.

Zero lead times cause no change: the corresponding Poisson demand is identically zero, and Appendix~\ref{app:policy-cost} provides the same-epoch settlement convention for the finite policy. If \(\lambda=0\), choose no orders and the empty initial system. There is no demand, inventory, backlog, or setup cost; both sides of the comparison are zero. These cases complete the stated parameter domain.
\section{Proof of Theorem~\ref{thm:impossibility}: the Policy-Class Impossibility Result}\label{app:impossibility}

Throughout the boundary construction, demand has rate one, $p=1$, $h_1=K_2=L_1=0$, $L_2>0$, $h_2>0$, and $K_1=k>0$. Write $D_2\sim\operatorname{Poisson}(L_2)$, $\phi(x)=h_2x^++x^-$, and $F_{\mathrm{imp}}(z)=\mathbb E\phi(z-D_2)$. Let $X$ denote downstream net inventory, $W\ge0$ available upstream stock, and $Y=X+W$ total net echelon inventory. The physical running cost is
\begin{equation}
 c(X,W)=h_2X^++X^-+h_2W=h_2Y+(1+h_2)X^-\ge\phi(Y).
 \label{eq:imp-physical-floor}
\end{equation}
We retain integer quantities throughout. As in \eqref{eq:impossibility-scalar}, \(\sigma \) denotes the lowest inventory position in a cycle. All limits below concern a sequence of finite instances.

To verify the inequality, use $X\le Y$, which follows from $W\ge0$. Thus, $X^-\ge Y^-$, and $h_2Y+(1+h_2)Y^-=\phi(Y)$. This comparison allows any allocation of physical stock between the locations. The equality holds whenever all available upstream material is used to remove downstream backlog. That is precisely the property enforced by the competing policy below.

\subsection{The Exact Classical Numerator Under Shipment Accounting}

\begin{lemma}\label{lem:imp-classical-equality}
On this boundary, the infima of shipment-accounted and lot-accounted classical costs both equal $J_{\mathrm{imp}}(k)$ in \eqref{eq:impossibility-scalar}.
\end{lemma}

\proof{Proof.} Fix classical controls with $Q_2=nQ_1$. After an event has been settled, any unfilled downstream request implies that the available upstream stock is less than $Q_1$. At a demand epoch, the nominal counter creates at most one new $Q_1$-request. If older requests remain, the pre-event upstream stock is below $Q_1$, so demand alone cannot release them. Otherwise, the new request releases at most $Q_1$ units. At an upstream receipt, exactly $Q_2$ units arrive. If requests were pending, the pre-receipt remainder is less than $Q_1$, and \(\lfloor(W+Q_2)/Q_1\rfloor\le n\). If no requests were pending, the receipt creates none. Thus, every positive downstream aggregate dispatch is at most $Q_2$. Distinct upstream orders have distinct receipt times, and a demand and an upstream receipt coincide with probability zero because $L_2>0$ is deterministic and demand is Poisson. Any finite initial clearing contributes only a boundary term that approaches zero. Rate balance consequently gives an actual downstream shipment rate of at least $1/Q_2$.

The upstream inventory position cycles through $R_2+1,\ldots,R_2+Q_2$. The demand residue is uniform in the stationary distribution, and the position at time $t-L_2$ is independent of demand during $(t-L_2,t]$. Therefore, $Y(t)$ has the distribution $R_2+1+J_{Q_2}-D_2$, where $J_{Q_2}$ is uniform on $\{0,\ldots,Q_2-1\}$ and independent of $D_2$. Equation~\eqref{eq:imp-physical-floor} and the shipment-rate bound imply
\[
 \Csh(\Pi)\ge\frac{k}{Q_2}
 +\frac1{Q_2}\sum_{j=0}^{Q_2-1}F_{\mathrm{imp}}(R_2+1+j)\ge J_{\mathrm{imp}}(k).
\]
The convergence argument in Lemma~\ref{lem:rnq-admissible} gives the same long-run bound for any fixed initial demand residue: its distribution converges to uniform, the lead-time-window moment bounds control the expected running cost, and averaging over time preserves the limit. The same bound holds for lot accounting. Conversely, choose $Q_1=Q_2=m$ and $R_1=R_2=\sigma -1$. The upstream and downstream request counters coincide. Every upstream receipt fills one waiting downstream lot immediately, so $W=0$, $X=Y$, and each stage processes one $m$-lot per $m$ demands. Both downstream setup conventions give $k/m$, and the running cost is the block average in \eqref{eq:impossibility-scalar}. Take infima to conclude.\Halmos\endproof

The dispatch cap applies to the aggregate physical movement at an epoch, including every simultaneously filled downstream request. It, therefore, already includes any consolidation savings available under shipment accounting. The inventory lower bound uses the upstream lot size $Q_2$, irrespective of the downstream reorder point or the inventory retained upstream. These two observations allow the infimum to range over all integer ratios $n$. The upstream cycle position is uniform because every state of the finite demand-residue chain advances to the next at rate one, giving equal invariant probabilities. Its long-run averages are independent of the initial residue.

\subsection{A Feasible Policy with Frequent Upstream Replenishment}

Fix an integer $S\ge10$. Order one unit from the external supplier at each demand and maintain upstream base-stock position $S$. Dispatch all available upstream stock whenever a demand makes $X<0$, and dispatch each arriving unit whenever its receipt finds $X<0$. Otherwise retain received units upstream. This is an adapted rule using available integer material. It maintains $X^-=Y^-$ after each event: if $X<0$, all available upstream stock has been dispatched, while $Y\ge0$ implies that enough stock has been available to remove backlog. Hence,
\begin{equation}
 Y(t)=S-D(t-L_2,t],\qquad I=F_{\mathrm{imp}}(S).
 \label{eq:imp-policy-inventory}
\end{equation}
Moreover, $X^+\le S$, $W\le S$, and $X^-\le D(t-L_2,t]$. These bounds give finite moments of inventory and cost on every fixed time window.

The flow balances make the source of \eqref{eq:imp-policy-inventory} explicit. Let $N(t)$ count demand, let $U_1(t)$ count units dispatched downstream, and choose compatible initial constants for the stationary history. Upstream orders have cumulative quantity $U_2(t)=N(t)$, up to that initial constant, while upstream receipts at time $t$ reflect orders through $t-L_2$. Consequently,
\[
 \begin{aligned}
 Y(t)-Y(0)&=N(t-L_2)-N(-L_2)-[N(t)-N(0)],\\
 W(t)-W(0)&=N(t-L_2)-N(-L_2)-[U_1(t)-U_1(0)].
 \end{aligned}
\]
The first balance describes total inventory protection; the second describes the material accumulated for consolidation. Since the expected change in $W$ is bounded, the second identity gives $\mathbb E[U_1(t)-U_1(0)]/t\to1$. The policy can change the number of downstream shipments while preserving the number of units eventually delivered.

A stationary version can be constructed from past demand. A demand-free interval of length $L_2$ empties the upstream pipeline and leaves $Y=S$. If the next interval of length $L_2/2$ contains $S+1$ demands, then at its last demand no order caused by that interval has yet arrived. The rule has exhausted the $S$ available units, yielding $X=Y=-1$ and $W=0$, regardless of their allocation before the interval. Such patterns occur in disjoint blocks with a fixed positive probability. Every compatible starting allocation, therefore, reaches the same state after such a block. This gives a construction determined by past demand and shows that a version started from a finite initial state has the same average costs.

More explicitly, use blocks of length $2L_2$, requiring silence in the first $L_2$ units and exactly $S+1$ demands in the next $L_2/2$ units. Their probability is $e^{-3L_2/2}(L_2/2)^{S+1}/(S+1)!>0$, and the events in distinct blocks are independent. After the indicated burst, both the allocation and the remaining upstream pipeline are determined by the observed burst alone. A most recent successful block exists almost surely in a two-sided demand history. Starting from the state determined by that block defines the present allocation using only past observations. Starting from any compatible finite allocation gives the same state after the first successful block. Inventory and backlog have uniformly bounded moments because stock is bounded by $S$ and backlog by Poisson demand during a lead time. Thus, the expected average cost before that coupling approaches zero as the evaluation horizon grows.

The upstream flow has rate one. Since $W$ is bounded, downstream dispatched quantity also has rate one; bounded mean backlog then gives customer delivery rate one. Terminal inventories and pipelines have bounded expectations, and expected costs and quantities over a fixed terminal window are bounded. The inventory-position paths also satisfy the unfinished-path condition required for admissibility. The upstream position resets after every demand. The downstream position satisfies $-D(t-L_2,t]\le X(t)\le S$, where $D(t-L_2,t]$ counts demand during the preceding lead time. Removing completed cycles leaves distinct integer levels, so the number of remaining downward transitions is bounded by the range of levels visited. For $L_2\le1$, choose $j=\lfloor t\rfloor-1$. The lead-time window $(t-L_2,t]$ then lies inside the two-unit interval $(j,j+2]$, so $D(t-L_2,t]\le D(j,j+2]$. The expected maximum of these counts up to time $T$ is at most the square root of their summed second moments, which is $O(\sqrt{T+1})$. Division by $T$ gives the required limit.

An explicit estimate verifies the unfinished-path condition. Put $C_j=D(j,j+2]$, so $\mathbb E C_j^2=6$. The enclosing intervals with $-1\le j\le\lceil T\rceil+1$ cover every lead-time window encountered before $T$. No independence between overlapping counts is needed: $\max_j C_j\le(\sum_j C_j^2)^{1/2}$. If $r_{1,T}$ counts the remaining downstream demand transitions, including intermediate event states gives the conservative bound \(\mathbb E r_{1,T}\le S+1+\sqrt{6(\lceil T\rceil+3)}=o(T)\). Upstream unit orders close their demand transitions immediately. Thus, both inventory-position paths satisfy the unfinished-path condition, in addition to the physical flow and finite-window cost conditions.

Let $\nu_D$ and $\nu_R$ be the rates of demand-triggered and receipt-triggered downstream shipments, respectively. At a demand time $t$ selected for a positive shipment, the pre-demand value is $X(t-)=0$. Write $D(a,b)=N(b-)-N(a)$ to exclude demand at the right endpoint. With $N^-(t)=D(t-L_2,t)$, the available shipment is $S-N^-(t)$. Define $\Gamma$ as the expected sum of $N^-(t)$ over these selected demand times per unit time. This sum is bounded by the corresponding sum over all demand times. Since $N^-(t)$ depends only on past demand, averaging it over Poisson demand arrivals gives its mean $L_2$. Demand arrives at rate one, so the expected sum per unit time is also $L_2$. Hence, $0\le\Gamma\le L_2$.

A receipt-triggered shipment contains one unit. For each upstream order, conditional on the history when it is placed, the additional demand before its receipt has the Poisson distribution with mean $L_2$. At the receipt time $t$, that unit has just left the pipeline. If the receipt triggers a dispatch, the downstream net inventory immediately before the dispatch is negative and the upstream stock was zero before the receipt. The postreceipt echelon level is, therefore, at most zero. It equals $S$ minus the other demands during the lead-time window, so at least $S$ such demands are necessary. The rate of all upstream receipts is one. Applying this necessary condition to that receipt stream gives $\nu_R\le t_S$, where $t_S=\Pr(D_2\ge S)$. Both rate bounds use the full demand or receipt stream before applying the policy's dispatch condition.

The shipped-unit balance is \(1=S\nu_D-\Gamma+\nu_R\): demand-triggered shipments contribute rate $S\nu_D-\Gamma$, and receipt-triggered shipments contribute rate $\nu_R$. Adding the two shipment rates proves
\begin{equation}
 \frac1S\le\nu=\nu_D+\nu_R
 =\frac1S+\frac\Gamma S+\left(1-\frac1S\right)\nu_R
 \le\frac{1+L_2}{S}+\left(1-\frac1S\right)t_S.
 \label{eq:imp-shipment-rate}
\end{equation}

These rate arguments also have finite-horizon expressions. If $\tau_j$ are the demand epochs after time zero, independent future increments at $\tau_j$ give
\[
 \mathbb E\sum_{j:\,\tau_j\le T-L_2}
 \mathbf 1\{D(\tau_j,\tau_j+L_2]\ge S\}
 =t_S\,\mathbb E N(T-L_2)=t_S(T-L_2)
 \quad(T\ge L_2).
\]
Every receipt-triggered shipment generated by these orders is counted on the left. Initial-pipeline receipts contribute a bounded expectation. Similarly, because $N^-(t)$ depends only on past demand, the expected sum of $N^-(t)$ over all demands in $(0,T]$ equals $\int_0^T\mathbb E N^-(t)\,dt=L_2T$ in stationarity. Both bounds, therefore, allow inventory-dependent shipment epochs.

\subsection{Poisson Bounds and the Full Classical Policy Optimization}

Set $L_2=1/S$, $h_2=L_2^S/S!$, and $q_{\mathrm{tail}}=L_2/(S+1)=1/[S(S+1)]$. For $j\ge0$, $\Pr(D_2=S+j)\le h_2q_{\mathrm{tail}}^j$.
\begin{equation}
 t_S\le\frac{h_2}{1-q_{\mathrm{tail}}},\qquad
 \frac{\mathbb E(D_2-S)^+}{h_2}\le\frac{q_{\mathrm{tail}}}{(1-q_{\mathrm{tail}})^2},\qquad
 S-\frac1S\le\frac{I}{h_2}\le S-\frac1S+\frac{2q_{\mathrm{tail}}}{(1-q_{\mathrm{tail}})^2}<S.
 \label{eq:imp-tail-estimates}
\end{equation}
Here we used $h_2<1$ and $F_{\mathrm{imp}}(S)=h_2(S-L_2)+(1+h_2)\mathbb E(D_2-S)^+$. The last strict inequality follows from $2q_{\mathrm{tail}}/(1-q_{\mathrm{tail}})^2<1/S$ for $S\ge10$. Since $Sh_2\to0$, \eqref{eq:imp-shipment-rate} yields $S\nu\to1$, while \eqref{eq:imp-tail-estimates} gives $I/(h_2S)\to1$.

\begin{lemma}\label{lem:imp-global-cycle-limit}
For fixed $0<\alpha\le1/2$ and $k=\alpha h_2S^2$,
\begin{equation}
 (1+\sqrt{2\alpha})S-\frac52-\frac1S
 \le\frac{J_{\mathrm{imp}}(k)}{h_2}\le\frac{I}{h_2}+\sqrt{2\alpha}\,S.
 \label{eq:imp-global-cycle-bounds}
\end{equation}
In particular, $J_{\mathrm{imp}}(k)/(h_2S)\to1+\sqrt{2\alpha}$.
\end{lemma}

\proof{Proof.} The upper slope of $F_{\mathrm{imp}}$ is at most $h_2$, so the block starting at $S$ gives
\[
 J_{\mathrm{imp}}(k)\le I+\inf_{m\in\mathbb N}\left\{\frac{k}{m}+\frac{h_2(m-1)}2\right\}
 \le I+\sqrt{2kh_2}.
\]
For the last inequality put $x=\sqrt{2k/h_2}$ and $m=\lceil x\rceil$. Writing $d=m-x\in[0,1)$, the difference between $x^2/(2m)+(m-1)/2$ and $x$ is $d^2/(2m)-1/2\le0$. For a reference candidate, choose lot size $S$ and the block starting at $S$. Since $F_{\mathrm{imp}}(S+j)\le I+h_2j$, $I<h_2S$, and $\alpha\le1/2$,
\[
 J_{\mathrm{imp}}(k)\le\frac{k}{S}+\frac1S\sum_{j=0}^{S-1}F_{\mathrm{imp}}(S+j)
 \le\alpha h_2S+I+\frac{h_2(S-1)}2<2h_2S.
\]

We now exclude all lot sizes and placements that could escape the claimed lower bound. Because $\phi(x)\ge h_2|x|$, the average loss over any $m$ consecutive positions is at least $h_2(m-1)/4$. Indeed, pairing the two ends of an equally spaced block and using the triangle inequality gives a minimum average absolute deviation of $\lfloor m^2/4\rfloor/m\ge(m-1)/4$; the estimate holds for every realized demand and every placement of the block. Thus, every lot size $m\ge10S$ costs more than the reference candidate. For $m<10S$, a block whose lowest integer position satisfies $\sigma \le S-3$ incurs at least $\Pr(D_2=S-2)/m$ in shortage cost: on $D_2=S-2$, that position is short by at least $S-2-\sigma\ge1$, and it has weight $1/m$. Moreover, \(\Pr(D_2=S-2)/h_2=e^{-1/S}S^3(S-1)>S^2(S-1)^2\). After division by $10S$, the right side exceeds $2S$ for $S\ge10$. These placements cost more than the reference candidate.

Every remaining block has $\sigma \ge S-2$. Since $F_{\mathrm{imp}}(z)\ge h_2(z-L_2)$, its cost divided by $h_2$ is at least
\[
 S-2-\frac1S+\frac{\alpha S^2}{m}+\frac{m-1}{2}
 \ge(1+\sqrt{2\alpha})S-\frac52-\frac1S.
\]
The last inequality follows by writing the setup and cycle-inventory terms as a square:
\[
 \frac{\alpha S^2}{m}+\frac m2-\sqrt{2\alpha}\,S
 =\left(\frac{\sqrt\alpha\,S}{\sqrt m}-\frac{\sqrt m}{\sqrt2}\right)^2\ge0.
\]
The square measures the cost of choosing a lot size above or below the scale that balances the two terms. Its nonnegativity holds for every positive integer $m$, so no continuous optimizer is substituted for the physical lot size. The excluded blocks exceed $2S$, whereas the displayed lower bound is below $2S$ because $\alpha\le1/2$. Thus, the bound covers the infinite optimization domain and proves the lemma.\Halmos\endproof

The probability mass at $S-2$ has a specific role. Although this mass is small, its ratio to $h_2$ grows as a polynomial of degree four in $S$. It prevents a classical policy from avoiding cycle inventory by moving the entire block to low positions with substantial backlog. The separate large-lot exclusion ensures that spreading this shortage contribution across an arbitrarily large cycle cannot circumvent that restriction. Thus, the argument retains both the integer placement decision and the unbounded lot-size decision before applying the elementary setup--inventory tradeoff.

\subsection{The Ratio, a Finite Example, and Positive Primitives}

The feasible policy has cost $I+k\nu$, so the preceding limits give
\[
 \frac{J_{\mathrm{imp}}(\alpha h_2S^2)}{I+\alpha h_2S^2\nu}
 \longrightarrow\frac{1+\sqrt{2\alpha}}{1+\alpha}.
\]
With $x=\sqrt{2\alpha}$, differentiation gives a numerator proportional to $1-x-x^2/2$. The maximum on $0<\alpha\le1/2$ is attained at $x=\sqrt3-1$, or $\alpha=2-\sqrt3$, and equals $\rho_\circ$. For each $r<\rho_\circ$, choose a sufficiently large finite $S$. Lemma~\ref{lem:imp-classical-equality} and $\OPTsh\le I+k\nu$ then prove \eqref{eq:impossibility-ratio}. The optimal cost $\OPTsh$ is positive: every admissible policy costs at least $\min_zF_{\mathrm{imp}}(z)$, the minimum mismatch cost for a single stage with lead time $L_2$. This minimum is positive because $h_2>0$ and the Poisson demand $D_2$ is nondegenerate.

Set $S=1000$ and $k/h_2=267949$ for the finite example with rational parameters. Then $\alpha=267949/10^6<1/2$, so all global exclusions apply. The integer minimum of $267949/m+(m-1)/2$ occurs at $m=732$, since \((732\cdot731)/2\le267949\le(732\cdot733)/2\). Using the retained placement bound $\sigma \ge998$ gives
\begin{align}
 \frac{J_{\mathrm{imp}}(k)}{h_2}&\ge998-\frac1{1000}+\frac{267949}{732}+\frac{731}{2}
 =\frac{316507567}{183000},\label{eq:imp-finite-numerator}\\
 \frac{I+k\nu}{h_2}&<1000+267949\left(\frac{1001}{10^6}+\frac1{10^{12}}\right)
 =1268.216949267949.\label{eq:imp-finite-denominator}
\end{align}
The tail contribution in \eqref{eq:imp-shipment-rate} is below $10^{-12}$ because $h_2/(1-q_{\mathrm{tail}})<10^{-12}$. Dividing the two rational bounds gives a number greater than $1.3637647$ and, hence, greater than $1.36$. No truncated Poisson evaluation or simulated policy cost is used.

The integer minimizer can be checked without searching lot sizes. If $f(m)=267949/m+(m-1)/2$, then $f(m+1)-f(m)=1/2-267949/[m(m+1)]$. This difference changes sign at $m=732$, by the two inequalities above. Even coarse rounding of the proved bounds suffices to verify the requested threshold: the numerator exceeds $1729$, and the denominator is below $1269$, while \(1729/1269>34/25=1.36\) and \(25\cdot1729-34\cdot1269=79>0\). These small integers already prove impossibility at $1.36$.

Finally, fix a finite boundary instance with strict ratio greater than $r$. Introduce $h_1=\varepsilon>0$, $K_2=\eta>0$, and $L_1=\delta>0$, retaining every other primitive. Maintain an internal copy of the boundary policy, called the shadow policy, whose downstream lead time is zero. Use its inventory state to reproduce its upstream orders and downstream dispatches. Upstream stock availability and the dispatch decisions are unchanged; physical downstream deliveries occur $\delta$ later. The construction remains adapted and materially feasible. If $P_1(t)$ is the quantity dispatched during $(t-\delta,t]$, actual downstream net inventory is $X(t)-P_1(t)$. Its backlog exceeds the shadow backlog by at most $P_1(t)$, and the total net echelon inventory $Y$ is unchanged. The mean pipeline quantity is $\delta$ by unit flow balance. Downstream physical inventory is bounded by $S$; the weaker bound $S+P_1(t)$ also covers inventory assigned to that echelon while in transit. Therefore, the new policy cost is at most
\begin{equation}
 I+k\nu+\eta+\varepsilon(S+\delta)+(1+h_2)\delta.
 \label{eq:imp-positive-perturbation}
\end{equation}
For the cost estimate, write $B_0=X^-$ and $B_\delta=(X-P_1)^-$. Then $0\le B_\delta-B_0\le P_1$, and Stage-2 physical inventory changes from $Y+B_0$ to $Y+B_\delta$. The increase in the original holding and backlog charges is consequently at most $(1+h_2)P_1$. Each unit dispatched downstream spends exactly $\delta$ time in its new pipeline. Integrating pipeline quantity over time, therefore, contributes $\delta$ per dispatched unit, apart from the initial and terminal windows. Unit flow balance and their bounded expectations give $\mathbb E P_1=\delta$ in stationarity. Upstream setups add $\eta$ because one upstream order is placed per demand; the downstream dispatch sequence has its original setup cost $k\nu$.

The perturbed policy has the same inventory-position paths as the shadow policy, so it satisfies the same unfinished-path condition. The additional fixed-window pipelines have finite moments, and all the other admissibility conditions established above continue to hold. For every classical policy, the dispatch-size bound remains $Q_2$, and its upstream echelon marginal law is unchanged. With $Y=\mathit{IL}_2$, material feasibility gives $\mathit{IL}_1\le\mathit{IP}_1\le Y$. Hence, $\mathsf B=\mathit{IL}_1^-\ge Y^-$. Its physical running cost, therefore, satisfies \(c_I\ge h_2I_2+\mathsf B=h_2Y+(1+h_2)\mathsf B\ge\phi(Y)\). Thus, its shipment cost is still at least the original $J_{\mathrm{imp}}(k)$. Taking $\varepsilon,\eta,\delta$ sufficiently small preserves the strict comparison with $r$. This proves the positive-primitive claim and completes Theorem~\ref{thm:impossibility}.

Specifically, for $r>0$, let $\Delta=J_{\mathrm{imp}}(k)-r(I+k\nu)>0$ on the selected finite boundary instance. Choose the three positive perturbations so that $\eta+\varepsilon(S+\delta)+(1+h_2)\delta<\Delta/(2r)$. The new competitor then costs less than $(J_{\mathrm{imp}}(k)-\Delta/2)/r$, while every classical policy still costs at least $J_{\mathrm{imp}}(k)$. The resulting strict gap establishes the claim without assuming the continuity of an optimized policy value. Values $r\le0$ are immediate from positive costs on these instances.

The feasible competitor uses variable downstream shipment quantities. Consequently, we compare with $\OPTsh$; the same competitor cost provides no upper bound on the fixed-batch optimum $\OPTlot$. The benchmark ordering alone implies no impossibility bound for $\OPTlot$.
\section{Numerical Evaluation and Accuracy}\label{app:computation}

\subsection{Policy Evaluation and the Denominator}

For each enumerated \((Q_1,n,\ell)\), let \(J_n\) and \(V_{Q_1}\) be independent cycle positions, uniform on \(\{0,\ldots,n-1\}\) and \(\{0,\ldots,Q_1-1\}\), respectively, and independent of demand. Form \(D_{\rm eff}=D_1+V_{Q_1}+(D_2-\ell-Q_1J_n)^+\). An integer \(p/(h_1+h_2+p)\)-quantile of this effective demand minimizes the downstream mismatch loss. Indeed, for any integer-valued \(X\), \(\E\phi_{a,b}(s+1-X)-\E\phi_{a,b}(s-X) =(a+b)\Pr(X\le s)-b\). These differences are nondecreasing in \(s\). Choosing their first nonnegative index gives a minimizer, and a zero difference gives adjacent minimizers with equal cost. With endpoint \(s\), the controls are \((R_1,R_2,Q_1,Q_2)=(s-Q_1,s-Q_1+\ell,Q_1,nQ_1)\). Adding the setup costs \(k_1/Q_1+k_2/(nQ_1)\) to \(H_0+h_2\E(\ell+Q_1J_n-D_2)^+ +\E\phi_{h_1+h_2,p}(s-D_{\rm eff})\) gives the lot-accounted cost in Proposition~\ref{prop:rnq-account-main}. Either tied integer quantile is valid.

The broad grids initially use \(Q_1\le20\), \(n\le15\). When a selected control approaches a search limit, both limits increase to 40; this happened at 13 of the 1,801 distinct points. For clipped Stage-2 support \(\{0,\ldots,T_2\}\), the placement search covers \(-(n-1)Q_1\le\ell\le T_2\). Below this range, every summand in the cost formula has an upstream shortage, and a shift of the downstream endpoint absorbs any further decrease; above it, every summand has upstream stock, and additional reserve only increases the holding cost. The downstream-lot checks use the independent uniform position $V_{Q_1}$ to bound the cost of each unsearched downstream lot size from below and compare those bounds with the best-found policy cost. These checks pass at all broad-grid points. Larger upstream ratios are not excluded analytically, so the reported policy is the best one found within the stated menus.

The denominator is the largest evaluated feasible value of \((1-u)V_0+u\mathcal A(k_1/u,k_2/u)\), including the zero-setup value. A finite linear program maximizes a version of this objective in which each single-stage problem is restricted to a finite menu of cycles. Restricting the cycles can overstate their infima, so the selected scale and allocation are reevaluated with the integer-cycle searches and stopping arguments below. Fixed-scale allocation searches provide additional candidates. The resulting denominator can fall below the supremum \(\LBsp\). Every exact candidate value remains a valid lower bound by Appendix~\ref{app:zero-setup}.

For each allocated single-stage component with slopes \(a,b>0\), the cycle search uses the bound
\[
 \mathcal H_m\bigl(\E\phi_{a,b}(\,\cdot-D)\bigr)
 \ge\min_x\E\phi_{a,b}(x-J_m)
 \ge\frac{ab}{a+b}\frac{m-1}{2}.
\]
The first inequality follows by replacing independent demand by its mean inside a convex loss and then shifting the chosen endpoint. For the second, a convex function \(g\) satisfies, when \(m\ge2\),
\[
 \frac1m\sum_{j=0}^{m-1}g(j)
 \ge\frac{g(0)/2+\sum_{j=1}^{m-2}g(j)+g(m-1)/2}{m-1}
 \ge\frac1{m-1}\int_0^{m-1}g(x)\,dx.
\]
The first inequality follows by bounding each interior value by the chord joining the endpoints and summing those chord bounds. The second integrates the chord above \(g\) on each unit interval. Apply this to \(g(x)=\phi_{a,b}(s-x)\). The integral is minimized at \(s=b(m-1)/(a+b)\), with value \(ab(m-1)/(2(a+b))\) after division by \(m-1\). The case \(m=1\) is immediate. Once \(m\) exceeds \(\sqrt{2k(a+b)/(ab)}\), the resulting lower bound \(k/m+ab(m-1)/(2(a+b))\) is increasing. A reference candidate whose cost lies below that bound excludes every larger cycle. Without this stopping argument, a finite menu gives an upper estimate of the component infimum, unsuitable for a lower-bound denominator.

\subsection{Components with a Zero Holding or Backlog Rate}

An allocation may assign zero holding or backlog cost to a component even when the original system has strictly positive rates. The two-sided stopping bound above then has a zero denominator and must be replaced by the exact one-sided infimum. For nonnegative integer lead-time demand \(D\) with finite mean, define \(F(x)=\E\phi_{a,b}(x-D)\). If \(a=0\) or \(b=0\), then
\begin{equation}
 \mathcal C_F(k)=0\qquad(k\ge0).
 \label{eq:app-one-sided-component}
\end{equation}
This statement concerns the optimized single-stage component. Its setup charge is retained while the lot size and placement are allowed to vary.

If \(b=0\), choose endpoint \(s=0\). All cycle positions \(s-J_m\) are nonpositive, as is their difference from nonnegative demand, so the holding loss is exactly zero. The cost is \(k/m\), which approaches zero as the finite integer cycle size \(m\) increases. If \(a=0<b\), fix any positive tolerance \(\epsilon\), and first choose a finite cycle size with \(k/m<\epsilon/2\). Then choose a sufficiently large finite integer endpoint \(s\) so that \(b\E(D+J_m-s)^+<\epsilon/2\). Such an endpoint exists because \(D+J_m\) has finite mean. The total cost is below \(\epsilon\). The nonnegativity of costs gives the reverse bound in both cases, proving \eqref{eq:app-one-sided-component}. When both slopes become zero, only \(k/m\) remains.

This treatment matters for a numerical lower bound. Minimizing a one-sided component over only a finite menu can return a positive value even though its true infimum is zero. Substituting that positive number would overstate the component lower bound. The exact value zero should, therefore, be used whenever an allocated slope becomes zero. For strictly positive slopes, the stopping argument excludes the unsearched cycle sizes. These two cases together explain how a feasible allocation and scale can be evaluated without turning a finite search minimum into an unjustified system lower bound. The other component and the zero-setup contribution can keep the full serial cost positive.

\subsection{Numerical Summaries}

The three broad panels each contain \(25\times25\) points. Their intersections give 1,801 distinct primitive vectors. After deduplication, the ratio of the best-found policy cost to the evaluated $\LBsp$ candidate has a maximum of 1.2265, a median of 1.1208, and a 90th percentile of 1.1928. The respective panel maxima are 1.2265, 1.2021, and 1.1874. Summaries count each distinct primitive vector once, whereas a displayed panel retains its own intersections.

The separate 625-point stress grid has maximum ratio 1.6004. At that point, the best-found policy is \((-30,6,30,30)\), with cost 0.060576 and an evaluated $\LBsp$ candidate of 0.037852. Its common-lot size search was extended through 320. The comparison at \(K_2=0\) described next gives a lot-accounted bracket of width \(\lambda K_2/30\) under the approximated demand distributions. These computations remain subject to the numerical errors discussed under Infinite Poisson Tails.

More explicitly, at \(K_2=0\) each summand of \eqref{eq:app-exact-cost}, together with \(k_1/Q\), is the cost of a common-lot policy with reserve \(\ell+jQ\). The average cannot improve on the best common-lot choice. If lot size 30 attains the minimum \(C_0\), the monotonicity in the upstream setup coefficient gives \(C_0\le\inf_{\Pi\in\mathcal P_{\RnQ}}\Clot(\Pi) \le C_0+\lambda K_2/30\). The computation checks all common lot sizes through 320 and excludes larger downstream lot sizes with the stated stopping argument. This explains the numerical bracket without imposing an upper bound on the upstream multiplier.

\subsection{Infinite Poisson Tails}

We retain probability mass up to a cutoff with a nominal upper-tail probability of at most \(10^{-13}\), and assign the omitted mass to the cutoff. In exact arithmetic this replaces \(D\sim\operatorname{Poisson}(\mu)\) by \(D^T=\min(D,T)\). Let \(q_j=\Pr(D\ge j)\), \(f_j=\Pr(D=j)=e^{-\mu}\mu^j/j!\), and \(\epsilon(\mu,T)=\E(D-T)^+\). Then
\[
 \epsilon(\mu,T)=\mu q_T-Tq_{T+1},\qquad
 \epsilon(\mu,T)\le\frac{f_{T+1}}{(1-\mu/(T+2))^2}
 \quad\text{if }T+2>\mu.
\]
The inequality follows from \(f_{T+r}\le f_{T+1}[\mu/(T+2)]^{r-1}\) and the geometric sum \(\sum_{r\ge1}r x^{r-1}=(1-x)^{-2}\). Thus, a loss whose absolute slope is at most \(L\) changes by at most \(L\epsilon(\mu,T)\). The bound is uniform over cycle placements and sizes, so taking an infimum does not amplify it.

For the complete policy expression, put \(L=\max\{h_1+h_2,p\}\) and \(\epsilon_i=\epsilon(\lambda L_i,T_i)\). Coupling each clipped demand with its original value gives the uniform error bound \(L\epsilon_1+(L+h_2)\epsilon_2\). Setup rates and the constant \(H_0\) are unchanged. This controls the unbounded losses through their first-moment tails, including in the stress region. Directed rounding and interval arithmetic for the probability and cost calculations would turn these bounds into intervals guaranteed to contain the exact values. The calculations use ordinary floating-point arithmetic; the weighted-loss routine additionally removes weights at most \(10^{-16}\) and renormalizes. The summaries, therefore, report numerical evidence. The analytic guarantee does not depend on these evaluations.

For example, \(\mu=1\) and \(T=20\) give \(\epsilon(1,20)\le e^{-1}(22/21)^2/21!<2.2\times10^{-20}\). The numerical size of a cost error also depends on its loss slope, which can be large in the stress experiment. The probability-tail target alone is, therefore, insufficient to guarantee cost accuracy. The additional weight deletion has a separate bound: if mass \(\delta<1\) is removed from a finite distribution and the remainder is normalized, a loss with range at most \(B\) on that support changes in expectation by at most \(\delta B\). This follows by writing the original expectation as the mixture of the retained and deleted conditional expectations. For a loss function with bounded absolute slope, \(B\) is at most its slope bound times the support diameter, regardless of any horizontal shift. Keeping these two errors separate would permit an implementation with rigorous numerical error bounds without changing the optimization formulas.

\subsection{Shipment Accounting}

The policy numerator charges every complete lot. Consequently, it also bounds the shipment-accounted cost of the same controls. For a separate direct check, unit downstream lots with \(R_1=-1\) have \(\mathit{IP}_1=\min\{0,\mathit{IL}_2\}\). If \(Y\) is uniform on \(\{R_2+1,\ldots,R_2+Q_2\}\), independent of \(D\sim\operatorname{Poisson}(\lambda L_2)\), their shipment-accounted cost is
\begin{equation}
 (h_2+p)\lambda L_1+p\E(D-Y)^++h_2\E(Y-D)^++k_1\left[\Pr(D\le Y-1)+\frac{\Pr(D\ge R_2+1)}{Q_2}\right]+\frac{k_2}{Q_2}.
 \label{eq:app-unit-shipment-cost}
\end{equation}
The nonpositive downstream position gives \(I_1=0\). The lead-time identity implies \(\E\mathsf B=\lambda L_1+\E(D-Y)^+, \E I_2=\lambda L_1+\E(Y-D)^+\). The second equality uses \(I_2=\mathit{IL}_2+\mathsf B\) and \(\mathit{IL}_2=Y-D\) in distribution. Thus, upstream echelon holding includes goods in the downstream pipeline.

A customer demand causes a unit dispatch precisely when the pre-demand Stage-2 level is positive. Poisson compensation gives rate \(\lambda\Pr(D\le Y-1)\). Upstream orders arrive at rate \(\lambda/Q_2\). Just before an upstream order arrives, the Stage-2 level equals \(R_2-D\), where the demand since that order was placed is an independent Poisson variable with mean \(\lambda L_2\). Pending downstream requests exist precisely when \(D\ge R_2+1\). All units released at that event share one dispatch charge, giving the second rate \((\lambda/Q_2)\Pr(D\ge R_2+1)\). With positive \(L_2\), demand and arrival epochs are distinct almost surely. When \(L_2=0\), the two indicators are still disjoint at an ordering demand: the first applies if \(R_2\ge0\), and the second if \(R_2\le-1\). This proves \eqref{eq:app-unit-shipment-cost} also at zero lead time. The reported ratios to the evaluated $\LBsp$ candidates continue to use the lot-accounted numerator.

For the stress-grid instance with the largest ratio, the separate evaluation of \eqref{eq:app-unit-shipment-cost} fixes \(R_1=-1\) and \(Q_1=1\), and enumerates \(R_2\in\{-10,\ldots,100\}\) and \(Q_2\in\{1,\ldots,200\}\). The smallest evaluated cost among these 22,200 pairs is \(0.055303\), attained at \((-1,9,1,33)\), as reported in Section~\ref{sec:numerics}. The exact stationary formula is evaluated by a finite Poisson sum.

\subsection{Shipment Simulations}

The separate event-driven simulations charge one setup per positive physical dispatch. The complete-lot rule dispatches all available requested lots together. The modified rule of \citet{HuYang2014} additionally permits an incomplete lot when a full lot is unavailable. At \(Q_1=1\), the rules coincide event by event; the simulations check agreement in their dispatch sequences and cost components. For common lot sizes \(Q_1=Q_2\in\{5,10,20,30,45,60\}\), each dispatch contains one lot, and the simulation intervals contain the exact formula values.

Both simulation searches at this instance first evaluate a coarse menu for 15,000 time units in 10 batches, refine its best 50 candidates for 120,000 units in 24 batches, and select among the best 20 for 400,000 units in 40 batches. Their respective warm-ups are 5,000, 10,000, and 20,000. For fixed $Q_1$, $Q_2$, and $R_2-R_1$, the training calculation adds the same integer to both reorder points and chooses the shift that minimizes the estimated holding and backlog cost. It reuses the event path and setup counts, weighting each inventory-position state by the time spent there. For the final evaluation, we keep the selected controls fixed and use eight independent seeds, each for 600,000 time units after a 20,000-unit warm-up. The reported interval is the replication mean plus or minus the 97.5th percentile of Student's \(t\) distribution with seven degrees of freedom times the replication standard error. It measures uncertainty for those fixed controls; the menu search remains finite.

The classical menu has 735 candidates: \(Q_1\in\{1,\ldots,8\}\), \(Q_2\in\{24,\ldots,40\}\) restricted to multiples of \(Q_1\), and \(R_2-R_1\in\{4,\ldots,18\}\). It selects \((-1,9,1,32)\), with reporting mean 0.055482 and interval \([0.055245,0.055720]\). Equation~\eqref{eq:app-unit-shipment-cost} gives 0.055336 for these controls. The direct evaluation above selects \(Q_2=33\), while the simulation selects \(Q_2=32\); both choices lie in the simulation menu. Their formula costs differ by about \(0.000034\), and the simulation estimates reverse their ordering. The incomplete-lot menu allows every \(Q_2\) in that range and restricts \(Q_1\) to \(\{2,\ldots,8\}\), giving 1,785 candidates. It selects \((-2,8,2,33)\), with mean 0.055608 and interval \([0.055394,0.055822]\). Its reported replications contain 55 to 82 incomplete dispatches each. These results support the comparison in Section~\ref{sec:numerics} within the evaluated search ranges.

The validation has three distinct parts. The event comparison at unit lots checks that the two dispatch rules implement the same material flows when their feasible shipment sizes coincide. The common-lot comparisons check inventory integration and setup counting against an independent exact formula. Finally, the reporting replications check the cost of controls selected on other demand paths. Their upstream-order and downstream-unit rates are also compared with the required rates \(\lambda/Q_2\) and \(\lambda\); the reported rates satisfy the specified 1.5\% relative-error screen. This flow check detects implementation errors or severe transient effects, but it is not used as a statistical confidence statement. The independent-seed intervals are conditional on fixed controls and do not measure uncertainty for policies outside the search.
\label{end:appendix}
\end{document}